\documentclass[12pt]{article}
\usepackage[utf8]{inputenc}
\usepackage{fullpage}
\usepackage[hang,flushmargin]{footmisc}
\usepackage{amsmath}
\usepackage{amssymb}
\usepackage{amsthm}
\usepackage{nccmath}
\usepackage{enumerate}
\usepackage{float}
\usepackage{tikz-cd}
\usepackage{bbm}
\usepackage[bookmarksnumbered,bookmarksopen]{hyperref}
\usepackage{lmodern}
\usepackage{titling}
\usepackage[titles]{tocloft}
\usepackage[nobysame]{amsrefs}
\usepackage{graphicx}
\graphicspath{{./Figures/}}

\newcommand\myurl[1]{\href{https://arxiv.org/abs/#1}{{\fontfamily{lmtt}\selectfont{arxiv: #1}}}}
\newcommand\mydoi[1]{\href{https://doi.org/#1}{{\fontfamily{lmtt}\selectfont{doi: #1}}}}
\BibSpec{article}{
+{}{\PrintAuthors} {author}
+{,}{ \textit} {title}
+{,}{ } {journal}
+{}{ \textbf} {volume}
+{}{ \parenthesize} {date}
+{,}{ } {pages}
+{,}{ } {note}
+{,}{ \myurl} {eprint}
+{,}{ \mydoi} {url}
+{.}{} {transition}
+{}{ } {review}
}

\theoremstyle{plain}
\newtheorem{thm}{Theorem}[section]
\newtheorem{lem}[thm]{Lemma}
\newtheorem{rmk}[thm]{Remark}

\newtheorem{prop}[thm]{Proposition}
\newtheorem{defn}[thm]{Definition}

\newtheorem{exmp}[thm]{Example}
\newtheorem{con}[thm]{Convention}

\newcommand{\sgn}{\textup{sgn}\,}
\newcommand{\im}{\textup{im}\,}
\newcommand{\id}{\textup{id}}
\newcommand{\Aut}{\textup{Aut}\,}
\newcommand{\rdim}{\textup{rdim}\,}

\newcommand{\pslr}{\textup{PSL}_{2}(\mathbb{R})}
\newcommand{\pslc}{\textup{PSL}_{2}(\mathbb{C})}
\newcommand{\ev}{ev^{\beta}_{j}}
\newcommand{\evz}{ev^{\beta}_{0}}
\newcommand{\evi}{evi^{\beta}_{j}}
\newcommand{\evb}{evb^{\beta}_{j}}
\newcommand{\evbz}{evb^{\beta}_{0}}
\newcommand{\M}{\mathcal{M}_{k+1,l}(\beta)}
\newcommand{\Mg}{\mathcal{M}_{k+1,l;m}(\beta)}
\newcommand{\Mgz}{\mathcal{M}_{k+1,l;0}(\beta)}
\newcommand{\Mgg}{\mathcal{M}_{k+1,l;0,m}(\beta)}
\newcommand{\Mh}{\mathcal{M}_{k+1,l;\perp}(\beta)}

\newcommand{\Ml}{\mathcal{M}_{k+1,l;+,m}(\beta)}
\newcommand{\Mr}{\mathcal{M}_{k+1,l;-,m}(\beta)}
\newcommand{\Ms}{\mathcal{M}_{k+1,l;\pm,m}(\beta)}
\newcommand{\Mx}{\mathcal{M}_{l+1}(\beta)}

\newcommand{\qop}{\mathfrak{q}_{k,l}^{\beta}}
\newcommand{\qg}{\mathfrak{q}_{k,l;m}^{\beta}}

\newcommand{\qgg}{\mathfrak{q}_{k,l;0,m}^{\beta}}
\newcommand{\qh}{\mathfrak{q}_{k,l;\perp}^{\beta}}

\newcommand{\qs}{\mathfrak{q}_{k,l;\pm,m}^{\beta}}
\newcommand{\qx}{\mathfrak{q}_{\emptyset,l}^{\beta}}
\newcommand{\co}{\mathcal{CO}}
\newcommand{\oc}{\mathcal{OC}}
\newcommand{\occ}{\overline{\mathcal{OC}}}
\newcommand{\dl}{\mathfrak{q}^{\gamma,b}_{1,0}}
\newcommand{\zl}{\mathfrak{q}^{\gamma,b}_{0,0}}
\newcommand{\lml}{\mathfrak{q}^{\gamma,b}_{2,0}}
\newcommand{\xmx}{\mathfrak{q}^{\gamma}_{\emptyset,2}}
\newcommand{\xmcx}{\overline{\mathfrak{q}}^{\gamma}_{\emptyset,2}}
\newcommand{\xml}{\mathfrak{q}^{\gamma,b}_{1,1;1}}
\newcommand{\xtl}{\mathfrak{q}^{\gamma,b}_{0,1}}
\newcommand{\ltcx}{\mathfrak{p}^{\gamma,b}_{1,0}}
\newcommand{\qhx}{QH^{*}(X)}
\newcommand{\qhcx}{\overline{QH}^{*}(X)}
\newcommand{\cqhx}{({QH}^{*}(X))^{\vee}}
\newcommand{\qhl}{HF^{*}(L)}

\newcommand{\pdl}{PD([L])}

\newcommand{\Cinf}{\widehat{\mathbb{C}}}
\newcommand{\Cp}{\mathbb{C}P}

\newcommand{\Rinf}{\widehat{\mathbb{R}}}

\newcommand{\D}{D_{\pm}}

\title{Algebraic structures in Lagrangian Floer cohomology modelled on differential forms}
\author{Peleg Bar-Lev}
\date{February 2024}

\begin{document}

\vspace*{1em}

\let\thefootnote\relax\footnotetext{\textit{Date:} April 2024.}

\begin{center}
{\Large\textbf{Algebraic Structures in Lagrangian Floer Cohomology Modelled on Differential Forms}}

\vspace{1em}

{\Large Peleg Bar-Lev}
\end{center}

\vspace{1em}

\textbf{Abstract.} We define a structure of an algebra on the Lagrangian Floer cohomology of a Lagrangian submanifold over the quantum cohomology of the ambient symplectic manifold. The structure is analogous to the one defined by Biran-Cornea, but is constructed in the differential forms model. In the spirit of Ganatra and Hugtenburg, we define another such algebra structure using a closed-open map. We show that the two structures coincide. As an application, we show that the module structure for the 2-dimensional Clifford torus is given by multiplication by a Novikov coefficient, similarly to the Biran-Cornea module structure for this case.

\tableofcontents

\section{Introduction}

\subsection{Main results}

Let $(X^{2n},\omega)$ be a closed symplectic manifold and $L\subset X$ a closed, connected, Lagrangian submanifold with a relative spin structure $\mathfrak{s}$ as in \cite{LIFT}*{Definition 8.1.2}. In particular, $L$ is oriented. Let $J$ be an $\omega$-tame almost complex structure on $X$. In \cite{LQH}, Biran-Cornea construct a quantum homology $QH_{*}(L)$ modelled on the pearl complex, assuming monotonicity. They show that it is an algebra over $QH_{*}(X)$, and canonically isomorphic as a ring to the Floer homology $HF_{*}(L)$ modelled on periodic orbits of a Hamiltonian system.
\par
In the present thesis we consider the Floer cohomology $\qhl$ modelled on differential forms, defined using an $A_{\infty}$ algebra as in the construction of Fukaya-Oh-Ohta-Ono in \citelist{\cite{tor}\cite{torbd}\cite{cyc}}. It can be thought of as a Lagrangian analogue to the quantum cohomology $QH^{*}(X)$. We use the setup of Solomon-Tukachinsky as in \cite{Ainf} and therefore work under their regularity assumptions (see Section \ref{reg}).
\par
As a module, $\qhx$ is isomorphic to the de Rham cohomology $H^{*}(X)$ with coefficients in a Novikov ring, and as a ring it is equipped with a quantum product defined by deforming the usual wedge product using pseudoholomorphic curves. The chains of $\qhl$ are also differential forms, but now both the coboundary and the product are deformed. Precise definitions are given in Section \ref{cohomology}.
\par
Although the pearl model is not known to be equivalent to the differential forms model, it is widely believed that this is the case (when working over the same coefficient ring). It is therefore natural to expect that $\qhl$ has a structure of an algbera over $\qhx$, analogous to the algbera structure in \cite{LQH}. We show that this is indeed the case.
\par
In Proposition \ref{algmap}, we define a map
\begin{equation*}
    \circledast:\qhx\otimes\qhl\rightarrow\qhl,
\end{equation*}
by using pseudoholomorphic disks with a geodesic constraint as in \cite{RQH}*{Section 3}, in an analogous manner to the module operation in \cite{LQH}*{Section 3.2}.
\begin{thm}\label{algebra}
The map $\circledast$ makes $\qhl$ a unital graded algebra over $\qhx$.
\end{thm}
To prove Theorem \ref{algebra}, we construct the moduli space of pseudoholomorphic disks with a geodesic through $4$ marked points. By analysing its boundary, we obtain structure equations as in \citelist{\cite{Ainf}\cite{RQH}}. An analogous moduli space was used in \cite{LQH}*{Section 3.2}.
\par
Furthermore, we define closed-open and open-closed maps, and demonstrate their algebraic properties. In Proposition \ref{cochain}, we define the closed-open map
\begin{equation*}
    \co:\qhx\rightarrow\qhl,
\end{equation*}
Additionally, in Definition \ref{ocdef} we define an open-closed map
\begin{equation*}
    \oc:\qhl\rightarrow\qhx.
\end{equation*}
\par
On the chain level, the closed-open map is defined as a deformation of the pullback along the inclusion $L\hookrightarrow X$ (that is, the restriction map of differential forms from $X$ to $L$), by using pseudoholomorphic disks. To define the open-closed map, we construct the quantum cohomology of currents $\qhcx$, which is a module over $\qhx$, and is isomorphic as a module to $\qhx$. The map $\oc$ is then constructed from the map $\co$ by using the duality between currents and forms, and by applying the isomorphism mentioned above. The use of currents is required to push differential forms forward along maps which are not necessarily submersions.
\par
These maps extend to Hugtenburg's closed-open and open-closed maps which, instead of $\qhl$, involve Hochschild cohomology (see \cite{OC}*{sections 4.8.1 and 4.8.2}). In pseudo-cycle model, analogous maps are Sheridan's $\co^{0}$ and $\oc^{0}$ (see \cite{pseudocyc}*{sections 2.5 and 2.6}), but an equivalence of this model to the model of differential forms is not known in the literature.
\par
Denoting by
\begin{align*}
    \circ:\qhl\otimes\qhl\rightarrow\qhl
\end{align*}
and by
\begin{align*}
    *:\qhx\otimes\qhx\rightarrow\qhx
\end{align*}
the products which are defined in Section \ref{cohomology}, and by $\pdl\in H^{n}(X)$ the Poincaré dual to $[L]\in H_{n}(X)$, we prove the following theorem.
\begin{thm}\label{maps}
The map $\co$ is an algebra homomorphism, and the map $\oc$ is a degree $n$ module homomorphism. Furthermore, for $y\in\qhx$, $a\in\qhl$ we have
\begin{align*}
    \co(y)\circ a&=(-1)^{|y||a|}a\circ\co(y)
\end{align*}
and
\begin{align*}  
    y*\oc(a)&=\oc(\co(y)\circ a).
\end{align*}
We also have
\begin{align*}
    \oc(1)=\pdl.
\end{align*}
\end{thm}
\par
In the spirit of Ganatra \cite{sympH}*{Section 5.5} and Hugtenburg \cite{OC}*{Section 4.8.2}, since $\co$ is an algebra homomorphism, we can define on $\qhl$ an additional algebra structure over $\qhx$, given by
\begin{align*}
    \qhx\otimes\qhl&\rightarrow\qhl\\
    (y,a)&\mapsto\co(y)\circ a.
\end{align*}
We show that this algebra structure coincides with the one given by $\circledast$ in the following sense.
\begin{thm}\label{compare}
For $y\in\qhx$, $a\in\qhl$ we have
\begin{align*}
    y\circledast a=\co(y)\circ a.
\end{align*}
\end{thm}
Here $|\cdot|$ denotes the degree of a homogeneous element of a graded module.
\par
To prove Theorem \ref{compare}, we construct the moduli space of pseudoholomorphic disks with an interior marked point constrained to one side of a geodesic, and obtain structure equations by analysing its boundary. To the author's best knowledge, such moduli spaces have not previously appeared in the literature.
\par
As an application of Theorem \ref{compare}, we show in Section \ref{exm} that the module structure for the Clifford torus $\mathbb{T}^{2}_{\textup{clif}}\subset\Cp^{2}$ is given by multiplication by a Novikov coefficient. This can be compared to the computation of the Biran-Cornea module structure for this case, which is also given by multiplication by a Novikov coefficient (see \cite{LQH}*{Theorem 2.3.2}).
\par
A priori, the constructions in this work depend on the choice of the symplectic form $\omega$ and the almost complex structure $J$. We conjecture that invariance can be shown in the presence of a pseudo-isotopy that satisfies the additional regularity assumptions of \cite{Ainf}*{Section 4}. However, we do expect dependence on the relative spin structure $\mathfrak{s}$. We use the standard complex structures on the Riemann sphere $\Cinf$ and on the upper half-plane $\mathcal{H}$, and the induced structure on genus zero nodal domains $\Sigma$. Note that these complex structures are unique up to biholomorphism.

\subsection{Regularity assumptions}\label{reg}

Our regularity assumptions can be summarized as follows.
\begin{enumerate}
    \item The moduli spaces $\M$ and $\Mx$ defined in Section \ref{op} are smooth orbifolds with corners.
    \item The evaluation maps denoted by $\evbz$ and $\evz$ in Section \ref{op} are proper submersions.
    \item The evaluation maps denoted by $\evbz$ in Section \ref{gop} are proper submersions.
    \item The evaluation maps denoted by $\evbz$ in Section \ref{horop} are proper submersions.
    \item The evaluation maps denoted by $\evbz$ in Section \ref{space} are proper submersions.
\end{enumerate}
Assumptions 1 and 2 are the assumptions of Solomon-Tukachinsky in \cite{Ainf}, and assumption 3 is their assumption in \cite{RQH}. Assumption 4 is the assumption of Hugtenburg in \cite{OC}. Note that, even if these assumptions hold for some $\omega$-tame almost complex structure $J$, we do not know if they hold for generic such $J$.
\par
As mentioned in \cite{RQH}*{Section 1.3.12}, assumptions 1, 2 and 3 hold for homogeneous spaces. Similarly, \cite{OC}*{Lemma 4.19} shows that assumption 4 also hold in this case. Applying the same argument as in \cite{RQH}*{Section 1.3.12}, while replacing the moduli space $\Mg$ as defined in section \ref{gop} with the moduli space $\Mgg$ defined in section \ref{space}, shows that assumption 5 also holds for homogeneous spaces.
\par
In particular, our assumptions hold for $(X,L)=(\mathbb{C}P^{n}, \mathbb{R}P^{n})$ with the standard structures $J=J_{0}$ and $\omega=\omega_{FS}$ (note that in this case, for $L$ to be orientable we require $n$ odd). More generally, Grassmannians, flag varieties and products thereof satisfy our assumptions.
\par
Using virtual fundamental class techniques, our results should extend to general target manifolds. The only exception is Lemma \ref{BPlem}, which is used to verify the existence of bounding pairs in Proposition \ref{BPexist}. As Remark \ref{BPrmk} states, the existence of bounding pairs could also be verified if assuming monotonicity.

\subsection{Acknowledgements}
I would like to express my gratitude to my advisor, Sara Tukachinsky, for her dedicated support and guidance. I would also like to thank Roi Blumberg, Itai Bar-Deroma, and Kai Hugtenburg, for interesting and helpful conversations. The author was partially supported by ISF grant 2793/21.

\section{Background}\label{background}

\subsection{Fiber products and transversality}\label{fibtrans}

We use the definitions of manifolds with corners, their boundary, and their fiber products, given in \cite{corner}. We call a map $f:M\rightarrow N$ of manifolds with corners smooth if all its partial derivatives, including one-sided derivatives, exist and are continuous at every point. This notion of a smooth map agrees with the notion of a weakly smooth map from \cite{corner}. When $N$ is boundaryless this implies a stronger notion of smoothness (smoothness in the sense of \cite{corner}), which allows the definition of fiber products. We call such a smooth map a submersion if it a submersion as defined in \cite{corner}*{Definition 3.2}. In particular, projections to the components of a direct product are surjective submersions.
\par
We also use the definitions of transversality and orientation given in \cite{corner}, including the conventions for orienting boundaries and fiber products. In particular, for transversality see \cite{corner}*{Definition 6.1}. To orient fiber products, let $Y_{1},Y_{2}$ be oriented manifolds with corners, and let $Z$ be an oriented manifold without corners. Let $f_{1}:Y_{1}\rightarrow Z$, $f_{2}:Y_{2}\rightarrow Z$ be transversal smooth maps, and let $W=Y_{1}\times_{Z}Y_{2}$ be the fiber product taken with respect to these maps. Denote by $p_{i}:W\rightarrow Y_{i}$ the projections. Fixing $w=(y_{1},y_{2})\in W$ and writing $z=f_{1}(y_{1})=f_{2}(y_{2})$, we have a short exact sequence given by
\begin{center}
\begin{tikzcd}
    0\rightarrow T_{w}W\arrow{rr}{d_{w}p_{1}\oplus d_{w}p_{2}}&&T_{y_{1}}Y_{1}\oplus T_{y_{2}}Y_{2}\arrow{rr}{d_{y_{1}}f_{1}-d_{y_{2}}f_{2}}&&T_{z}Z\rightarrow0.
\end{tikzcd}
\end{center}
Choosing a splitting gives an isomorphism
\begin{align*}
    \tau:T_{y_{1}}Y_{1}\oplus T_{y_{2}}Y_{2}\rightarrow T_{w}W\oplus T_{z}Z.
\end{align*}
We require the sign of this isomorphism to be $\sgn(\tau)=(-1)^{\dim Y_{2}\cdot\dim Z}$.
\par
Recall the following properties of the orientation of fiber products.
\begin{prop}[\cite{corner}*{Proposition 7.5}]\label{fiber}
Let $Y_{1},Y_{2},Y_{3}$ be oriented manifolds with corners, $Z_{1},Z_{2}$ oriented manifolds without corners, and $f_{1}:Y_{1}\rightarrow Z_{1}$, $f_{2}:Y_{2}\rightarrow Z_{1}$, $g_{2}:Y_{2}\rightarrow Z_{2}$, $g_{3}:Y_{3}\rightarrow Z_{2}$ smooth maps. Assume that $f_{1}\pitchfork f_{2}$ and $g_{2}\pitchfork g_{3}$. Then we have
\begin{align*}
    Y_{1}\times_{Z_{1}}Y_{2}&=(-1)^{(\dim Y_{1}+\dim Z_{1})(\dim Y_{2}+\dim Z_{1})}Y_{2}\times_{Z_{1}}Y_{1},\\(Y_{1}\times_{Z_{1}}Y_{2})\times_{Z_{2}}Y_{3}&=Y_{1}\times_{Z_{1}}(Y_{2}\times_{Z_{2}}Y_{3}),\\Y_{1}\times_{Z_{1}\times Z_{2}}(Y_{2}\times Y_{3})&=(-1)^{\dim Z_{2}(\dim Y_{2}+\dim Z_{1})}Y_{1}\times_{Z_{1}}(Y_{2}\times_{Z_{2}}Y_{3}),
\end{align*}
where the fiber products are taken with respect to the given maps.
\end{prop}
\begin{prop}[\cite{corner}*{Proposition 7.4}]\label{fiberbdry}
Let $Y_{1},Y_{2}$ be oriented manifolds with corners, $Z$ an oriented manifold without corners, and $f_{1}:Y_{1}\rightarrow Z$, $f_{2}:Y_{2}\rightarrow Z$ smooth maps. Assume that $f_{1}\pitchfork f_{2}$. Then we have
\begin{align*}
    \partial(Y_{1}\times_{Z}Y_{2})&=\partial Y_{1}\times_{Z}Y_{2}+(-1)^{\dim Y_{1}+\dim Z}Y_{1}\times_{Z}\partial Y_{2},
\end{align*}
where the fiber products are taken with respect to the given maps (or to their restriction to the boundary).
\end{prop}
We will also need the following lemmas on transversality, which are easy to verify.
\begin{lem}\label{transcomp}
Let $f_{1}:Y_{1}\rightarrow Z$ and $f_{2}:Y_{2}\rightarrow Z$, and let $g:W\rightarrow Y_{2}$ be a surjective submersion. Then $f_{1}\pitchfork f_{2}$ if and only if $f_{1}\pitchfork f_{2}\circ g$.
\end{lem}
\begin{lem}\label{transprod}
Let $f_{1,i}:Y_{1,i}\rightarrow Z_{i}$ and $f_{2,i}:Y_{2,i}\rightarrow Z_{i}$ be transversal for $i=1,2$, and write $F_{j}=f_{j,1}\times f_{j,2}:Y_{j,1}\times Y_{j,2}\rightarrow Z_{1}\times Z_{2}$ for $j=1,2$. Then $F_{1}\pitchfork F_{2}$.
\end{lem}
\par
We use the definition of orbifolds with corners given in \cite{orb}. In particular, the category of orbifolds with corners is obtained as a localization by refinements of the category of étale proper groupoids with corners. Recall that the object set $Y_{0}$ and the morphism set $Y_{1}$ of an étale proper groupoid $Y$ are manifolds with corners, and the associated topological space $|Y|$ is the quotient of objects by morphisms. 
\par
A morphism of étale proper groupoids with corners is given by a functor $F:Y\rightarrow Z$ where the underlying maps $F_{0}:Y_{0}\rightarrow Z_{0}$ and $F_{1}:Y_{1}\rightarrow Z_{1}$ of manifolds with corners are smooth. Such a morphism induces a continuous map $|F|:|Y|\rightarrow|Z|$. A morphism $f:Y\rightarrow Z$ of orbifolds with corners is given by a pair of morphisms of étale proper groupoids with corners $F:Y'\rightarrow Y$ and $R:Y'\rightarrow Z$, with $R$ a refinement. We abbreviate $f=F|R$.
\par
Let $f=F|R:Y\rightarrow Z$ be a morphism of orbifolds with corners. We say that $f$ is a submersion if $F_{0}$ is a submersion, and that it is proper if $|F|$ is proper. In particular, projections to the components of a direct product are submersions. In addition, we say that $f$ is transversal to another morphism $g=G|S:Y\rightarrow W$ of orbifolds with corners if $F_{0}\pitchfork G_{0}$.
\par
Thus checking transversality for morphisms of orbifolds with corners reduces to checking transversality of the underlying maps on objects, which are maps of manifolds with corners. Henceforth, when checking transversality, we will only discuss these underlying maps. Furthermore, it follows that Lemma \ref{transcomp} also holds when $Y_{1},Y_{2}$ are orbifolds with corners, and Lemma \ref{transprod} also holds when the $Y_{1,i},Y_{2,i}$ are orbifolds with corners.
\par
We use the definitions of orientation, boundaries, and fiber products of orbifolds with corners given in \cite{orb}, including the conventions for orienting boundaries and fiber products. For manifolds with corners, these agree with the conventions of \cite{corner}. We have the following.
\begin{prop}\label{orbprop}
Proposition \ref{fiber} also holds when $Y_{1},Y_{2},Y_{3}$ are orbifolds with corners. Proposition \ref{fiberbdry} also holds when $Y_{1},Y_{2}$ are orbifolds with corners.
\end{prop}
\begin{proof}
The result follows directly from the construction in \cite{orb}*{Section 5.1}. This is because whenever $Y_{1},Y_{2}$ are orbifolds with corners and $Z$ is a manifold, \cite{orb}*{Example 3.1} shows that in their notation we have
\begin{align*}
    (Y_{1}\times_{Z}Y_{2})_{0}=(Y_{1})_{0}\times_{Z}(Y_{2})_{0},\\
    (Y_{1}\times_{Z}Y_{2})_{1}=(Y_{1})_{1}\times_{Z}(Y_{2})_{1}.
\end{align*}
\end{proof}

\subsection{Differential forms and currents}\label{orientint}

For an orbifold with corners $M$, let $A^{*}(M)$ be the differential graded algebra of smooth differential forms on $M$ with coefficients in $\mathbb{R}$. Let $M,N$ be orbifolds with corners, and let $f:M\rightarrow N$. Let $\rdim f=\dim M-\dim N$. Following \cite{orb}*{Section 4.1}, we have the pullback
\begin{align*}
    f^{*}:A^{k}(N)\rightarrow A^{k}(M).
\end{align*}
In the case $f$ is a proper submersion we also have the pushforward
\begin{align*}
    f_{*}:A^{k}(M)\rightarrow A^{k-\rdim f}(N),
\end{align*}
given by integration over the fiber.
\par
Let $M,N,P$ be oriented orbifolds with corners, $f:M\rightarrow N$, $g:P\rightarrow N$, $\xi\in A^{*}(M)$, and $\tau\in A^{*}(N)$. Assume that $f$ is a proper submersion. The following properties of the pushforward are given in \cite{orb}*{Theorem 1}.
\begin{prop}\label{subcomp}
If $g$ is a proper submersion then
\begin{align*}
    f_{*}\circ g_{*}=(f\circ g)_{*}.
\end{align*}
\end{prop}
\begin{prop}\label{int}
We have
\begin{align*}
    f_{*}(f^{*}\tau\wedge\xi)=\tau\wedge f_{*}\xi.
\end{align*}
\end{prop}
\begin{prop}\label{proj}
Let $M\times_{N}P$ be the fiber product taken with respect to $f$ and to $g$, and denote by
\begin{align*}
    p&:M\times_{N}P\rightarrow M,\\
    q&:M\times_{N}P\rightarrow P
\end{align*}
the projections. Then
\begin{align*}
    q_{*}p^{*}\xi=g^{*}f_{*}\xi.
\end{align*}
\end{prop}
\begin{prop}[Stokes' theorem]\label{Stokes}
Assume $\partial N=\emptyset$. Then we have
\begin{align*}
    d(f_{*}\xi)=f_{*}(d\xi)+(-1)^{\dim M+|\xi|}(f|_{\partial M})_{*}\xi.
\end{align*}
\end{prop}
We turn to discuss currents on compact orbifolds. The space $\overline{A}^{k}(M)$ of currents of degree $k$ on a compact orbifold $M$ is the continuous dual of $A^{\dim M-k}(M)$ with the weak* topology. Note that $\overline{A}^{*}(M)$ is a module over $A^{*}(M)$ with
\begin{align*}
    (\xi_{1}\wedge\lambda)(\xi_{2})=\lambda(\xi_{2}\wedge\xi_{1})
\end{align*}
for $\xi_{1},\xi_{2}\in A^{*}(M)$ and $\lambda\in\overline{A}^{*}(M)$.
\begin{defn}
Define
\begin{align*}
    d:\overline{A}^{k}(M)\rightarrow\overline{A}^{k+1}(M)
\end{align*}
by
\begin{align*}
    d\lambda(\xi)=(-1)^{|\xi|+1}\lambda(d\xi).
\end{align*}
\end{defn}
\begin{prop}
$d$ makes $\overline{A}^{*}(M)$ a cochain complex. Denote its cohomology by $\overline{H}^{*}(M)$.
\end{prop}
For $\xi_{1},\xi_{2}\in A^{*}(M)$, denote by
\begin{align*}
    \langle\xi_{2},\xi_{1}\rangle_{M}=\int_{M}\xi_{2}\wedge\xi_{1}
\end{align*}
the Poincaré pairing on $M$.
\begin{defn}\label{curhom}
Define
\begin{align*}
    \varphi:A^{k}(M)\rightarrow\overline{A}^{k}(M)
\end{align*}
by
\begin{align*}
    \varphi(\xi_{1})(\xi_{2})=\langle\xi_{2},\xi_{1}\rangle_{M}.
\end{align*}
\end{defn}
\begin{prop}\label{curhomprop}
$\varphi$ is an injective module homomorphism. If $\partial M=\emptyset$, then $\varphi$ satifies $d\circ\varphi=\varphi\circ d$ and thus descends to a map
\begin{align*}
    \hat{\varphi}:H^{*}(M)\rightarrow\overline{H}^{*}(M)
\end{align*}
in cohomology.
\end{prop}
\begin{prop}[\cite{dR}*{Theorem 14}]\label{curhomiso}
$\hat{\varphi}$ is a module isomorphism.
\end{prop}
\begin{defn}
Define
\begin{align*}
    f_{*}:\overline{A}^{k}(M)\rightarrow\overline{A}^{k-\rdim f}(N)
\end{align*}
by
\begin{align*}
    f_{*}\lambda(\xi)=\lambda(f^{*}\xi).
\end{align*} 
\end{defn}
\begin{prop}
If $f$ is a proper submersion, then $f_{*}\circ\varphi=\varphi\circ f_{*}$. 
\end{prop}
\begin{prop}[\cite{orb}*{Proposition 6.1}]
The analogue of propositions \ref{subcomp} and \ref{int} holds where the differential form $\xi$ is replaced by a current $\lambda$.
\end{prop}
\begin{rmk}
Our sign conventions for the module structure on currents, for the Poincaré pairing, for the map $\varphi$, and for the pushforward of currents differ from the ones in \cite{orb}.
\end{rmk}

\subsection{Operators}\label{op}

Denote by $\mu:H_{2}(X,L;\mathbb{Z})\rightarrow\mathbb{Z}$ the Maslov index as in \cite{Maslov}*{Section 2}, and let $\Pi$ be a quotient of $H_{2}(X,L;\mathbb{Z})$ by some subgroup of $\ker(\omega\oplus\mu)$, so that $\omega,\mu$ descend to $\Pi$. Denote by $\varpi:H_{2}(X;\mathbb{Z})\rightarrow\Pi$ the composition of the projection $H_{2}(X,L;\mathbb{Z})\rightarrow\Pi$ with the map $H_{2}(X;\mathbb{Z})\rightarrow H_{2}(X,L;\mathbb{Z})$ induced by the projection, and denote by $\beta_{0}$ the zero element of $\Pi$.
\par
Let $\Lambda$ be the Novikov ring given by
\begin{align*}
    \Lambda=\Big\{\sum_{i=0}^{\infty}a_{i}T^{\beta_{i}}\:\Big|\:a_{i}\in\mathbb{R},\,\beta_{i}\in\Pi,\,\omega(\beta_{i})\geq0,\,\lim_{i\rightarrow\infty}\omega(\beta_{i})=\infty\Big\},
\end{align*}
where $T^{\beta}$ is a formal variable of degree $\mu(\beta)$ and $T^{\beta_{1}}T^{\beta_{2}}=T^{\beta_{1}+\beta_{2}}$. Let $t_{0},...,t_{N}$ be formal variables with degrees in $\mathbb{Z}$, and let
\begin{align*}
    R=\Lambda[[t_{0},...,t_{N}]].
\end{align*}
Define a valuation $v:R\rightarrow\mathbb{R}$ by
\begin{align*}
    v\Big(\sum_{i=0}^{\infty}a_{i}T^{\beta_{i}}\prod_{j=0}^{N}t_{j}^{\kappa_{ij}}\Big)=\inf_{a_{i}\neq0}\Big(\omega(\beta_{i})+\sum_{j=0}^{N}\kappa_{ij}\Big).
\end{align*}
Let $\mathcal{I}\subset R$ be the ideal $\mathcal{I}=\{c\in R\:|\:v(c)>0\}$.
\par
The valuation $v$ induces a valuation on tensor products with $R$. Thinking of $R$ as a differential graded algebra with trivial differential, set
\begin{align*}
    &C=A^{*}(L)\otimes R,\quad E=A^{*}(X)\otimes R,
\end{align*}
where $\otimes$ stands for the completed tensor product over $\mathbb{R}$ of differential graded algebras, with respect to the induced valuation. 
\par
\begin{con}
We denote by $\alpha$ a general element of $C$, and by $\eta$ a general element of $E$, possibly adding an index. We use the same notation for lists of elements of the same set. If not clear from the context, we specify whether we refer to an element or to a list.
\end{con}
\begin{con}
We keep writing $|\xi|$ for the degree of a homogeneous element $\xi$ of a graded module $\Upsilon$. For example, the degree of $\alpha\in C$ combines its degree as a differential form with the degree coming from the formal variables.
\end{con}
\par
In the following, we recall the definition of the operators and the regularity assumptions from \cite{Ainf}*{Section 2.2}. We refer there for complete detail, including examples where these assumptions hold. For $\beta\in\Pi$ and $k,l\geq0$, let $\M$ be the moduli space of open genus zero stable maps of degree $\beta$, with $k+1$ cyclically ordered boundary marked points and $l$ internal marked points, defined up to reparameterization.
\par
Elements of $\M$ are represented by tuples of the form $(\Sigma,u,\vec{z},\vec{w})$, where $\Sigma$ is a genus zero nodal Riemann surface with one boundary component, $u:(\Sigma,\partial\Sigma)\rightarrow(X,L)$ is $J$-holomorphic of degree $\beta$, and $\vec{z}=(z_{0},...,z_{k})$, $\vec{w}=(w_{1},...,w_{l})$ where $z_{j}\in\partial\Sigma$, $w_{j}\in\mathring{\Sigma}$ are distinct from one another and from the nodal points and the $z_{j}$ are cyclically ordered respecting the orientation on $\partial\Sigma$.
\par
We assume all the moduli spaces $\M$ are orbifolds with corners (see Section \ref{reg}), where the objects are the stable maps described above, and the morphisms are isomorphisms of such stable maps. Denote by $(\M)_{0}$ the set of objects of $\M$, which is a manifold with corners.
\par
For a manifold with corners $N$, any map $(\M)_{0}\rightarrow N$ of manifolds with corners which is invariant under isomorphisms of stable maps can also be thought of as a map of orbifolds with corners. Such a map is smooth as a map of orbifolds with corners if and only the underlying map on objects is smooth as a map of manifolds with corners.
\par
The moduli space $\M$ is a stratified space, where open stable maps whose domain $\Sigma$ has $d+1$ disk components and $e$ sphere components belong to a startum of codimension $d+2e$. This space has a natural orientation induced by the relative spin structure $\mathfrak{s}$, constructed in \cite{LIFT}*{Chapter 8}. We also have the following (see \cite{LIFT}*{Theorem 2.1.32}).
\begin{prop}\label{dim}
We have
\begin{align*}
    \dim\M=n-2+k+2l+\mu(\beta)\equiv n+k\pmod{2}.
\end{align*}
\end{prop}
\par
Denote by
\begin{align*}
    \evb:\M\rightarrow L,\quad j=0,...,k,\\
    \evi:\M\rightarrow X,\quad j=1,...,l,
\end{align*}
the evaluation maps at the boundary and interior marked points, respectively. Note the abuse of notation since $k,l$ are not specified. We assume all the maps $\evbz$ are proper submersions  (see Section \ref{reg}).
\par
In the following, let $\alpha=\bigotimes_{j=1}^{k}\alpha_{j}\in C^{\otimes k}$ and $\eta=\bigotimes_{j=1}^{l}\eta_{j}\in E^{\otimes l}$. Define
\begin{align*}
    \varepsilon(\alpha)=\sum_{j=1}^{k}j|\alpha_{j}|+\frac{k(k+1)}{2}+1,
\end{align*}
and for $(k,l,\beta)\neq(1,0,\beta_{0}),(0,0,\beta_{0})$ define
\begin{align*}
    \qop:C^{\otimes k}\otimes E^{\otimes l}\rightarrow C
\end{align*}
by
\begin{align*}
    \qop(\alpha;\eta)=(-1)^{\varepsilon(\alpha)}(\evbz)_{*}\Big(\bigwedge_{j=1}^{l}(\evi)^{*}\eta_{j}\wedge\bigwedge_{j=1}^{k}(\evb)^{*}\alpha_{j}\Big).
\end{align*}
Define also
\begin{align*}
    \mathfrak{q}_{1,0}^{\beta_{0}}(\alpha)=d\alpha,\quad \mathfrak{q}_{0,0}^{\beta_{0}}=0,
\end{align*}
and set
\begin{align*}
    \mathfrak{q}_{k,l}=\sum_{\beta\in\Pi}T^{\beta}\qop.
\end{align*}
\par
Similarly, for $\beta\in H_{2}(X;\mathbb{Z})$ and $l\geq0$, let $\Mx$ be the moduli space of nodal, $J$-holomorphic, stable spheres in $X$ of degree $\beta$ with $l+1$ marked points, and denote by
\begin{align*}
    \ev:\Mx\rightarrow X
\end{align*}
the evaluation maps. We assume all the spaces $\Mx$ are smooth orbifolds and the maps $\evz$ are proper submersions (see Section \ref{reg}).
\par
Recall that $\mathfrak{s}$ determines $w_{\mathfrak{s}}\in H^{2}(X;\mathbb{Z}/2\mathbb{Z})$ such that $w_{2}(TL)=i^{*}w_{\mathfrak{s}}$. We can think of $w_{\mathfrak{s}}$ as acting on $H_{2}(X;\mathbb{Z})$ via composition with the map $H_{2}(X;\mathbb{Z})\rightarrow H_{2}(X;\mathbb{Z}/2\mathbb{Z})$ induced by the projection.
\par
For $(l,\beta)\neq(1,0),(0,0)$ define
\begin{align*}
    \qx:E^{\otimes l}\rightarrow E
\end{align*}
by
\begin{align*}
    \qx(\eta)=(-1)^{w_{\mathfrak{s}}(\beta)}(\evz)_{*}\Big(\bigwedge_{j=1}^{l}(\ev)^{*}\eta_{j}\Big).
\end{align*}
Define also
\begin{align*}
    \mathfrak{q}_{\emptyset,1}^{0}(\eta)=0,\quad \mathfrak{q}_{\emptyset,0}^{0}=0,
\end{align*}
and set
\begin{align*}
    \mathfrak{q}_{\emptyset,l}=\sum_{\beta\in H_{2}(X;\mathbb{Z})}T^{\varpi(\beta)}\qx.
\end{align*}
\par
These operators have the following properties.
\begin{prop}[Degree, \citelist{\cite{Ainf}*{Proposition 3.5}\cite{RQH}*{Lemma 2.15}}]\label{degree}
The map $\mathfrak{q}_{k,l}$ is of degree $2-k-2l$, and the map $\mathfrak{q}_{\emptyset,l}$ is of degree $4-2l$.
\end{prop}
\begin{prop}[Symmetry, \citelist{\cite{Ainf}*{Proposition 3.6}\cite{RQH}*{Lemma 2.16}}]\label{symmetry}
For a permutation $s\in S_{l}$ let $s\cdot\eta=\otimes_{j=1}^{l}\eta_{s(j)}$. Then
\begin{align*}
    \mathfrak{q}_{k,l}(\alpha;\eta)&=(-1)^{\sigma_{s}^{\eta}}\mathfrak{q}_{k,l}(\alpha;s\cdot\eta),\\
    \mathfrak{q}_{\emptyset,l}(\eta)&=(-1)^{\sigma_{s}^{\eta}}\mathfrak{q}_{\emptyset,l}(s\cdot\eta),
\end{align*}
where $\sigma_{s}^{\eta}=\sum_{j_{1}<j_{2},s^{-1}(j_{1})>s^{-1}(j_{2})}|\eta_{j_{1}}|\cdot|\eta_{j_{2}}|$.
\end{prop}
\begin{prop}[Energy zero]\label{ezero}
We have
\begin{align*}
    \mathfrak{q}_{k,l}^{\beta_{0}}(\alpha;\eta)=\begin{cases}-\eta_{1}|_{L},&(k,l)=(0,1),\\d\alpha_{1}\,&(k,l)=(1,0),\\(-1)^{|\alpha_{1}|}\alpha_{1}\wedge\alpha_{2},&(k,l)=(2,0),\\0,&\text{otherwise}.\end{cases}
\end{align*}
\end{prop}
\begin{prop}[Fundamental class, \cite{Ainf}*{Proposition 3.7}]\label{funclass}
Assume $\eta_{i}=1$ for some $1\leq i\leq l$. Then
\begin{align*}
    \qop(\alpha;\eta)=\begin{cases}-1,&(\beta,k,l)=(\beta_{0},0,1),\\0,&\text{otherwise}.\end{cases}
\end{align*}
\end{prop}
\par
The following is an immediate consequence of \cite{Ainf}*{propositions 3.1 and 3.2}, Proposition \ref{symmetry}, and \cite{RQH}*{Lemma 2.14}.
\begin{prop}[Unit]\label{unit}
Assume $\alpha_{i}=c\cdot1$ for some $c\in R$ and $1\leq i\leq k$. Then
\begin{align*}
    \qop(\alpha;\eta)=\begin{cases}(-1)^{|c|}c\cdot\alpha_{2},&(\beta,k,l,i)=(\beta_{0},2,0,1),\\(-1)^{|\alpha_{1}|(|c|+1)}c\cdot\alpha_{1},&(\beta,k,l,i)=(\beta_{0},2,0,2),\\0,&\text{otherwise}.\end{cases}
\end{align*}
Alternatively, assume $\eta_{i}=1$ for some $1\leq i\leq l$. Then
\begin{align*}
    \qx(\eta)=\begin{cases}\eta_{2},&(\beta,l,i)=(0,2,1),\\\eta_{1},&(\beta,l,i)=(0,2,2),\\0,&\text{otherwise}.\end{cases}
\end{align*}
\end{prop}
\par
The following is an immediate consequence of Proposition \ref{ezero} and of \cite{Ainf}*{Proposition 3.12}.
\begin{prop}[Integration]\label{topint}
We have
\begin{align*}
    \int_{L}\qop(\alpha;\eta)=\begin{cases}-\int_{L}\eta_{1}|_{L},&(\beta,k,l)=(\beta_{0},0,1),\\(-1)^{|\alpha_{1}|}\langle\alpha_{1},\alpha_{2}\rangle_{L},&(\beta,k,l)=(\beta_{0},2,0),\\0,&\text{otherwise}.\end{cases}
\end{align*}
\end{prop}
\par
\begin{con}
Given $k_{1}+k_{2}=k+1$ and $1\leq i\leq k_{1}$, we write $\alpha^{1}=\bigotimes_{j=1}^{i-1}\alpha_{j}$, $\alpha^{2}=\bigotimes_{j=i}^{i+k_{2}-1}\alpha_{j}$, and $\alpha^{3}=\bigotimes_{j=i+k_{2}}^{k}\alpha_{j}$. We also write $[l]$ for the $l$-tuple $(1,...,l)$, and for a splitting $J_{1}\sqcup J_{2}=[l]$ into disjoint tuples respecting the order of $[l]$, we write $\eta^{1}=\bigotimes_{j\in J_{1}}\eta_{j}$, $\eta^{2}=\bigotimes_{j\in J_{2}}\eta_{j}$, $l_{1}=|J_{1}|$, and $l_{2}=|J_{2}|$. 
\end{con}
Set
\begin{align*}
    \iota(\alpha,\eta;i,J_{1})=|\eta^{1}|+(|\eta^{2}|+1)(|\alpha^{1}|+i+1)+\sigma_{J_{1},J_{2}}^{\eta},
\end{align*}
where $\sigma_{J_{1},J_{2}}^{\eta}=\sum_{j_{1}\in J_{1},j_{2}\in J_{2},j_{2}<j_{1}}|\eta_{j_{1}}|\cdot|\eta_{j_{2}}|$.
\begin{prop}[Structure equation for $\M$, \cite{Ainf}*{Proposition 2.4}]\label{struc}
We have
\begin{align*}
    0&=-\mathfrak{q}_{k,l}(\alpha;d\eta)\\
    &+\sum_{\substack{k_{1}+k_{2}=k+1\\1\leq i\leq k_{1}\\J_{1}\sqcup J_{2}=[l]}}(-1)^{\iota(\alpha,\eta;i,J_{1})}\mathfrak{q}_{k_{1},l_{1}}(\alpha^{1}\otimes\mathfrak{q}_{k_{2},l_{2}}(\alpha^{2};\eta^{2})\otimes\alpha^{3};\eta^{1}).
\end{align*}
\end{prop}

\subsection{Geodesic constraints}\label{gop}

We recall the definition of the geodesic moduli spaces from \cite{RQH}*{Section 3.1}. For a nodal Riemann surface $\Sigma$ of genus $0$ with a connected boundary $\partial\Sigma$ and complex structure $j$, denote by $\overline{\Sigma}$ a copy of $\Sigma$ with the complex structure $-j$. Then $\Sigma_{\mathbb{C}}=\Sigma\coprod_{\partial\Sigma}\overline{\Sigma}$ is a closed nodal Riemann surface of genus $0$. Denote by $D$ the set of nodal points of $\Sigma_{\mathbb{C}}$ and let $\widetilde{\Sigma}_{\mathbb{C}}=\Sigma_{\mathbb{C}}\setminus D$. Denote by $\Cinf=\mathbb{C}\cup\{\infty\}$ the Riemann sphere. As in \cite{Jhol}*{Appendix D}, we can define the cross-ratio $(x_{1},x_{2},x_{3},x_{4})\in\Cinf$ of four points in $\widetilde{\Sigma}_{\mathbb{C}}$ such that no three of them coincide. We recall the definition below.
\par
For $\nu$ a component of $\Sigma_{\mathbb{C}}$ and $x\in\Sigma_{\mathbb{C}}$, if $x\in\nu$ let $x^{\nu}=x$, and otherwise let $x^{\nu}$ be the nodal point of $\nu$ corresponding to the branch in which $x$ lies. We have the following.
\begin{prop}[\cite{Jhol}*{Remark D.3.3.}]\label{root} Let $x_{1},x_{2},x_{3}\in\widetilde{\Sigma}_{\mathbb{C}}$ be pairwise distinct. Then there exists a unique component $\nu$ of $\Sigma_{\mathbb{C}}$ such that $x_{1}^{\nu},x_{2}^{\nu},x_{3}^{\nu}$ are pairwise distinct.
\end{prop}
Let $x_{1},x_{2},x_{3},x_{4}\in\widetilde{\Sigma}_{\mathbb{C}}$ such that there are three of them which are pairwise distinct, and choose any such three $x_{i_{1}},x_{i_{2}},x_{i_{3}}$. Let $\nu$ be as in Proposition \ref{root} for $x_{i_{1}},x_{i_{2}},x_{i_{3}}$, and identify $\nu$ with $\Cinf$. Then the cross-ratio is given by
\begin{align*}
    (x_{1},x_{2},x_{3},x_{4})=\frac{(x_{2}^{\nu}-x_{3}^{\nu})(x_{4}^{\nu}-x_{1}^{\nu})}{(x_{1}^{\nu}-x_{2}^{\nu})(x_{3}^{\nu}-x_{4}^{\nu})}\in\Cinf.
\end{align*}
We can now extend the cross-ratio smoothly to $x_{1},x_{2},x_{3},x_{4}\in\widetilde{\Sigma}_{\mathbb{C}}$ such that no three of the points coincide, but there are no three of them which are pairwise distinct. In this case there are two pairs of identical points, and we have
\begin{align*}
    (x_{1},x_{2},x_{3},x_{4})=
    \begin{cases}
    \infty,&x_{1}=x_{2},x_{3}=x_{4},\\
    1,&x_{1}=x_{3},x_{2}=x_{4},\\
    0,&x_{1}=x_{4},x_{2}=x_{3}.
    \end{cases}
\end{align*}
Let
\begin{align*}
    \Delta=\{\vec{x}\in\widetilde{\Sigma}_{\mathbb{C}}^{4}|\exists i,j,k:x_{i}=x_{j}=x_{k}\}.
\end{align*}
\begin{prop}[\citelist{\cite{Jhol}*{Lemma D.4.1}\cite{complex}*{Theorem 7.4}}]\label{ratio}
The cross-ratio is a well defined smooth map $\widetilde{\Sigma}_{\mathbb{C}}^{4}\setminus\Delta\rightarrow\Cinf$, and does not depend on the choice of the three points, on $\nu$, or on the identification of $\nu$ with $\Cinf$.
\end{prop}
\par
For $x\in\Sigma$ denote by $\bar{x}\in\overline{\Sigma}$ the corresponding point. For $m\in\{0,...,k\}$, whenever applicable, define
\begin{align*}
    \chi_{m}:\M\rightarrow\Cinf
\end{align*}
by
\begin{align*}
    \chi_{m}([\Sigma,u,\vec{z},\vec{w}])=
    \begin{cases}
        (z_{0},w_{1},\bar{w}_{2},w_{2}),&m=0,\\
        (z_{0},z_{m},\bar{w}_{1},w_{1}),&otherwise.
    \end{cases}
\end{align*}
By Proposition \ref{ratio}, $\chi_{m}$ is invariant under isomorphisms of stable map, and is therefore a well defined smooth map. Let $I=[0,1]$, and let
\begin{align*}
    \Mg=I\times_{\Cinf}\M,
\end{align*}
where the fiber product is taken with respect to the inclusion $I\hookrightarrow\Cinf$ and to $\chi_{m}$.
\par
Denote by $\mathcal{H}\subset\Cinf$ the upper half-plane. On the open stratum of $\M$ we can identify $\Sigma$ with $\overline{\mathcal{H}}$. Under this identification, for $m\neq0$ the condition $\chi_{m}\in I$ corresponds to the points $z_{0},w_{1},z_{m}$ being constrained to lie on a hyperbolic geodesic (in this order), and for $m=0$ this condition corresponds to the points $z_{0},w_{2},w_{1}$ being constrained to lie on a hyperbolic geodesic (in this order).
\par
The following lemma appears in \cite{prep}.
\begin{lem}\label{transg}
The inclusion $I\hookrightarrow\Cinf$ is transversal to $\chi_{m}$.
\end{lem}
Thus, if $\M$ is a smooth orbifold with corners, then so are all the spaces $\Mg$. We also have the following lemma.
\begin{lem}[\cite{RQH}*{Lemma 3.11}]\label{forsgn}
Denote by $\pi:\Mg\rightarrow\mathcal{M}_{k,l}(\beta)$ the forgetful map that forgets the boundary marked point $z_{0}$ (and hence the geodesic), shifts the labeling of the rest, and stabilizes the resulting map. Then $\pi$ restricts to a diffeomorphism from the open stratum to an open subset of the open stratum that changes orientation by $\sgn\pi=n+1$.
\end{lem}
\par
Denote by
\begin{align*}
    \evb:\Mg\rightarrow L,\quad j=0,...,k,\\
    \evi:\Mg\rightarrow X,\quad j=1,...,l,
\end{align*}
the evaluation maps (note the abuse of notation, since these maps are restrictions of the ones previously denoted by $\evb$, $\evi$). Assuming as in \cite{RQH} that $\evbz$ is a proper submersion (see Section \ref{reg}), define
\begin{align*}
    \qg:C^{\otimes k}\otimes E^{\otimes l}\rightarrow C
\end{align*}
by
\begin{align*}
    \qg(\alpha;\eta)=(-1)^{\varepsilon(\alpha)+|\alpha|+k}(\evbz)_{*}\Big(\bigwedge_{j=1}^{l}(\evi)^{*}\eta_{j}\wedge\bigwedge_{j=1}^{k}(\evb)^{*}\alpha_{j}\Big).
\end{align*}
Set
\begin{align*}
    \mathfrak{q}_{k,l;m}=\sum_{\beta\in\Pi}T^{\beta}\qg.
\end{align*}
\par
These operators have the following properties.
\begin{prop}[Degree, \cite{RQH}*{Lemma 3.7}]\label{gdegree}
The map $\mathfrak{q}_{k,l;m}$ is of degree $3-k-2l$.
\end{prop}
\begin{prop}[Unit, \cite{RQH}*{Lemma 3.13}]\label{gunit}
Assume $\alpha_{i}=c\cdot1$ for some $c\in R$ and $1\leq i\leq k$, $i\neq m$. Then $\mathfrak{q}_{k,l;m}(\alpha;\eta)=0$.
\end{prop}
\par
Set
\begin{align*}
    \iota_{0}(\alpha,\eta;i,J_{1})&=|\eta^{1}|+(|\eta^{2}|+1)(|\alpha^{1}|+i)+\sigma_{J_{2},J_{1}}^{\eta},\\
    \iota_{1}(\alpha,\eta;i,J_{1})&=|\eta^{1}|+(|\eta^{2}|+1)(|\alpha^{1}|+i)+\sigma_{J_{1},J_{2}}^{\eta},\\
    \iota_{2}(\alpha,\eta;i,J_{1})&=|\eta^{1}|+|\eta^{2}|(|\alpha^{1}|+i+1)+\sigma_{J_{1},J_{2}}^{\eta},\\
    \iota_{3}(\alpha,\eta;J_{1})&=|\eta|+|\alpha|+k+n+\sigma^{\eta}_{J_{2},J_{1}},
\end{align*}
and let
\begin{align*}
    m'(m,i,k_{2})=
    \begin{cases}
    m,&1\leq m<i,\\
    i,&i\leq m<i+k_{2},\\
    m-k_{2}+1,&i+k_{2}\leq m\leq k.
    \end{cases}
\end{align*}
\begin{prop}[Structure equation for $\Mg$, $m\neq0$, \cite{RQH}*{Proposition 3.1}]\label{gstruc}
We have
\begin{align*}
    0&=-\mathfrak{q}_{k,l;m}(\alpha;d\eta)\\
    &+\sum_{\substack{k_{1}+k_{2}=k+1\\1\leq i\leq k_{1}\\J_{1}\sqcup J_{2}=[l]\\1\in J_{1}}}(-1)^{\iota_{1}(\alpha,\eta;i,J_{1})}\mathfrak{q}_{k_{1},l_{1};m'(m,i,k_{2})}(\alpha^{1}\otimes \mathfrak{q}_{k_{2},l_{2}}(\alpha^{2};\eta^{2})\otimes\alpha^{3};\eta^{1})\\
    &+\sum_{\substack{k_{1}+k_{2}=k+1\\m-k_{2}+1\leq i\leq m\\J_{1}\sqcup J_{2}=[l]\\1\in J_{2}}}(-1)^{\iota_{2}(\alpha,\eta;i,J_{1})}\mathfrak{q}_{k_{1},l_{1}}(\alpha^{1}\otimes \mathfrak{q}_{k_{2},l_{2};m-i+1}(\alpha^{2};\eta^{2})\otimes\alpha^{3};\eta^{1}).
\end{align*}
\end{prop}
\begin{prop}[Structure equation for $\Mgz$, \cite{RQH}*{Proposition 3.2}]\label{gzstruc}
We have
\begin{align*}
    0&=-\mathfrak{q}_{k,l;0}(\alpha;d\eta)\\
    &+\sum_{\substack{k_{1}+k_{2}=k+1\\1\leq i\leq k_{1}\\J_{1}\sqcup J_{2}=[l]\\1,2\in J_{1}}}(-1)^{\iota_{1}(\alpha,\eta;i,J_{1})}\mathfrak{q}_{k_{1},l_{1};0}(\alpha^{1}\otimes \mathfrak{q}_{k_{2},l_{2}}(\alpha^{2};\eta^{2})\otimes\alpha^{3};\eta^{1})\\
    &+\sum_{\substack{k_{1}+k_{2}=k+1\\1\leq i\leq k_{1}\\J_{1}\sqcup J_{2}=[l]\\2\in J_{1},1\in J_{2}}}(-1)^{\iota_{0}(\alpha,\eta;i,J_{1})}\mathfrak{q}_{k_{1},l_{1};i}(\alpha^{1}\otimes \mathfrak{q}_{k_{2},l_{2}}(\alpha^{2};\eta^{2})\otimes\alpha^{3};\eta^{1})\\
    &+\sum_{\substack{k_{1}+k_{2}=k+1\\1\leq i\leq k_{1}\\J_{1}\sqcup J_{2}=[l]\\1,2\in J_{2}}}(-1)^{\iota_{2}(\alpha,\eta;i,J_{1})}\mathfrak{q}_{k_{1},l_{1}}(\alpha^{1}\otimes \mathfrak{q}_{k_{2},l_{2};0}(\alpha^{2};\eta^{2})\otimes\alpha^{3};\eta^{1})\\
    &+\sum_{\substack{J_{1}\sqcup J_{2}=[l]\\1,2\in J_{2}}}(-1)^{\iota_{3}(\alpha,\eta;J_{1})}\mathfrak{q}_{k_{1},l_{1}}(\alpha;\mathfrak{q}_{\emptyset,l_{2}}(\eta^{2})\otimes\eta^{1})\\
\end{align*}
\end{prop}
\par
We also have the following property.
\begin{prop}[Fundamental class on the geodesic]\label{gfunclass}
For $m\neq0$ we have
\begin{align*}
    \qg(\alpha;1\otimes\eta)=
    \begin{cases}
    (-1)^{n}\alpha_{1},&(\beta,k,l,m)=(\beta_{0},1,1,1),\\
    0,&\text{otherwise}.
    \end{cases}
\end{align*}
\end{prop}
\begin{proof}
The proof is similar to that of the non-geodesic case in \cite{Ainf}*{Proposition 3.7}. Denote by $\pi:\Mg\rightarrow\mathcal{M}_{k+1,l-1}(\beta)$ the forgetful map that forgets the first interior marked point (and hence the geodesic), shifts the labeling of the rest, and stabilizes the resulting map. It is defined whenever $(\beta,k,l,m)\neq(\beta_{0},1,1,1)$. Denote by $\evb,\evi$ the evaluation maps for $\Mg$ and by $\widetilde{\evb},\widetilde{\evi}$ the ones for $\mathcal{M}_{k+1,l-1}(\beta)$, and note that
\begin{align*}
    \evb&=\widetilde{\evb}\circ\pi,\ \ \quad 1\leq j\leq k,\\
    \evi&=\widetilde{evi^{\beta}_{j-1}}\circ\pi,\quad 2\leq j\leq l.
\end{align*}
Recall the map $\varphi$ from Definition \ref{curhom}. By propositions \ref{subcomp} and \ref{int}, we have
\begin{align*}
    \varphi(\qg(\alpha;1\otimes\eta))&=\pm\varphi\Big((\evbz)_{*}\Big(\bigwedge_{j=2}^{l}(\evi)^{*}\eta_{j}\wedge\bigwedge_{j=1}^{k}(\evb)^{*}\alpha_{j}\wedge1\Big)\Big)\\
    &=\pm(\evbz)_{*}\Big(\bigwedge_{j=2}^{l}(\evi)^{*}\eta_{j}\wedge\bigwedge_{j=1}^{k}(\evb)^{*}\alpha_{j}\wedge\varphi(1)\Big)\\
    &=\pm(\widetilde{\evbz})_{*}\pi_{*}\Big(\pi^{*}\Big(\bigwedge_{j=1}^{l-1}(\widetilde{\evi})^{*}\eta_{j}\wedge\bigwedge_{j=1}^{k}(\widetilde{\evb})^{*}\alpha_{j}\Big)\wedge\varphi(1)\Big)\\
    &=\pm \mathfrak{q}_{k,l-1}^{\beta}(\alpha;\eta)\wedge\pi_{*}\varphi(1).
\end{align*}
Now,
\begin{align*}
    \deg\pi_{*}\varphi(1)=\deg\varphi(1)-\rdim\pi=0-1=-1,
\end{align*}
therefore $\pi_{*}\varphi(1)=0$. Since by Proposition \ref{curhomprop} $\varphi$ is injective, it follows that $\qg(\alpha;\eta)=0$.
\par
We turn to the case $(\beta,k,l,m)=(\beta_{0},1,1,1)$. In this case $evb^{\beta_{0}}_{0}=evb^{\beta_{0}}_{1}$, so that by Proposition \ref{int} we have
\begin{align*}
    \mathfrak{q}_{1,1;1}^{\beta_{0}}(\alpha_{1};1)=-(evb^{\beta_{0}}_{1})_{*}(evb^{\beta_{0}}_{1})^{*}\alpha_{1}=-\alpha_{1}\wedge(evb^{\beta_{0}}_{1})_{*}1.
\end{align*}
Let $\pi:\mathcal{M}_{2,1;1}(\beta_{0})\rightarrow\mathcal{M}_{1,1}(\beta_{0})$ be the forgetful map from Lemma \ref{forsgn}. Denote by $\widetilde{evb^{\beta_{0}}_{0}}$, $\widetilde{evi^{\beta_{0}}_{1}}$ the evaluation maps for $\mathcal{M}_{1,1}(\beta_{0})$. By Lemma \ref{forsgn} and Proposition \ref{funclass} we have
\begin{align*}
    (evb^{\beta_{0}}_{1})_{*}1&=(\widetilde{evb^{\beta_{0}}_{0}})_{*}\pi_{*}1\\&=(-1)^{n+1}(\widetilde{evb^{\beta_{0}}_{0}})_{*}1\\&=
    (-1)^{n+1}(\widetilde{evb^{\beta_{0}}_{0}})_{*}(\widetilde{evi^{\beta_{0}}_{1}})^{*}1\\&=
    (-1)^{n}\mathfrak{q}_{0,1}^{\beta_{0}}(1)\\&=(-1)^{n+1}.
\end{align*}
\end{proof}

\subsection{Horocyclic constraints}\label{horop}

In \cite{OC}, Hugtenburg defines the moduli space of disks with a horocyclic constraint on the points $z_{0},w_{1},w_{2}$, which we denote by $\Mh$. Recall that a horocycle in the disk is a circle tangent to the boundary of that disk. The open stratum of $\Mh$ is the subset of the open stratum of $\M$, in which the points $z_{0},w_{1},w_{2}$ lie on a horocycle in this order (counter-clockwise). This subset is well defined since horocycles are preserved by automorphisms of the disk. The space $\Mh$ is then the closure of its open stratum inside of $\M$.
\par
More precisely, the space $\Mh$ is defined in \cite{OC}*{Section 4.4} as a fiber product
\begin{align*}
    \Mh=I\times_{D^{2}}\M
\end{align*}
with respect to given maps $I\rightarrow D^{2}$ and $\M\rightarrow D^{2}$. The proof of \cite{OC}*{Lemma 4.19} shows that under our regularity assumptions these maps are transversal, thus $\Mh$ is a smooth orbifold with corners.
\par
Denote by
\begin{align*}
    \evb:\Mh\rightarrow L,\quad j=0,...,k,\\
    \evi:\Mh\rightarrow X,\quad j=1,...,l,
\end{align*}
the evaluation maps (note the abuse of notation, since these maps are restrictions of the ones denoted by $\evb$, $\evi$ in section \ref{op}). As in \cite{OC}*{Assumptions 4.18}, we assume that the spaces $\Mh$ are oriented orbifolds with corners, and that the maps $\evbz$ are proper submersions (see Section \ref{reg}).
\par
Define
\begin{align*}
    \qh:C^{\otimes k}\otimes E^{\otimes l}\rightarrow C
\end{align*}
by
\begin{align*}
    \qh(\alpha;\eta)=(-1)^{\varepsilon(\alpha)+|\alpha|+|\eta|+n+1}(\evbz)_{*}\Big(\bigwedge_{j=1}^{l}(\evi)^{*}\eta_{j}\wedge\bigwedge_{j=1}^{k}(\evb)^{*}\alpha_{j}\Big),
\end{align*}
and set
\begin{align*}
    \mathfrak{q}_{k,l;\perp}=\sum_{\beta\in\Pi}T^{\beta}\qh.
\end{align*}
\par
\begin{con}
Given $k_{1}+k_{2}+k_{3}=k+1$ and $1\leq i_{1}<i_{2}\leq k_{1}$, write $\alpha^{1}=\bigotimes_{j=1}^{i_{1}-1}\alpha_{j}$, $\alpha^{2}=\bigotimes_{j=i_{1}}^{i_{1}+k_{2}-1}\alpha_{j}$, $\alpha^{3}=\bigotimes_{j=i_{1}+k_{2}}^{i_{2}+k_{2}-1}\alpha_{j}$, $\alpha^{4}=\bigotimes_{j=i_{2}+k_{2}}^{i_{2}+k_{2}+k_{3}-1}\alpha_{j}$, and $\alpha^{5}=\bigotimes_{j=i_{2}+k_{2}+k_{3}}^{k}\alpha_{j}$. For a splitting $J_{1}\sqcup J_{2}\sqcup J_{3}=[l]$ into disjoint tuples respecting the order of $[l]$, write $\eta^{1}=\bigotimes_{j\in J_{1}}\eta_{j}$, $\eta^{2}=\bigotimes_{j\in J_{2}}\eta_{j}$, $\eta^{3}=\bigotimes_{j\in J_{3}}\eta_{j}$, $l_{1}=|J_{1}|$, $l_{2}=|J_{2}|$, and $l_{3}=|J_{3}|$. Write also $i'=i_{2}+k_{2}$, $J'=J_{1}\sqcup J_{2}$, $\alpha'=\bigotimes_{j=1}^{i'-1}\alpha_{j}$, and $\eta'=\bigotimes_{j\in J'}\eta_{j}$. 
\end{con}
Set
\begin{align*}
    \iota_{4}(\alpha,\eta;i,J_{1})&=|\eta^{2}|(|\alpha^{1}|+i+1)+1+\sigma_{J_{1},J_{2}}^{\eta},\\
    \dagger&=\iota(\alpha',\eta';i_{1},J_{1})+\iota(\alpha,\eta;i',J').
\end{align*}
\begin{prop}[Structure equation for $\Mh$, \cite{OC}*{Proposition 5.1}]\label{hstruc}
We have
\begin{align*}
    0&=-\mathfrak{q}_{k,l}(\alpha;d\eta)\\
    &+\sum_{\substack{k_{1}+k_{2}=k+1\\1\leq i\leq k_{1}\\J_{1}\sqcup J_{2}=[l]\\1,2\in j_{1}}}(-1)^{\iota(\alpha,\eta;i,J_{1})}\mathfrak{q}_{k_{1},l_{1};\perp}(\alpha^{1}\otimes \mathfrak{q}_{k_{2},l_{2}}(\alpha^{2};\eta^{2})\otimes\alpha^{3};\eta^{1})\\
    &+\sum_{\substack{k_{1}+k_{2}=k+1\\1\leq i\leq k_{1}\\J_{1}\sqcup J_{2}=[l]\\1,2\in j_{2}}}(-1)^{\iota_{4}(\alpha,\eta;i,J_{1})}\mathfrak{q}_{k_{1},l_{1}}(\alpha^{1}\otimes \mathfrak{q}_{k_{2},l_{2};\perp}(\alpha^{2};\eta^{2})\otimes\alpha^{3};\eta^{1})\\
    &+\sum_{\substack{J_{1}\sqcup J_{2}=[l]\\1,2\in J_{2}}}(-1)^{\sigma_{J_{1},J_{2}}^{\eta}}\mathfrak{q}_{k,l_{1}}(\alpha;\mathfrak{q}_{\emptyset,l_{2}}(\eta^{2})\otimes\eta^{1})\\
    &+\sum_{\substack{k_{1}+k_{2}+k_{3}=k+1\\1\leq i_{1}<i_{2}\leq k_{1}\\J_{1}\sqcup J_{2}\sqcup J_{3}=[l]\\1\in J_{3},2\in J_{2}}}(-1)^{\dagger}\mathfrak{q}_{k_{1},l_{1}}(\alpha^{1}\otimes \mathfrak{q}_{k_{2},l_{2}}(\alpha^{2};\eta^{2})\otimes\alpha^{3}\otimes \mathfrak{q}_{k_{4},l_{3}}(\alpha^{4};\eta^{3})\otimes\alpha^{5};\eta^{1}).
\end{align*}
\end{prop}
\par
We also have the following property.
\begin{prop}[Unit]\label{horunit}
Assume $\alpha_{i}=c\cdot1$ for some $c\in R$ and $1\leq i\leq k$. Then $\qh(\alpha;\eta)=0$.
\end{prop}
\begin{proof}
Denote by $\pi:\Mh\rightarrow\mathcal{M}_{k,l;\perp}(\beta)$ the forgetful map that forgets the $i$th boundary marked point, shifts the labeling of the following boundary points, and stabilizes the resulting map. Write $\alpha=\alpha^{1}\otimes c\cdot1\otimes\alpha^{2}$. Similarly to the proof of Proposition \ref{gfunclass}, we get
\begin{align*}
    \varphi(\qh&(\alpha;\eta))=\pm c\cdot \mathfrak{q}_{k-1,l;\perp}^{\beta}(\alpha^{1}\otimes\alpha^{2};\eta)\wedge\pi_{*}\varphi(1).
\end{align*}
Now,
\begin{align*}
    \deg\pi_{*}\varphi(1)=\deg\varphi(1)-\rdim\pi=0-1=-1,
\end{align*}
thus $\pi_{*}\varphi(1)=0$. Since by Proposition \ref{curhomprop} $\varphi$ is injective, it follows that $\qh(\alpha;\eta)=0$.
\end{proof}

\subsection{Bounding pairs}\label{defop}
For a graded module $\Upsilon$, we denote by $\Upsilon_{r}\subset\Upsilon$ or by $(\Upsilon)_{r}\subset\Upsilon$ the subset of homogeneous elements having degree $r$. Let $\gamma\in(\mathcal{I}E)_{2}$ be closed and let $b\in(\mathcal{I}C)_{1}$. Let
\begin{align*}
    \alpha^{b;s_{0},...,s_{k}}&=b^{\otimes s_{0}}\otimes\alpha_{1}\otimes b^{\otimes s_{1}}\otimes\cdot\cdot\cdot\otimes\alpha_{k}\otimes b^{\otimes s_{k}},\\
    \eta^{\gamma;t}&=\frac{1}{t!}\eta\otimes\gamma^{\otimes t}.
\end{align*}
Define deformed operators as follows.
\begin{align*}
    \mathfrak{q}_{k,l}^{\gamma,b}(\alpha;\eta)&=\sum_{t,s_{0},...,s_{k}\geq0}\mathfrak{q}_{k+s_{0}+...+s_{k},l+t}(\alpha^{b;s_{0},...,s_{k}};\eta^{\gamma;t}),\\
    \mathfrak{q}_{\emptyset,l}^{\gamma}(\eta)&=\sum_{t\geq0}\mathfrak{q}_{\emptyset,l+t}(\eta^{\gamma;t}),\\
    \mathfrak{q}_{k,l;0}^{\gamma,b}(\alpha;\eta)&=\sum_{t,s_{0},...,s_{k}\geq0}\mathfrak{q}_{k+s_{0}+...+s_{k},l+t;0}(\alpha^{b;s_{0},...,s_{k}};\eta^{\gamma;t}),\\
    \mathfrak{q}_{k,l;m}^{\gamma,b}(\alpha;\eta)&=\sum_{t,s_{0},...,s_{k}\geq0}\mathfrak{q}_{k+s_{0}+...+s_{k},l+t;m+s_{0}+...+s_{m-1}}(\alpha^{b;s_{0},...,s_{k}};\eta^{\gamma;t}),\\
    \mathfrak{q}_{k,l;\perp}^{\gamma,b}(\alpha;\eta)&=\sum_{t,s_{0},...,s_{k}\geq0}\mathfrak{q}_{k+s_{0}+...+s_{k},l+t;\perp}(\alpha^{b;s_{0},...,s_{k}};\eta^{\gamma;t}).
\end{align*}
Similarly, define operators of the form $\mathfrak{q}_{\bullet}^{\beta,\gamma,b}$. We have the following proposition (see \citelist{\cite{RQH}*{Lemma 3.17}\cite{OC}*{Remark 4.36}}).
\begin{prop}\label{defprop}
All the structure equations and properties of sections \ref{op}, \ref{gop} and \ref{horop} hold for the deformed operators as well, with the exception of Proposition \ref{ezero} which reads
\begin{align*}
    \mathfrak{q}_{k,l}^{\beta_{0},\gamma,b}(\alpha;\eta)=\begin{cases}db-\gamma|_{L},&(k,l)=(0,0),\\-\eta_{1}|_{L},&(k,l)=(0,1),\\d\alpha_{1}\,&(k,l)=(1,0),\\(-1)^{|\alpha_{1}|}\alpha_{1}\wedge\alpha_{2},&(k,l)=(2,0),\\0,&\text{otherwise},\end{cases}
\end{align*}
and Proposition \ref{topint} which reads
\begin{align*}
    \int_{L}\mathfrak{q}_{k,l}^{\beta,\gamma,b}(\alpha;\eta)=\begin{cases}-\int_{L}\gamma|_{L},&(\beta,k,l)=(\beta_{0},0,0),\\-\int_{L}\eta_{1}|_{L},&(\beta,k,l)=(\beta_{0},0,1),\\(-1)^{|\alpha_{1}|}\langle\alpha_{1},\alpha_{2}\rangle_{L},&(\beta,k,l)=(\beta_{0},2,0),\\0,&\text{otherwise}.\end{cases}
\end{align*}
\end{prop}
\begin{rmk}
For the rest of this thesis, we may refer to the structure equations and properties of the non-deformed operators (except for propositions \ref{ezero} and \ref{topint}) and apply them to the deformed operators without further comment.
\end{rmk}
\par
We turn to define bounding pairs, which are required for the construction of Lagrangian Floer cohomology.
\begin{defn}\label{bpdef}
We say $(\gamma,b)$ is a bounding pair if there exists $c\in(\mathcal{I}R)_{2}$ such that
\begin{align*}
    \mathfrak{q}_{0,0}^{\gamma,b}=c\cdot1.
\end{align*}
\end{defn}
\begin{rmk}\label{deftopint}
Note that if $(\gamma,b)$ is a bounding pair then
\begin{align*}
    \int_{L}\mathfrak{q}_{0,0}^{\gamma,b}=\int_{L}c\cdot1=0,
\end{align*}
and Proposition \ref{topint} holds for the deformed operators without the adjustment made in Proposition \ref{defprop}.
\end{rmk}
\par
The following lemma and proposition ensure the existence of bounding pairs under our assumptions.
\begin{lem}\label{BPlem}
Assume that $\mathfrak{q}^{\beta}_{k,l}$ doesn't vanish. Then either $\beta=\beta_{0}$ or $\mu(\beta)\geq2$.
\end{lem}
\begin{proof}
Let $\beta\neq\beta_{0}$. By assumption, $\mathcal{M}_{1,0}(\beta)\neq\emptyset$. Proposition \ref{dim} together with the assumption that $evb^{\beta}_{0}$ is a submersion imply
\begin{align*}
    0\leq\rdim evb^{\beta}_{0}=(n-2+\mu(\beta))-n=\mu(\beta)-2.
\end{align*}
\end{proof}
\begin{rmk}\label{BPrmk}
In the proof of Lemma \ref{BPlem}, the assumption that $\evbz$ is a submersion is not only a regularity assumption, but is essential to the proof. Note that even if this assumption is dropped, the conclusion of the lemma still holds in case $L$ is monotone.
\end{rmk}
\begin{prop}\label{BPexist}
Let $b=\sum_{j}e_{j}b_{j}\in(\mathcal{I}C)_{1}$ be such that $b_{j}\in A^{\leq1}(L)$ and $e_{j}\in R_{1-|b_{j}|}$, and let $\gamma=\sum_{i}\lambda_{i}\gamma_{i}\in(\mathcal{I}E)_{2}$ be such that $\gamma_{i}\in A^{\leq2}(X)$ and $\lambda_{i}\in R_{2-|\gamma_{i}|}$. Assume also that $b,\gamma$ are closed and that $\gamma|_{L}\in A^{0}(L)\otimes R_{2}$. Then $(\gamma,b)$ is a bounding pair. In particular, $(0,0)$ is a bounding pair.
\end{prop}
\begin{proof}
Applying Proposition \ref{struc} with $k=l=0$ yields
\begin{align*}
    0&=\mathfrak{q}_{1,0}^{\gamma,b}(\mathfrak{q}_{0,0}^{\gamma,b}).
\end{align*}
Consider
\begin{align*}
    \mathfrak{q}_{0,0}^{\gamma,b}=\sum_{\beta\in\Pi}\sum_{s,t\geq0}\sum_{i_{1},...,i_{t},j_{1},...,j_{s}}\frac{1}{t!}T^{\beta}\mathfrak{q}^{\beta}_{s,t}(\otimes_{r=1}^{s}e_{r}b_{j_{r}};\otimes_{r=1}^{t}\lambda_{r}\gamma_{i_{r}}).
\end{align*}
Fix $t,s,i_{1},...,i_{t},j_{1},...,j_{s}$ and let $\hat{\gamma}=\otimes_{r=1}^{t}\gamma_{i_{r}}$, $\hat{b}=\otimes_{r=1}^{s}b_{j_{r}}$. By Proposition \ref{degree} and Lemma \ref{BPlem}, for $\beta\neq\beta_{0}$ we have
\begin{align*}
    |\mathfrak{q}^{\beta}_{s,t}(\hat{b};\hat{\gamma})|
    &=2-s-2t+\sum_{r=1}^{s}|b_{j_{r}}|+\sum_{r=1}^{t}|\gamma_{i_{r}}|-\mu(\beta)\\
    &\leq2-s-2t+s+2t-2\\
    &=0,
\end{align*}
with equality possible only when $\mu(\beta)=2$, $|\gamma_{i_{r}}|=2$ for all $1\leq r\leq t$, and $|b_{j_{r}}|=1$ for all $1\leq r\leq s$, in which case $\mathfrak{q}^{\beta}_{s,t}(\hat{b};\hat{\gamma})\in A^{0}(L)$. Therefore $\mathfrak{q}_{0,0}^{\beta,\gamma,b}\in A^{0}(L)\otimes R_{0}$. Furthermore, by Proposition \ref{defprop} we have $\mathfrak{q}_{0,0}^{\beta_{0},\gamma,b}=-\gamma|_{L}$. Therefore
\begin{align*}
    \mathfrak{q}_{0,0}^{\gamma,b}=-\gamma|_{L}+\sum_{\mu(\beta)=2}T^{\beta}\mathfrak{q}_{0,0}^{\beta,\gamma,b}\in A^{0}(L)\otimes R_{2}.
\end{align*}
\par
Consider now
\begin{align*}
    0
    &=\mathfrak{q}_{1,0}^{\gamma,b}(\mathfrak{q}_{0,0}^{\gamma,b})\\
    &=\sum_{\beta\in\Pi}\sum_{t\geq0,s\geq\rho\geq0}\sum_{i_{1},...,i_{t},j_{1},...,j_{s}}\frac{1}{t!}T^{\beta}\mathfrak{q}^{\beta}_{1+s,t}(\otimes_{r=1}^{\rho}e_{r}b_{j_{r}}\otimes\mathfrak{q}_{0,0}^{\gamma,b}\otimes\otimes_{r=\rho+1}^{s}e_{r}b_{j_{r}};\otimes_{r=1}^{t}\lambda_{r}\gamma_{i_{r}}).
\end{align*}
Fix $t,s,\rho,i_{1},...,i_{t},j_{1},...,j_{s}$ and let $\hat{\gamma}=\otimes_{r=1}^{t}\gamma_{i_{r}}$, $\hat{b}_{-}=\otimes_{r=1}^{\rho}b_{j_{r}}$, $\hat{b}_{+}=\otimes_{r=\rho+1}^{s}b_{j_{r}}$, $f\in A^{0}(L)$. Again by Proposition \ref{degree} and Lemma \ref{BPlem},
\begin{align*}
    |\mathfrak{q}^{\beta}_{1+s,t}(\hat{b}_{-}\otimes f\otimes\hat{b}_{+};\hat{\gamma})|
    &=2-(1+s)-2t+\sum_{r=1}^{s}|b_{j_{r}}|+\sum_{r=1}^{t}|\gamma_{i_{r}}|-\mu(\beta)\\
    &\leq1-s-2t+s+2t-2\\
    &<0
\end{align*}
unless $\beta=\beta_{0}$. Furthermore, by Proposition \ref{defprop} we have $\mathfrak{q}_{1,0}^{\beta_{0},\gamma,b}(f)=df$. Therefore
\begin{align*}
    0=d\mathfrak{q}_{0,0}^{\gamma,b},
\end{align*}
so that $\mathfrak{q}_{0,0}^{\gamma,b}=c\cdot1$ for some $c\in R_{2}$.
\end{proof}

\section{Geodesics through 4 points}

\subsection{The moduli space}\label{space}
In order to prove Theorem \ref{algebra}, we need the moduli space of disks with a geodesic constraint on the points $z_{0},w_{1},w_{2},z_{m}$ (in this order), denoted by $\Mgg$. We shall define this space as a fiber product of smooth orbifolds with corners following \cite{RQH}*{Section 3.1}. We shall show that this fiber product is indeed transversal, making it a smooth orbifold with corners.
\par
Recall that $\chi_{m}:\M\rightarrow\Cinf$ is defined in Section \ref{gop} via
\begin{align*}
    \chi_{m}([\Sigma,u,\vec{z},\vec{w}])=\begin{cases}
        (z_{0},w_{1},\bar{w}_{2},w_{2}),&m=0,\\
        (z_{0},z_{m},\bar{w}_{1},w_{1}),&otherwise.
    \end{cases}
\end{align*}
Define $\chi_{0,m}:\M\rightarrow\Cinf^{2}$ for $m\neq0$ by $\chi_{0,m}=(\chi_{0},\chi_{m})$ whenever applicable.
\par
To show transversality we require the following lemmas.
\begin{lem}\label{quot}
Let $\Sigma$ be a genus zero Riemann surface with a connected boundary (possibly empty), that is, either the upper half-plane $\mathcal{H}$ or the Riemann sphere $\Cinf$. If $\Sigma=\mathcal{H}$, let $k,l$ be such that $k\geq1$ and $k+2l\geq3$, and if $\Sigma=\Cinf$, let $k=0$ and $l\geq3$. Write
\begin{align*}
    \Delta=\{\vec{x}\in\Sigma^{k+l}|\exists i,j:x_{i}=x_{j}\},
\end{align*}
and let
\begin{align*}
    \Omega=((\partial\Sigma)^{k}\times\mathring{\Sigma}^{l})\setminus\Delta.
\end{align*}
Then
\begin{align*}
    G=\Aut(\Sigma)=
    \begin{cases}
    \pslr,&\Sigma=\mathcal{H},\\
    \pslc,&\Sigma=\Cinf,
    \end{cases}
\end{align*}
acts on $\Omega$ freely and properly, so that $\Omega/G$ is a smooth manifold with the quotient map $g:\Omega\rightarrow\Omega/G$ a surjective submersion.
\end{lem}
\begin{proof}
Recall that an element $\phi\in\pslc$ is determined by its action on any $3$ given points, thus if $\phi$ has $l\geq3$ fixed points then it is the identity. Hence the action of $G=\pslc$ is free. Furthermore, any element $\phi\in\pslr$ with $k$ boundary fixed points and $l$ interior fixed points can be considered as an element of $\pslc$ with $k+2l$ fixed points by taking the complex conjugate, thus if $k+2l\geq3$ then $\phi$ is the identity. Hence the action of $G=\pslr$ is free.
\par
To see that the action is proper, consider the map
\begin{align*}
    \Phi:G\times\Omega&\rightarrow\Omega\times\Omega\\
    (\phi,\vec{x})&\mapsto(\phi.\vec{x},\vec{x}).
\end{align*}
Consider $\Omega_{0}$, defined in the same manner as $\Omega$ with $k_{0}=l_{0}=1$ for $\Sigma=\mathcal{H}$ and with $l_{0}=3$ for $\Sigma=\Cinf$. We similarly define
\begin{align*}
    \Phi_{0}:G\times\Omega_{0}&\rightarrow\Omega_{0}\times\Omega_{0}\\
    (\phi,\vec{x})&\mapsto(\phi.\vec{x},\vec{x}).
\end{align*}
By the same reasoning as above, $\Phi_{0}$ is bijective, and by the invariance of domain theorem, it is a homeomorphism onto its image. Therefore $\Phi_{0}$ is a homeomorphism. Let $P:\Omega\rightarrow\Omega_{0}$ be any projection.
\par
Now let $K\subset\Omega\times\Omega$ be compact, and let $\{(\phi_{i},\vec{x}_{i})\}_{i\in\mathbb{N}}\subset\Phi^{-1}(K)$ be a sequence. Then, possibly passing to a subsequence, $\{(\phi_{i}.\vec{x}_{i},\vec{x}_{i})\}_{i\in\mathbb{N}}\subset K$ converges. In particular, since $(P\times P)(K)$ is compact, $\{(\phi_{i}.P(\vec{x}_{i}),P(\vec{x}_{i}))\}_{i\in\mathbb{N}}$ converges. Since $\Phi_{0}^{-1}((P\times P)(K))$ is compact, it follows that, again possibly passing to a subsequence, $\{\Phi_{0}^{-1}(\phi_{i}.P(\vec{x}_{i}),P(\vec{x}_{i}))\}_{i\in\mathbb{N}}=\{(\phi_{i},P(\vec{x}_{i}))\}_{i\in\mathbb{N}}$ converges. Therefore $\{(\phi_{i},\vec{x}_{i})\}_{i\in\mathbb{N}}$ converges, so that $\Phi^{-1}(K)$ is compact and the action is proper.
\end{proof}
\begin{lem}\label{fixquot}
In the setting of Lemma \ref{quot}, let $s\leq k$, $t\leq l$ be such that $s+2t\leq\dim G$. Fix some $z_{1},...,z_{s}\in\partial\Sigma$ and $w_{1},...,w_{t}\in\mathring{\Sigma}$, and let $\widehat{G}\subset G$ be the subgroup consisting of those reparameterizations which fix $z_{1},...,z_{s}$ and $w_{1},...,w_{t}$. Let
\begin{align*}
    \widehat{\Omega}=(\{(z_{1},...,z_{s})\}\times(\partial\Sigma)^{k-s}\times\{(w_{1},...,w_{t})\}\times\mathring{\Sigma}^{l-t})\setminus\Delta\subset\Omega.
\end{align*}
Then $\widehat{G}$ acts on $\widehat{\Omega}$ freely and properly, so that $\widehat{\Omega}/\widehat{G}$ is a smooth manifold with  the quotient map $q:\widehat{\Omega}\rightarrow\widehat{\Omega}/\widehat{G}$ a surjective submersion. Furthermore, the map
\begin{align*}
    \psi:\widehat{\Omega}/\widehat{G}\rightarrow\Omega/G
\end{align*}
given by
\begin{align*}
    [\vec{x}]\mapsto[\vec{x}]
\end{align*}
is a well defined diffeomorphism.
\end{lem}
\begin{proof}
The proof that the action is free and proper is similar to the proof of Lemma \ref{quot}. Clearly $\psi$ is well defined and bijective.
\par
We claim $\psi$ is smooth. Denote by $f:\widehat{\Omega}\rightarrow\Omega$ the inclusion, then the following diagram commutes:
\begin{center}
\begin{tikzcd}
    \widehat{\Omega}\arrow{r}{f}\arrow{d}{q}&\Omega\arrow{d}{g}\\
    \widehat{\Omega}/\widehat{G}\arrow{r}{\psi}&\Omega/G.
\end{tikzcd}
\end{center}
Locally, in a small open $\widehat{U}\subset\widehat{\Omega}/\widehat{G}$, we have a section $\widehat{\sigma}:\widehat{U}\rightarrow\widehat{\Omega}$. Then $q\circ\widehat{\sigma}=\id$, so that
\begin{align*}
    \psi|_{\widehat{U}}=\psi\circ q\circ\widehat{\sigma}=g\circ f\circ\widehat{\sigma}.
\end{align*}
Since smoothness is a local property, it follows that $\psi$ is smooth.
\par
To show $\psi^{-1}$ is smooth, as a first step, we construct a smooth map $\lambda:\Omega\rightarrow G$ by cases.
\par
\textbf{Case 1.} $\Sigma=\Cinf$. We can assume without loss of generality that $w_{1}=0$ if $t\geq1$, $w_{2}=\infty$ if $t\geq2$, and $w_{3}=1$ if $t=3$. Define
\begin{align*}
    \lambda(\vec{x})(u)=\frac{x_{3}-x_{2}}{x_{3}-x_{1}}\frac{u-x_{1}}{u-x_{2}}.
\end{align*}
In particular,
\begin{align*}
    \lambda(\vec{x})(x_{1})=0,\quad\lambda(\vec{x})(x_{2})=\infty,\quad\lambda(\vec{x})(x_{3})=1.
\end{align*}
\par
\textbf{Case 2.} $\Sigma=\mathcal{H}$ and $t=0$. We can assume without loss of generality that $z_{1}=0$ if $s\geq1$, $z_{2}=\infty$ if $s\geq2$, and $z_{3}=1$ if $s=3$. Define
\begin{align*}
    \lambda(\vec{x})(u)=\frac{x_{3}-x_{2}}{x_{3}-x_{1}}\frac{u-x_{1}}{u-x_{2}}.
\end{align*}
In particular,
\begin{align*}
    \lambda(\vec{x})(x_{1})=0,\quad\lambda(\vec{x})(x_{2})=\infty,\quad\lambda(\vec{x})(x_{3})=1.
\end{align*}
\par
\textbf{Case 3.} $\Sigma=\mathcal{H}$ and $t=1$. We can assume without loss of generality that $w_{1}=i$ and $z_{1}=0$ if $s=1$. Define
\begin{align*}
    \lambda(\vec{x}^{\,b},\vec{x}^{\,i})(u)=\Im x^{i}_{1}\frac{u-x^{b}_{1}}{(x^{b}_{1}-\Re x^{i}_{1})u+(\Im x^{i}_{1})^{2}+(\Re x^{i}_{1})^{2}-x^{b}_{1}\Re x^{i}_{1}}.
\end{align*}
In particular,
\begin{align*}
    \lambda(\vec{x}^{\,b},\vec{x}^{\,i})(x^{b}_{1})=0,\quad\lambda(\vec{x}^{\,b},\vec{x}^{\,i})(x^{i}_{1})=i.
\end{align*}
\par
Define $h:\Omega\rightarrow\widehat{\Omega}$ by $h(\vec{x})=\lambda(\vec{x})(\vec{x})$, then the following diagram commutes:
\begin{center}
\begin{tikzcd}
    \widehat{\Omega}\arrow{d}{q}&\Omega\arrow{l}{h}\arrow{d}{g}\\
    \widehat{\Omega}/\widehat{G}&\Omega/G\arrow{l}{\psi^{-1}}.
\end{tikzcd}
\end{center}
\par
Locally, in a small open $U\subset\Omega/G$, we have a section $\sigma:U\rightarrow\Omega$. Then $g\circ\sigma=\id$, so that
\begin{align*}
    \psi^{-1}|_{U}=\psi^{-1}\circ g\circ\sigma=q\circ h\circ\sigma,
\end{align*}
thus $\psi^{-1}$ is smooth.
\end{proof}
\begin{lem}\label{diagobj}
Let $\mathcal{N}\subset\M$ be a codimension $d+2e$ stratum (allowing codimension $0$), which consists of those open stable maps with a given domain $\Sigma$ having $d+1$ disk components and $e$ sphere components, and with a given allocation of the marked points to the components of $\Sigma$ as well as a corresponding partition of $\beta$. Assume that both $\chi_{0}|_{\mathcal{N}}\neq\infty$ and $\chi_{m}|_{\mathcal{N}}\neq\infty$. Then there exists a commutative diagram
\begin{center}
\begin{tikzcd}
    &\mathcal{N}\arrow{d}{\pi}\arrow{ddrr}{\chi_{0,m}}&&\\
    Y\arrow{r}{\tau}\arrow{d}{p}&Z\arrow{drr}{\rho}&&\\
    W_{0}\times W_{m}\arrow{rrr}{r_{0}\times r_{m}}&&&\Cinf^{2}
\end{tikzcd}
\end{center}
with $Y,\,Z,\,W_{0},\,W_{m}$ smooth manifolds and $\pi,\tau,p$ surjective submersions. Write $p=(p_{0},p_{m})$, $\rho=(\rho_{0},\rho_{m})$, and $r=r_{0}\times r_{m}$.
\end{lem}
\begin{proof}
Choose a component $\nu_{0}$ of $\Sigma$ such that $(\nu_{0})_{\mathbb{C}}$ is a component of $\Sigma_{\mathbb{C}}$ as in Proposition \ref{root} for $z_{0},w_{2},w_{1}$, and similarly choose $\nu_{m}$ for $z_{0},w_{1},z_{m}$. Note that $\nu_{m}$ has to be a disk component. In the notation of section \ref{gop}, we have $w_{2}^{\nu_{0}}\notin\partial\nu_{0}$, since otherwise $\chi_{0}|_{\mathcal{N}}=\infty$. Similarly, $w_{1}^{\nu_{m}}\notin\partial\nu_{m}$, since otherwise $\chi_{m}|_{\mathcal{N}}=\infty$.
\par
To define the objects in the diagram, we consider two cases.
\par
\textbf{Case 1.} $\nu_{0}=\nu_{m}=\nu$. Fix some $z\in\partial\nu$, $w\in\mathring{\nu}$. Applying lemmas \ref{quot} and \ref{fixquot} for $\nu$, $k=2$, $l=2$, $s=1$, $t=1$, $z$, $w$ gives the smooth manifolds $Y=\widehat{\Omega}$, $Z=\Omega/G$ and the surjective submersion $\tau=\psi\circ q$.
\par
Define
\begin{align*}
    W_{0}=\mathring{\nu}\setminus\{w\},\quad W_{m}=\partial\nu\setminus\{z\},
\end{align*}
and let $p=(p_{0},p_{m})$ be the projection to the fourth and second components, which is a surjective submersion.
\par
Define also
\begin{align*}
    \pi([u,\vec{z},\vec{w}])=[z_{0}^{\nu},z_{m}^{\nu},w_{1}^{\nu},w_{2}^{\nu}].
\end{align*}
The map $\pi$ is locally a projection from a direct product, and is thus a surjective submersion.
\par
Finally, define
\begin{align*}
    \rho([\vec{x}])&=(\rho_{0}([\vec{x}]),\rho_{m}([\vec{x}]))=(((x_{1},x_{3},\bar{x}_{4},x_{4}),(x_{1},x_{2},\bar{x}_{3},x_{3}))\\
    r(x_{4},x_{2})&=(r_{0}(x_{4}),r_{m}(x_{2}))=((z,w,\bar{x_{4}},x_{4}),(z,x_{2},\bar{w},w)).
\end{align*}
Clearly the diagram commutes.
\par
\textbf{Case 2.} $\nu_{0}\neq\nu_{m}$. Fix some $z\in\partial\nu_{m}$, $w\in\mathring{\nu}_{m}$. Applying lemmas \ref{quot} and \ref{fixquot} for $\nu_{m}$, $k=2$, $l=1$, $s=1$, $t=1$, $z$, $w$ gives smooth manifolds $Y_{m}=\widehat{\Omega}$, $Z_{m}=\Omega/G$ and a surjective submersion $\tau_{m}=\psi\circ q$.
\par
\textbf{Case 2a.} $\nu_{0}=\mathcal{H}$. Fix some $z'\in\partial\nu_{0}$, $w'\in\mathring{\nu}_{0}$. Applying lemmas \ref{quot} and \ref{fixquot} for $\nu_{0}$, $k=1$, $l=2$, $s=1$, $t=1$, $z'$, $w'$ gives smooth manifolds $Y_{0}=\widehat{\Omega}$, $Z_{0}=\Omega/G$ and a surjective submersion $\tau_{0}=\psi\circ q$.
\par
\textbf{Case 2b.} $\nu_{0}=\Cinf$. Fix some distinct $z',w'\in\nu_{0}$. Applying lemmas \ref{quot} and \ref{fixquot} for $\nu_{0}$, $l=3$, $t=2$, $z'$, $w'$ gives smooth manifolds $Y_{0}=\widehat{\Omega}$, $Z_{0}=\Omega/G$ and a surjective submersion $\tau_{0}=\psi\circ q$.
\par
In both cases 2a and 2b, define
\begin{align*}
    Y=Y_{0}\times Y_{m},&\quad Z=Z_{0}\times Z_{m},
\end{align*}
so that $\tau=\tau_{0}\times\tau_{m}$ is a surjective submersion.
\par
Let
\begin{align*}
    W_{0}=\mathring{\nu}_{0}\setminus\{z',w'\},\quad W_{m}=\partial\nu_{m}\setminus\{z\},
\end{align*}
and let $p=(p_{0},p_{m})$ be the projection to the sixth and second components, which is a surejctive submersion.
\par
Define also
\begin{align*}
    \pi([u,\vec{z},\vec{w}])=[z_{0}^{\nu_{m}},z_{m}^{\nu_{m}},w_{1}^{\nu_{m}},z_{0}^{\nu_{0}},w_{1}^{\nu_{0}},w_{2}^{\nu_{0}}].
\end{align*}
Again, the map $\pi$ is locally a projection from a direct product, and is thus a surjective submersion.
\par
Finally, define
\begin{align*}
    \rho([\vec{x}])&=(\rho_{0}([\vec{x}]),\rho_{m}([\vec{x}]))=((x_{4},x_{5},\bar{x}_{6},x_{6}),(x_{1},x_{2},\bar{x}_{3},x_{3}))\\
    r(x_{6},x_{2})&=(r_{0}(x_{6}),r_{m}(x_{2}))=((z',w',\bar{x_{6}},x_{6}),(z,x_{2},\bar{w},w)).
\end{align*}
Clearly the diagram commutes.
\end{proof}
\begin{rmk}
When applying Lemma \ref{diagobj}, we only use the commutativity of the diagram and the listed properties, without the need to specifically describe the objects and the maps.
\end{rmk}
\begin{prop}\label{transgg}
The inclusion $I^{2}\hookrightarrow\Cinf^{2}$ is transversal to $\chi_{0,m}$.
\end{prop}
\begin{proof}
Let $\mathcal{N}\subset\M$ be a codimension $d+2e$ stratum (allowing codimension $0$) that consists of those open stable maps with a given domain $\Sigma$ having $d+1$ disk components and $e$ sphere components, and with a given allocation of the marked points to the components of $\Sigma$ as well as a corresponding partition of $\beta$. It is enough to show transversality in each such $\mathcal{N}$ with $\mathcal{N}\cap\chi_{0,m}^{-1}(I^{2})\neq\emptyset$, which implies that $\chi_{0}|_{\mathcal{N}}\neq\infty$ and $\chi_{m}|_{\mathcal{N}}\neq\infty$. Fix such $\mathcal{N}$ and apply Lemma \ref{diagobj}.
\par
By abuse of notation, denote by $I:I\hookrightarrow\Cinf$ and by $I^{2}:I^{2}\hookrightarrow\Cinf^{2}$ the inclusions. Lemma \ref{transg} states that $\chi_{i}\pitchfork I$ for $i=0,m$. Applying Lemma \ref{transcomp} twice gives $\rho_{i}\pitchfork I$ and then $r_{i}\pitchfork I$. Lemma \ref{transprod} now gives $r\pitchfork I^{2}$. Again, applying Lemma \ref{transcomp} twice gives $\rho\pitchfork I^{2}$ and then $\chi_{0,m}\pitchfork I^{2}$.
\end{proof}
We can thus define
\begin{align*}
    \Mgg=I^{2}\times_{\Cinf^{2}}\M,
\end{align*}
where the fiber product is taken with respect to the inclusion $I^{2}\hookrightarrow\Cinf^{2}$ and to $\chi_{0,m}$. By proposition \ref{transgg} it is a smooth orbifold with corners. As before, denote by $\evb$ and $\evi$ the evaluation maps, and assume that $\evbz$ is a proper submersion (see Section \ref{reg}). Define the corresponding operators
\begin{align*}
    \qgg:E^{\otimes k}\otimes C^{\otimes l}\rightarrow C
\end{align*}
by
\begin{align*}
    \qgg(\alpha;\eta)=(-1)^{\varepsilon(\alpha)}(\evbz)_{*}\Big(\bigwedge_{j=1}^{l}(\evi)^{*}\eta_{j}\wedge\bigwedge_{j=1}^{k}(\evb)^{*}\alpha_{j}\Big),
\end{align*}
and set
\begin{align*}
    \mathfrak{q}_{k,l;0,m}(\alpha;\eta)=\sum_{\beta\in\Pi}T^{\beta}\qgg.
\end{align*}
\begin{rmk}
Note that the operators $\qgg$ and $\mathfrak{q}_{k,l;0,m}$ are different from the ones denoted in \cite{RQH} by $\mathfrak{q}_{k,l;0,3}^{\beta}$ and $\mathfrak{q}_{k,l;0,3}$.
\end{rmk}
\par
We then have the following.
\begin{prop}[Unit]\label{ggunit}
Assume $\alpha_{i}=c\cdot1$ for some $c\in R$ and $1\leq i\leq k$, $i\neq m$. Then $\qgg(\alpha;\eta)=0$.
\end{prop}
\begin{proof}
The proof is verbatim the same as that of Proposition \ref{horunit}, with the moduli spaces and operators with subscript $\perp$ replaced by those with subscript $0,m$.
\end{proof}
Set
\begin{align*}
    \iota_{5}(\alpha,\eta;i,J_{1})&=|\eta^{1}|+|\eta^{2}|(|\alpha^{1}|+i)+\sigma_{J_{1},J_{2}}^{\eta},\\
    \iota_{6}(\alpha,\eta;J_{1})&=(|\eta^{1}|+1)(|\eta^{2}|+1)+n,
\end{align*}
and as in section \ref{gop}, let
\begin{align*}
    m'(m,i,k_{2})=
    \begin{cases}
    m,&1\leq m<i,\\
    i,&i\leq m<i+k_{2},\\
    m-k_{2}+1,&i+k_{2}\leq m\leq k.
    \end{cases}
\end{align*}
\begin{prop}[Structure equation for $\Mgg$]\label{ggstruc}
We have
\begin{align*}
    0&=-\mathfrak{q}_{k,l;0,m}(\alpha;d\eta)\\
    &+\sum_{\substack{k_{1}+k_{2}=k+1\\1\leq i\leq k_{1}\\J_{1}\sqcup J_{2}=[l]\\1,2\in J_{1}}}(-1)^{\iota(\alpha,\eta;i,J_{1})}\mathfrak{q}_{k_{1},l_{1};0,m'(m,i,k_{2})}(\alpha^{1}\otimes \mathfrak{q}_{k_{2},l_{2}}(\alpha^{2};\eta^{2})\otimes\alpha^{3};\eta^{1})\\
    &+\sum_{\substack{k_{1}+k_{2}=k+1\\m-k_{2}<i\leq m\\J_{1}\sqcup J_{2}=[l]\\1,2\in J_{2}}}(-1)^{\iota(\alpha,\eta;i,J_{1})}\mathfrak{q}_{k_{1},l_{1}}(\alpha^{1}\otimes \mathfrak{q}_{k_{2},l_{2};0,m-i+1}(\alpha^{2};\eta^{2})\otimes\alpha^{3};\eta^{1})\\
    &+\sum_{\substack{k_{1}+k_{2}=k+1\\m-k_{2}<i\leq m\\J_{1}\sqcup J_{2}=[l]\\2\in J_{1},1\in J_{2}}}(-1)^{\iota_{5}(\alpha,\eta;i,J_{1})}\mathfrak{q}_{k_{1},l_{1};i}(\alpha^{1}\otimes \mathfrak{q}_{k_{2},l_{2};m-i+1}(\alpha^{2};\eta^{2})\otimes\alpha^{3};\eta^{1})\\
    &+\sum_{\substack{J_{1}\sqcup J_{2}=[l]\\1,2\in J_{2}}}(-1)^{\iota_{6}(\alpha,\eta;J_{1})}\mathfrak{q}_{k,l_{1}+1;m}(\alpha;\mathfrak{q}_{\emptyset,l_{2}}(\eta^{2})\otimes\eta^{1}).
\end{align*}
\end{prop}
\begin{figure}[H]
    \centering
    \includegraphics[scale=0.2]{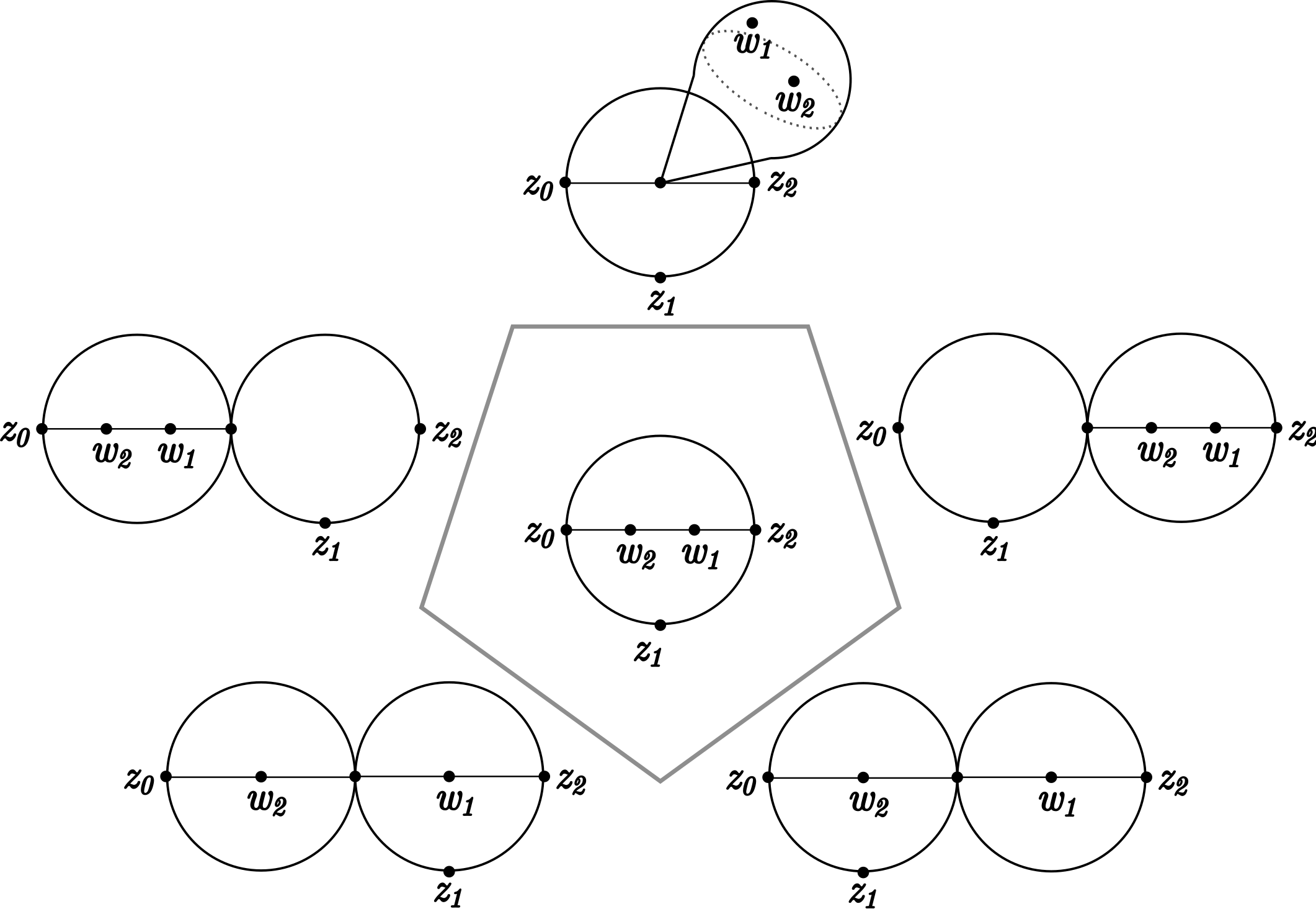}
    \caption{Boundary components of $\mathcal{M}_{3,2;0,2}(\beta_{0})$}
\end{figure}
The proof of Proposition \ref{ggstruc} is given in the next two sections. Defining deformed operators as in Section \ref{defop},
\begin{align*}
    \mathfrak{q}_{k,l;0,m}^{\gamma,b}(\alpha;\eta)&=\sum_{t,s_{0},...,s_{k}\geq0}\mathfrak{q}_{k+s_{0}+...+s_{k},l+t;0,m+s_{0}+...+s_{m-1}}(\alpha^{b;s_{0},...,s_{k}};\eta^{\gamma;t}),
\end{align*}
we get the following.
\begin{prop}
Propositions \ref{ggunit} and \ref{ggstruc} hold for the deformed operators as well.
\end{prop}

\subsection{Description of the boundary}\label{boundary}
We turn to proving Proposition \ref{ggstruc}. We begin by describing the boundary of the moduli space $\Mgg=I^{2}\times_{\Cinf^{2}}\M$. By Proposition \ref{fiberbdry}, we have
\begin{align*}
    \partial\Mgg=\partial I^{2}\times_{\Cinf^{2}}\M+I^{2}\times_{\Cinf^{2}}\partial\M.
\end{align*}
\par
We begin by analysing the first summand. Write $I^{2}=I_{0}\times I_{m}$, where $I_{j}$ represents the copy of $I$ in the range of $\chi_{j}$. We obtain
\begin{align*}
    \partial I^{2}=\partial I_{0}\times I_{m}-I_{0}\times\partial I_{m}.
\end{align*}
Therefore, by Proposition \ref{fiber},
\begin{align*}
    \partial I^{2}\times_{\Cinf^{2}}\M
    &=(\partial I_{0}\times I_{m})\times_{\Cinf^{2}}\M-(I_{0}\times\partial I_{m})\times_{\Cinf^{2}}\M\\
    &=(I_{m}\times\partial I_{0})\times_{\Cinf^{2}}\M-(I_{0}\times\partial I_{m})\times_{\Cinf^{2}}\M\\
    &=I_{m}\times_{\Cinf}(\partial I_{0}\times_{\Cinf}\M)-I_{0}\times_{\Cinf}(\partial I_{m}\times_{\Cinf}\M)\\
    &=I_{m}\times_{\Cinf}(\chi_{0}^{-1}(1)-\chi_{0}^{-1}(0))-I_{0}\times_{\Cinf}(\chi_{m}^{-1}(1)-\chi_{m}^{-1}(0)).
\end{align*}
\par
It is shown in \cite{prep} that $\chi_{0}^{-1}(0)=\chi_{m}^{-1}(0)=\chi_{m}^{-1}(1)=\emptyset$, and that $\chi_{0}^{-1}(1)$ corresponds to the marked points $w_{1},w_{2}$ bubbling off to a sphere. More precisely,
\begin{align*}
    \chi_{0}^{-1}(1)\cong\sum_{\substack{\beta_{1}+\varpi(\beta_{2})=\beta\\J_{1}\sqcup J_{2}=[l]\\1,2\in J_{2}}}(-1)^{w_{\mathfrak{s}}(\beta_{2})}\mathcal{M}_{k+1,1+|J_{1}|}(\beta_{1})\times_{X}\mathcal{M}_{1+|J_{2}|}(\beta_{2})
\end{align*}
where the fiber product is taken with respect to $evi^{\beta_{1}}_{1},ev^{\beta_{2}}_{0}$. Combining the above, we get
\begin{align*}
    \partial I^{2}\times_{\Cinf^{2}}\M&=I_{m}\times_{\Cinf}\chi_{0}^{-1}(1)\\
    &\cong\sum_{\substack{\beta_{1}+\varpi(\beta_{2})=\beta\\J_{1}\sqcup J_{2}=[l]\\1,2\in J_{2}}}(-1)^{w_{\mathfrak{s}}(\beta_{2})}\mathcal{M}_{k+1,1+|J_{1}|;m}(\beta_{1})\times_{X}\mathcal{M}_{1+|J_{2}|}(\beta_{2}).
\end{align*}
\par
We turn to analysing the second summand. By \cite{Ainf}*{Proposition 2.8}, we have
\begin{align*}
    \partial\M\cong\sum_{\substack{\beta_{1}+\beta_{2}=\beta\\k_{1}+k_{2}=k+1\\J_{1}\sqcup J_{2}=[l]\\1\leq i\leq k_{1}\\(\beta_{1},k_{1},l_{1})\neq(\beta_{0},1,0)\\(\beta_{2},k_{1},l_{1})\notin\{(\beta_{0},1,0),(\beta_{0},0,0)\}}}(-1)^{\delta_{1}}\mathcal{M}_{k_{1}+1,|J_{1}|}(\beta_{1})\times_{L}\mathcal{M}_{k_{2}+1,|J_{2}|}(\beta_{2})
\end{align*}
where $\delta_{1}=k_{2}(k_{1}+i)+i+n$ and the fiber product is taken with respect to $evb^{\beta_{1}}_{i},evb^{\beta_{2}}_{0}$. This fiber product represents the marked points $\{z_{j}\}_{j=i}^{i+k_{2}-1}$ and $\{w_{j}\}_{j\in J_{2}}$ bubbling off to a disk. Thus
\begin{align*}
    (I_{0}\times I_{m})\times_{\Cinf^{2}}\partial\M
    &=\sum_{\substack{\beta_{1}+\beta_{2}=\beta\\k_{1}+k_{2}=k+1\\J_{1}\sqcup J_{2}=[l]\\1\leq i\leq k_{1}\\(\beta_{1},k_{1},l_{1})\neq(\beta_{0},1,0)\\(\beta_{2},k_{1},l_{1})\notin\{(\beta_{0},1,0),(\beta_{0},0,0)\}}}(-1)^{\delta_{1}}\mathcal{B}_{\beta_{1},k_{1},J_{1},i}
\end{align*}
where
\begin{align*}
    \mathcal{B}_{\beta_{1},k_{1},J_{1},i}\cong(I_{0}\times I_{m})\times_{\Cinf^{2}}(\mathcal{M}_{k_{1}+1,|J_{1}|}(\beta_{1})\times_{L}\mathcal{M}_{k_{2}+1,|J_{2}|}(\beta_{2})).
\end{align*}
\par
Fix $\beta_{1},k_{1},J_{1},i$. We describe $\mathcal{B}=\mathcal{B}_{\beta_{1},k_{1},J_{1},i}$ explicitly by case, assuming it is nonempty. Let
\begin{align*}
    m'=m'(m,i,k_{2})=
    \begin{cases}
    m,&1\leq m<i,\\
    i,&i\leq m<i+k_{2},\\
    m-k_{2}+1,&i+k_{2}\leq m\leq k,
    \end{cases}
\end{align*}
as in Section \ref{gop}, and let
\begin{align*}
    \chi_{j}^{i}=\chi_{j}:\mathcal{M}_{k_{i},|J_{i}|}(\beta_{i})\rightarrow\Cinf.
\end{align*}
Denote by $p_{i}:\mathcal{B}\rightarrow\mathcal{M}_{k_{i}+1,|J_{i}|}(\beta_{i})$ the projections, and let $I_{j}^{i}$ represents the copy of $I$ in the range of $\chi_{j}^{i}$.
\par
\textbf{Case 1.} If $1,2\in J_{1}$, then all of the geodesic constraints are confined to the component corresponding to $\mathcal{M}_{k_{1}+1,|J_{1}|}(\beta_{1})$. We get $\chi_{0,m}|_{\mathcal{B}}=(\chi_{0}^{1},\chi_{m'}^{1})\circ p_{1}$. Thus, by Proposition \ref{fiber},
\begin{align*}
    \mathcal{B}&\cong(I_{0}^{1}\times I_{m'}^{1})\times_{\Cinf^{2}}(\mathcal{M}_{k_{1}+1,|J_{1}|}(\beta_{1})\times_{L}\mathcal{M}_{k_{2}+1,|J_{2}|}(\beta_{2}))\\
    &=((I_{0}^{1}\times I_{m'}^{1})\times_{\Cinf^{2}}\mathcal{M}_{k_{1}+1,|J_{1}|}(\beta_{1}))\times_{L}\mathcal{M}_{k_{2}+1,|J_{2}|}(\beta_{2})\\
    &=\mathcal{M}_{k_{1}+1,|J_{1}|;0,m'}(\beta_{1})\times_{L}\mathcal{M}_{k_{2}+1,|J_{2}|}(\beta_{2}).
\end{align*}
\par
\textbf{Case 2.} If $1,2\in J_{2}$, then the geodesic condition implies that $z_{m}$ lies in the component corresponding to $\mathcal{M}_{k_{2}+1,|J_{2}|}(\beta_{2})$, that is $i\leq m<i+k_{2}$. We get $\chi_{0,m}|_{\mathcal{B}}=(\chi_{0}^{2},\chi_{m-i+1}^{2})\circ p_{1}$. Thus, by Proposition \ref{fiber},
\begin{align*}
    \mathcal{B}&\cong(I_{0}^{2}\times I_{m-i+1}^{2})\times_{\Cinf^{2}}(\mathcal{M}_{k_{1}+1,|J_{1}|}(\beta_{1})\times_{L}\mathcal{M}_{k_{2}+1,|J_{2}|}(\beta_{2}))\\
    &=(\mathcal{M}_{k_{1}+1,|J_{1}|}(\beta_{1})\times_{L}\mathcal{M}_{k_{2}+1,|J_{2}|}(\beta_{2}))\times_{\Cinf^{2}}(I_{0}^{2}\times I_{m-i+1}^{2})\\
    &=\mathcal{M}_{k_{1}+1,|J_{1}|}(\beta_{1})\times_{L}(\mathcal{M}_{k_{2}+1,|J_{2}|}(\beta_{2})\times_{\Cinf^{2}}(I_{0}^{2}\times I_{m-i+1}^{2}))\\
    &=\mathcal{M}_{k_{1}+1,|J_{1}|}(\beta_{1})\times_{L}((I_{0}^{2}\times I_{m-i+1}^{2})\times_{\Cinf^{2}}\mathcal{M}_{k_{2}+1,|J_{2}|}(\beta_{2}))\\
    &=\mathcal{M}_{k_{1}+1,|J_{1}|}(\beta_{1})\times_{L}\mathcal{M}_{k_{2}+1,|J_{2}|;0,m-i+1}(\beta_{2}).
\end{align*}
\par
\textbf{Case 3.} If $2\in J_{1},1\in J_{2}$, we get
$\chi_{0,m}|_{\mathcal{B}}=(\chi_{i}^{1}\times\chi_{m-i+1}^{2})\circ (p_{1},p_{2})$. Thus, by propositions \ref{fiber} and \ref{dim},
\begin{align*}
    \mathcal{B}&\cong(I_{i}^{1}\times I_{m-i+1}^{2})\times_{\Cinf^{2}}(\mathcal{M}_{k_{1}+1,|J_{1}|}(\beta_{1})\times_{L}\mathcal{M}_{k_{2}+1,|J_{2}|}(\beta_{2}))\\
    &=(-1)^{k_{1}k_{2}}I_{i}^{1}\times_{\Cinf}(I_{m-i+1}^{2}\times_{\Cinf}(\mathcal{M}_{k_{2}+1,|J_{2}|}(\beta_{2})\times_{L}\mathcal{M}_{k_{1}+1,|J_{1}|}(\beta_{1})))\\
    &=(-1)^{k_{1}k_{2}+k_{1}(k_{2}+1)}(I_{i}^{1}\times_{\Cinf}\mathcal{M}_{k_{1}+1,|J_{1}|}(\beta_{1}))\times_{L}(I_{m-i+1}^{2}\times_{\Cinf}\mathcal{M}_{k_{2}+1,|J_{2}|}(\beta_{2}))\\
    &=(-1)^{k_{1}}\mathcal{M}_{k_{1}+1,|J_{1}|;i}(\beta_{1})\times_{L}\mathcal{M}_{k_{2}+1,|J_{2}|;m-i+1}(\beta_{2}).
\end{align*}
\par
Note that if $2\in J_{2}$ then the geodesic condition implies that $1\in J_{2}$ as well, hence these are all the possible cases. Combining the above, we get
\begin{align*}
    &(-1)^{\delta_{2}}\mathcal{B}\cong\\
    &\cong\begin{cases}
    \mathcal{M}_{k_{1}+1,|J_{1}|;0,m'}(\beta_{1})\times_{L}\mathcal{M}_{k_{2}+1,|J_{2}|}(\beta_{2}),&1,2\in J_{1},\\
    \mathcal{M}_{k_{1}+1,|J_{1}|}(\beta_{1})\times_{L}\mathcal{M}_{k_{2}+1,|J_{2}|;0,m'}(\beta_{2}),&1,2\in J_{2},i\leq m<i+k_{2},\\
    \mathcal{M}_{k_{1}+1,|J_{1}|;i}(\beta_{1})\times_{L}\mathcal{M}_{k_{2}+1,|J_{2}|;m'}(\beta_{2}),&2\in J_{1},1\in J_{2},i\leq m<i+k_{2},\\
    \emptyset,&otherwise.
    \end{cases}
\end{align*}
where
\begin{align*}
    \delta_{2}=
    \begin{cases}
    \delta_{1}+k_{1},&2\in J_{1},1\in J_{2},\\\delta_{1},&otherwise,
    \end{cases}
\end{align*}

\subsection{Applying Stokes' Theorem}\label{computegg}

We are now ready to derive the structure equation. Set
\begin{align*}
    \xi=\bigwedge_{j=1}^{l}(\evi)^{*}\eta_{j}\wedge\bigwedge_{j=1}^{k}(\evb)^{*}\alpha_{j},
\end{align*}
and note that
\begin{align*}
    \dim\Mgg=\dim\M-2\equiv n+k\pmod{2}.
\end{align*}
Applying Stokes' theorem (Proposition \ref{Stokes}), we obtain
\begin{align*}
    d((\evbz)_{*}\xi)=(\evbz)_{*}(d\xi)+(-1)^{n+k+|\alpha|+|\eta|}(\evbz|_{\partial\Mgg})_{*}\xi.
\end{align*}
Let us describe each of the above terms using the $\mathfrak{q}$ operators. We clearly have
\begin{align*}
    d((\evbz)_{*}\xi)&=(-1)^{\varepsilon(\alpha)}\mathfrak{q}_{1,0}^{\beta_{0}}(\qgg(\alpha;\eta))\\
    &=(-1)^{\iota(\alpha,\eta;1,\emptyset)+\varepsilon(\alpha)}\mathfrak{q}_{1,0}^{\beta_{0}}(\qgg(\alpha;\eta)),
\end{align*}
in addition to
\begin{align*}
    (\evbz)_{*}(d\xi)
    &=(-1)^{\varepsilon(\alpha)}\qgg(\alpha;d\eta)\\
    &+\sum_{i=1}^{k}(-1)^{|\eta|+|\alpha^{1}|+\varepsilon(\alpha^{1}\otimes d\alpha_{i}\otimes\alpha^{3})}\qgg(\alpha^{1}\otimes d\alpha_{i}\otimes\alpha^{2};\eta)\\
    &=(-1)^{\varepsilon(\alpha)}\qgg(\alpha;d\eta)\\
    &+\sum_{i=1}^{k}(-1)^{\varepsilon(\alpha)+|\eta|+|\alpha^{1}|+i}\qgg(\alpha^{1}\otimes\mathfrak{q}_{1,0}^{\beta_{0}}(\alpha_{i})\otimes\alpha^{2};\eta)\\
    &=(-1)^{\varepsilon(\alpha)}\qgg(\alpha;d\eta)\\
    &+\sum_{i=1}^{k}(-1)^{\varepsilon(\alpha)+\iota(\alpha,\eta;i,J)+1}\qgg(\alpha^{1}\otimes\mathfrak{q}_{1,0}^{\beta_{0}}(\alpha_{i})\otimes\alpha^{2};\eta).
\end{align*}
For the last term, note that it is enough to calculate $(\evbz|_{\mathcal{B}})_{*}\xi$ for each boundary component $\mathcal{B}$ and then sum over $\mathcal{B}$.
\par
\textbf{Case 1 (boundary bubbling).} Let $\mathcal{B}$ be a boundary component of the form 
\begin{align*}
    \mathcal{B}=(-1)^{\delta_{2}}\mathcal{M}_{k_{1}+1,l_{1};g_{1}}(\beta_{1})\times_{L}\mathcal{M}_{k_{2}+1,l_{2};g_{2}}(\beta_{2}),
\end{align*}
where the set $g_{i}\in\{\emptyset,\{0\},\{1\}...,\{k_{i}\},\{0,1\},...,\{0,k_{i}\}\}$ indicates the geodesic constraints. In particular,
\begin{align*}
    \dim\mathcal{M}_{k_{i}+1,l_{i};g_{i}}=\dim\mathcal{M}_{k_{i}+1,l_{i}}-|g_{i}|,
\end{align*}
where $|g_{i}|$ stands for the cardinality of $g_{i}$. Denote by
\begin{align*}
    p_{i}:\mathcal{B}\rightarrow\mathcal{M}_{k_{i}+1,l_{i};g_{i}}
\end{align*}
the projections and let
\begin{align*}
    \xi_{1}&=\bigwedge_{j\in J_{1}}(evi_{j}^{\beta_{1}})^{*}\eta_{j}\wedge\bigwedge_{j=1}^{i-1}(evb_{j}^{\beta_{1}})^{*}\alpha_{j}\wedge\bigwedge_{j=i+1}^{k_{1}}(evb_{j}^{\beta_{1}})^{*}\alpha_{j+k_{2}-1},\\
    \xi_{2}&=\bigwedge_{j\in J_{2}}(evi_{j}^{\beta_{2}})^{*}\eta_{j}\wedge\bigwedge_{j=1}^{k_{2}}(evb_{j}^{\beta_{2}})^{*}\alpha_{j+i-1},
\end{align*}
so that
\begin{align*}
    \xi=(-1)^{(|\alpha^{2}|+|\eta^{2}|)|\alpha^{3}|+|\eta^{2}||\alpha^{1}|+\sigma_{J_{1},J_{2}}^{\eta}}p_{1}^{*}\xi_{1}\wedge p_{2}^{*}\xi_{2}.
\end{align*}
Set
\begin{align*}
    \delta_{3}=\delta_{2}+(|\alpha^{2}|+|\eta^{2}|)|\alpha^{3}|+|\eta^{2}||\alpha^{1}|+\sigma_{J_{1},J_{2}}^{\eta}.
\end{align*}
Proposition \ref{proj} gives
\begin{align*}
    (evb^{\beta_{1}}_{i})^{*}(evb^{\beta_{2}}_{0})_{*}=p_{1*}p_{2}^{*}.
\end{align*}
By propositions \ref{subcomp} and \ref{int} we obtain
\begin{align*}
    (\evbz|_{\mathcal{B}})_{*}\xi
    &=(-1)^{\delta_{3}}(evb^{\beta_{1}}_{0})_{*}p_{1*}(p_{1}^{*}\xi_{1}\wedge p_{2}^{*}\xi_{2})\\
    &=(-1)^{\delta_{3}}(evb^{\beta_{1}}_{0})_{*}(\xi_{1}\wedge p_{1*}p_{2}^{*}\xi_{2})\\
    &=(-1)^{\delta_{3}}(evb^{\beta_{1}}_{0})_{*}(\xi_{1}\wedge(evb^{\beta_{1}}_{i})^{*}(evb^{\beta_{2}}_{0})_{*}\xi_{2}).
\end{align*}
\par
Letting
\begin{align*}
    \xi_{1}^{1}&=\bigwedge_{j\in J_{1}}(evi_{j}^{\beta_{1}})^{*}\eta_{j}\wedge\bigwedge_{j=1}^{i-1}(evb_{j}^{\beta_{1}})^{*}\alpha_{j},\\
    \xi_{1}^{2}&=\bigwedge_{j=i+1}^{k_{1}}(evb_{j}^{\beta_{1}})^{*}\alpha_{j+k_{2}-1},
\end{align*}
we get
\begin{align*}
    (\evbz|_{\mathcal{B}})_{*}\xi
    &=(-1)^{\delta_{4}}(evb^{\beta_{1}}_{0})_{*}(\xi_{1}^{1}\wedge(evb^{\beta_{1}}_{i})^{*}(evb^{\beta_{2}}_{0})_{*}\xi_{2}\wedge\xi_{1}^{2}),
\end{align*}
where
\begin{align*}
    \delta_{4}=\delta_{3}+(|\alpha^{2}|+|\eta^{2}|+\rdim evb_{0}^{\beta_{2}})|\alpha^{3}|.
\end{align*}
Note that in all cases we have $g:=|g_{1}|\equiv|g_{2}|\pmod{2}$, hence
\begin{align*}
    (\evbz|_{\mathcal{B}})_{*}\xi
    &=(-1)^{\delta_{5}}\mathfrak{q}_{k_{1},l_{1};g_{1}}^{\beta_{1}}(\alpha^{1}\otimes \mathfrak{q}_{k_{2},l_{2};g_{2}}^{\beta_{2}}(\alpha^{2};\eta^{2})\otimes\alpha^{3};\eta^{1}),
\end{align*}
for
\begin{align*}
    \delta_{5}&=\delta_{4}+\varepsilon(\alpha^{1}\otimes \mathfrak{q}_{k_{2},l_{2};g_{2}}^{\beta_{2}}(\alpha^{2};\eta^{2})\otimes\alpha^{3})+\varepsilon(\alpha^{2})\\
    &+g(|\alpha|+\rdim evb_{0}^{\beta_{2}}+|\eta^{2}|+|\alpha^{2}|+k+1).
\end{align*}
We calculate
\begin{align*}
    \rdim evb_{0}^{\beta_{2}}\equiv n+k_{2}-g-n\equiv k_{2}+g\pmod{2}.
\end{align*}
\par
We also need the following version of \cite{Ainf}*{Lemma 2.9}, adapted to the geodesic case. The proof is analogous to that of the non-geodesic case.
\begin{lem}\label{signs}
We have
\begin{align*}
    \varepsilon(\alpha^{1}\otimes \mathfrak{q}^{\beta_{2}}_{k_{2},l_{2};g_{2}}(\alpha^{2};\eta^{2})\otimes\alpha^{3})+\varepsilon(\alpha^{2})&\equiv\varepsilon(\alpha)+|\alpha|+k+|\alpha^{1}|+i|\eta^{2}|\\
    &+k_{2}|\alpha^{3}|+k_{1}k_{2}+ik_{2}+gi\pmod{2}.
\end{align*}
\end{lem}
Combining the above, we get
\begin{align*}
    \delta_{1}&=k_{2}(k_{1}+i)+i+n,\\
    \implies\delta_{2}&=k_{2}(k_{1}+i)+i+n+gk_{1},\\
    \implies\delta_{3}&=(|\alpha^{2}|+|\eta^{2}|)|\alpha^{3}|+|\eta^{2}||\alpha^{1}|+\sigma_{J_{1},J_{2}}^{\eta}
    \\&+k_{2}(k_{1}+i)+i+n+gk_{1},\\
    \implies\delta_{4}&=(|\alpha^{2}|+|\eta^{2}|)|\alpha^{3}|+|\eta^{2}||\alpha^{1}|+\sigma_{J_{1},J_{2}}^{\eta}
    \\&+k_{2}(k_{1}+i)+i+n+gk_{1},\\
    &+(|\alpha^{2}|+|\eta^{2}|+k_{2}+g)|\alpha^{3}|,
\end{align*}
therefore
\begin{align*}
    \delta_{4}&\equiv|\eta^{2}||\alpha^{1}|+(k_{2}+g)|\alpha^{3}|+\sigma_{J_{1},J_{2}}^{\eta}
    \\&+k_{2}(k_{1}+i)+i+n+gk_{1}\pmod{2}.
\end{align*}
Applying Lemma \ref{signs} yields
\begin{align*}
    \delta_{5}&\equiv
    |\eta^{2}||\alpha^{1}|+(k_{2}+g)|\alpha^{3}|+\sigma_{J_{1},J_{2}}^{\eta}
    \\&+k_{2}(k_{1}+i)+i+n+gk_{1}\\
    &+\varepsilon(\alpha)+|\alpha|+k+|\alpha^{1}|+i|\eta^{2}|\\
    &+k_{2}|\alpha^{3}|+k_{1}k_{2}+ik_{2}+gi\\
    &+g(|\alpha|+k_{2}+g+|\eta^{2}|+|\alpha^{2}|+k+1)\\
    &\equiv\varepsilon(\alpha)+\sigma_{J_{1},J_{2}}^{\eta}+(|\eta^{2}|+1)(|\alpha^{1}|+i)+n+|\alpha|+k\\
    &+g(|\eta^{2}|+|\alpha^{1}|+1+i)\\
    &\equiv\iota(\alpha,\eta;i,J_{1})+\varepsilon(\alpha)+n+|\alpha|+|\eta|+k+1\\
    &+g(|\eta^{2}|+|\alpha^{1}|+i+1)\pmod{2}.
\end{align*}
\par
\textbf{Case 2 (interior bubbling).} We now turn to boundary components of the form
\begin{align*}
    \mathcal{B}=(-1)^{w_{\mathfrak{s}}(\beta_{2})}\mathcal{M}_{k+1,1+|J_{1}|;m}(\beta_{1})\times_{X}\mathcal{M}_{1+|J_{2}|}(\beta_{2}).
\end{align*}
 Let
\begin{align*}
    \xi_{1}&=\bigwedge_{j\in J_{1}}(evi_{j}^{\beta_{1}})^{*}\eta_{j}\wedge\bigwedge_{j=1}^{k}(evb_{j}^{\beta_{1}})^{*}\alpha_{j+i-1},\\
    \xi_{2}&=\bigwedge_{j\in J_{2}}(ev_{j}^{\beta_{2}})^{*}\eta_{j},
\end{align*}
so that
\begin{align*}
    \xi=(-1)^{|\eta^{2}||\alpha|}p_{1}^{*}\xi_{1}\wedge p_{2}^{*}\xi_{2}.
\end{align*}
Denote by $p_{1}:\mathcal{B}\rightarrow\mathcal{M}_{k+1,1+|J_{1}|;m}(\beta_{1})$, $p_{2}:\mathcal{B}\rightarrow\mathcal{M}_{1+|J_{2}|}(\beta_{2})$ the projections, so that Proposition \ref{proj} gives
\begin{align*}
    (evi^{\beta_{1}}_{1})^{*}(ev^{\beta_{2}}_{0})_{*}=p_{1*}p_{2}^{*}.
\end{align*}
By propositions \ref{subcomp} and \ref{int} we obtain
\begin{align*}
    (\evbz|_{\mathcal{B}})_{*}\xi
    &=(-1)^{w_{\mathfrak{s}}(\beta_{2})+|\eta^{2}||\alpha|}(evb^{\beta_{1}}_{0})_{*}p_{1*}(p_{1}^{*}\xi_{1}\wedge p_{2}^{*}\xi_{2})\\
    &=(-1)^{w_{\mathfrak{s}}(\beta_{2})+|\eta^{2}||\alpha|}(evb^{\beta_{1}}_{0})_{*}(\xi_{1}\wedge p_{1*}p_{2}^{*}\xi_{2})\\
    &=(-1)^{w_{\mathfrak{s}}(\beta_{2})+|\eta^{2}||\alpha|}(evb^{\beta_{1}}_{0})_{*}(\xi_{1}\wedge(evi^{\beta_{1}}_{1})^{*}(ev^{\beta_{2}}_{0})_{*}\xi_{2})\\
    &=(-1)^{w_{\mathfrak{s}}(\beta_{2})+|\eta^{2}||\alpha|}(evb^{\beta_{1}}_{0})_{*}(\xi_{1}\wedge(evi^{\beta_{1}}_{1})^{*}(ev^{\beta_{2}}_{0})_{*}\xi_{2})\\
    &=(-1)^{w_{\mathfrak{s}}(\beta_{2})+|\eta^{2}||\eta^{1}|}(evb^{\beta_{1}}_{0})_{*}((evi^{\beta_{1}}_{1})^{*}(ev^{\beta_{2}}_{0})_{*}\xi_{2}\wedge\xi_{1})\\
    &=(-1)^{\varepsilon(\alpha)+|\alpha|+k+|\eta^{1}||\eta^{2}|}\mathfrak{q}_{k,l_{1}+1;m}^{\beta_{1}}(\alpha;\mathfrak{q}_{\emptyset,2}^{\beta_{2}}(\eta^{2})\otimes\eta^{1}).
\end{align*}
\par
We conclude that
\begin{align*}
    0&=(-1)^{\varepsilon(\alpha)}\big(d((\evbz)_{*}\xi)-(\evbz)_{*}(d\xi)\\
    &+(-1)^{n+k+|\alpha|+|\eta|+1}(\evbz|_{\partial\Mgg})_{*}\xi\big)\\
    &=(-1)^{\iota(\alpha,\eta;1,\emptyset)}\mathfrak{q}_{1,0}^{\beta_{0}}(\qgg(\alpha;\eta))\\
    &-\qgg(\alpha;d\eta)\\
    &+\sum_{i=1}^{k}(-1)^{\iota(\alpha,\eta;i,J)}\qgg(\alpha^{1}\otimes\mathfrak{q}_{1,0}^{\beta_{0}}(\alpha_{i})\otimes\alpha^{2};\eta)\\
    &+\sum_{\substack{\beta_{1}+\beta_{2}=\beta\\k_{1}+k_{2}=k+1\\1\leq i\leq k_{1}\\J_{1}\sqcup J_{2}=[l]\\1,2\in J_{1}\\(\beta_{2},k_{2},l_{2})\neq(\beta_{0},1,0)}}(-1)^{\iota(\alpha,\eta;i,J_{1})}\mathfrak{q}_{k_{1},l_{1};0,m'(m,i,k_{2})}^{\beta_{1}}(\alpha^{1}\otimes \mathfrak{q}_{k_{2},l_{2}}^{\beta_{2}}(\alpha^{2};\eta^{2})\otimes\alpha^{3};\eta^{1})\\
    &+\sum_{\substack{\beta_{1}+\beta_{2}=\beta\\k_{1}+k_{2}=k+1\\1\leq i\leq k_{1}\\J_{1}\sqcup J_{2}=[l]\\1,2\in J_{2}\\(\beta_{1},k_{1},l_{1})\neq(\beta_{0},1,0)}}(-1)^{\iota(\alpha,\eta;i,J_{1})}\mathfrak{q}_{k_{1},l_{1}}^{\beta_{1}}(\alpha^{1}\otimes \mathfrak{q}_{k_{2},l_{2};0,m-i+1}^{\beta_{2}}(\alpha^{2};\eta^{2})\otimes\alpha^{3};\eta^{1})\\
    &+\sum_{\substack{\beta_{1}+\beta_{2}=\beta\\k_{1}+k_{2}=k+1\\1\leq i\leq k_{1}\\J_{1}\sqcup J_{2}=[l]\\2\in J_{1},1\in J_{2}}}(-1)^{\iota_{5}(\alpha,\eta;i,J_{1})}\mathfrak{q}_{k_{1},l_{1};i}^{\beta_{1}}(\alpha^{1}\otimes \mathfrak{q}_{k_{2},l_{2};m-i+1}^{\beta_{2}}(\alpha^{2};\eta^{2})\otimes\alpha^{3};\eta^{1})\\
    &+\sum_{\substack{\beta_{1}+\varpi(\beta_{2})=\beta\\J_{1}\sqcup J_{2}=[l]\\1,2\in J_{2}}}(-1)^{\iota_{6}(\alpha,\eta;J_{1})}\mathfrak{q}_{k,l_{1}+1;m}^{\beta_{1}}(\alpha;\mathfrak{q}_{\emptyset,l_{2}}^{\beta_{2}}(\eta^{2})\otimes\eta^{1}).
\end{align*}
Multiplying by $T^{\beta}$ and summing over $\beta\in\Pi$, we get Proposition \ref{ggstruc}.

\section{One-sided interior constraints}

\subsection{The moduli space}\label{sspace}
In order to prove Theorem \ref{compare}, we shall need the moduli space of disks such that the marked point $w_{1}$ is constrained to be to the left of the oriented geodesic from $z_{0}$ to $z_{m}$, denoted by $\Ml$, and similarly the space where $w_{1}$ is constrained to be to the right of this geodesic, denoted by $\Mr$. We define these spaces as fiber products of orbifolds as follows.
\par
Recall that $\chi_{m}:\M\rightarrow\Cinf$ for $m\neq0$ is defined in Section \ref{gop} via
\begin{align*}
    \chi_{m}([\Sigma,u,\vec{z},\vec{w}])=(z_{0},z_{m},\bar{w}_{1},w_{1}).
\end{align*}
Set $\mathcal{P}=\chi_{m}^{-1}(\infty)\subset\M$ and $\mathcal{Q}=\M\setminus\mathcal{P}$. Define
\begin{align*}
    \tilde{\theta}_{m}=\Im\chi_{m}:\mathcal{Q}\rightarrow\mathbb{R},
\end{align*}
where $\Im$ stands for the imaginary part. Let $\Rinf=[-\infty,\infty]$, thought of as a manifold with boundary. Similarly, let $\mathbb{R}_{+}=[0,\infty]$ and $\mathbb{R}_{-}=[-\infty,0]$.
\begin{prop}\label{ext}
We can extend $\tilde{\theta}_{m}$ to a smooth map
\begin{align*}
    \hat{\theta}_{m}:\M\rightarrow\Rinf.
\end{align*}
\end{prop}
\begin{proof}
Fix $[\Sigma,u,\vec{z},\vec{w}]\in\M$, and take $\nu$ as in Proposition \ref{root} for $z_{0},z_{m},w_{1}$. Identifying $\nu$ with $\Cinf$, we identify $\nu\cap\Sigma$ with $\overline{\mathcal{H}}$ where $\mathcal{H}\subset\Cinf$ is the upper half-plane. Let $\phi\in\pslr$ be given by
\begin{align*}
    \phi(w)=\frac{w-z_{0}^{\nu}}{w-z_{m}^{\nu}}.
\end{align*}
In particular, $\phi(z_{0}^{\nu})=0$ and $\phi(z_{m}^{\nu})=\infty$. Writing $\phi(w_{1}^{\nu})=x+iy$, we have
\begin{align*}
    \chi_{m}([\Sigma,u,\vec{z},\vec{w}])
    &=(z_{0}^{\nu},z_{m}^{\nu},\bar{w}_{1}^{\nu},w_{1}^{\nu})\\
    &=(\phi(z_{0}^{\nu}),\phi(z_{m}^{\nu}),\phi(\bar{w}_{1}^{\nu}),\phi(w_{1}^{\nu}))\\
    &=(0,\infty,x-iy,x+iy)\\
    &=\frac{1}{2}-i\frac{x}{2y}.
\end{align*}
Assume first $[\Sigma,u,\vec{z},\vec{w}]\in\mathcal{Q}$. Note that $\mathcal{P}$ consists exactly of those open stable maps with a domain $\Sigma$ such that $w_{1}^{\nu}\in\partial\nu$. We thus have
\begin{align*}
    \tilde{\theta}_{m}([\Sigma,u,\vec{z},\vec{w}])=-\frac{x}{2y}\in\mathbb{R}.
\end{align*}
\par
Now assume $[\Sigma,u,\vec{z},\vec{w}]\in\mathcal{P}$, and let $\{[\Sigma^{r},u^{r},\vec{z^{r}},\vec{w^{r}}]\}_{r=0}^{\infty}$ be a sequence in $\mathcal{Q}$ such that in $\M$ we have
\begin{align*}
    [\Sigma^{r},u^{r},\vec{z^{r}},\vec{w^{r}}]\overset{r\rightarrow\infty}{\longrightarrow}[\Sigma,u,\vec{z},\vec{w}].
\end{align*}
Note that we can assume that this sequence lies in the open stratum of $\M$. Taking $\phi^{r}\in\pslr$ given by
\begin{align*}
    \phi^{r}(w)=\frac{w-z_{0}^{r}}{w-z_{m}^{r}},
\end{align*}
and writing $\phi^{r}(w_{1}^{r})=x^{r}+iy^{r}$, we obtain
\begin{align*}
    x^{r}+iy^{r}\overset{r\rightarrow\infty}{\longrightarrow}x+iy.
\end{align*}
Since $w_{1}^{\nu}\in\partial\nu$, we have $y=0$, and since $w_{1}^{\nu}\neq z_{0}^{\nu}$, we have $x\neq0$. We get
\begin{align*}
    \tilde{\theta}_{m}([\Sigma^{r},u^{r},\vec{z^{r}},\vec{w^{r}}])\overset{t\rightarrow\infty}{\longrightarrow}
    \begin{cases}
    \infty,&x<0,\\
    -\infty,&x>0.
    \end{cases}
\end{align*}
\par
Hence $\tilde{\theta}_{m}$ can be extended continuously to
\begin{align*}
    \hat{\theta}_{m}:\M\rightarrow\Rinf.
\end{align*}
Identifying $S^{1}$ with $(\frac{1}{2}+i\mathbb{R})\cup\{\infty\}\subset\Cinf$, define $\pi:\Rinf\rightarrow S^{1}$ by
\begin{align*}
    \pi(z)=
    \begin{cases}
    \frac{1}{2}+iz,&z\in\mathbb{R}\\
    \infty,&z=\pm\infty.
    \end{cases}
\end{align*}
Note that $\pi$ is smooth and satisfies
\begin{align*}
    \pi\circ\hat{\theta}_{m}=\chi_{m}.
\end{align*}
Since $d_{y}\pi$ is a linear isomorphism for all $y\in\Rinf$, and $\chi_{m}$ is smooth, $\hat{\theta}_{m}$ is smooth as well.
\end{proof}
Let
\begin{align*}
    f:\Rinf\rightarrow\Cinf
\end{align*}
be an embedding such that
\begin{align*}
    f(\mathbb{R}_{+})=[0,1],\quad f(\mathbb{R}_{-})=[-1,0].
\end{align*}
In particular, $f(\Rinf)=[-1,1]$. Define
\begin{align*}
    \theta_{m}=f\circ\hat{\theta}_{m}:\M\rightarrow\Cinf,
\end{align*}
and let $D_{\pm}=D_{1}(\pm1)$ be a closed disk of radius $1$ centered at $\pm1\in\mathbb{C}$.
\par
\begin{prop}\label{transs}
The inclusion $\D\hookrightarrow\Cinf$ is transversal to $\theta_{m}$.
\end{prop}
\begin{proof}
Note that the inclusion $\mathring{D}\hookrightarrow\Cinf$ is a submersion, so it remains to check transversality at $\partial\D\cap[-1,1]=\{0\}$. Recall that by proposition \ref{transg} the inclusion $I\hookrightarrow\Cinf$ is transversal to $\chi_{m}$ so that the inclusion $I\hookrightarrow\mathbb{C}$ is transversal to $\chi_{m}|_{\mathcal{Q}}$. Taking the imaginary part of both the inclusion $I\hookrightarrow\mathbb{C}$ and $\chi_{m}|_{\mathcal{Q}}$, it follows that the inclusion $\{0\}\hookrightarrow\mathbb{R}$ is transversal to $\tilde{\theta}_{m}$. In particular, for $[\Sigma,u,\vec{z},\vec{w}]\in\theta_{m}^{-1}(0)=\tilde{\theta}_{m}^{-1}(0)$, we have
\begin{align*}
    \im d_{[\Sigma,u,\vec{z},\vec{w}]}\theta_{m}=d_{0}f(\im d_{[\Sigma,u,\vec{z},\vec{w}]}\tilde{\theta}_{m})=d_{0}f(\mathbb{R})=\mathbb{R}.
\end{align*}
Note also that
\begin{align*}
    T_{0}\partial D_{\pm}=i\mathbb{R},
\end{align*}
thus
\begin{align*}
    \im d_{[\Sigma,u,\vec{z},\vec{w}]}\theta_{m}\oplus T_{0}\partial\D=\mathbb{R}\oplus i\mathbb{R}=\mathbb{C}.
\end{align*}
\end{proof}
We can thus define
\begin{align*}
    \Ms=\D\times_{\Cinf}\M,
\end{align*}
where the fiber product is taken with respect to the inclusion $\D\hookrightarrow\Cinf$ and to $\theta_{m}$. By Proposition \ref{transs} it is a smooth orbifold with corners. As before, denote by $\evb$ and $\evi$ the evaluation maps.
\begin{prop}
The map $\evbz:\Ms\rightarrow L$ is a proper submersion.
\end{prop}
\begin{proof}
We use the notation of \cite{orb} as described in Section \ref{fibtrans}.
\par
Note that the interior of $(\Ms)_{0}$ is diffeomorphic to an open subset of the interior of $(\M)_{0}$. In Section \ref{sbdry} below we show that any boundary component $\mathcal{B}_{0}$ of $(\Ms)_{0}$ is diffeomorphic either to a boundary component of $(\M)_{0}$ or to $(\Mg)_{0}$. Since both $\evbz:\M\rightarrow L$ and $\evbz:\Mg\rightarrow L$ are assumed to be submersions (see Section \ref{reg}), it follows that $\evbz:\Ms\rightarrow L$ is a submersion.
\par
Now note that $|\Ms|$ is a compact subset of $|\M|$. Since we assume that $\evbz:\M\rightarrow L$ is proper (see Section \ref{reg}), it follows that $\evbz:\Ms\rightarrow L$ is proper.
\end{proof}
We can thus define the corresponding operators
\begin{align*}
    \qs:E^{\otimes k}\otimes C^{\otimes l}\rightarrow C
\end{align*}
by
\begin{align*}
    \qs(\alpha;\eta)=(-1)^{\varepsilon(\alpha)}(\evbz)_{*}\Big(\bigwedge_{j=1}^{l}(\evi)^{*}\eta_{j}\wedge\bigwedge_{j=1}^{k}(\evb)^{*}\alpha_{j}\Big),
\end{align*}
and set
\begin{align*}
    \mathfrak{q}_{k,l;\pm,m}(\alpha;\eta)=\sum_{\beta\in\Pi}T^{\beta}\qs.
\end{align*}
\par
We then have the following.
\begin{prop}[Unit]\label{sunit}
Assume $\alpha_{i}=c\cdot1$ for some $c\in R$ and $1\leq i\leq k$, $i\neq m$ . Then $\qs(\alpha;\eta)=0$.
\end{prop}
\begin{proof}
The proof is verbatim the same as that of Proposition \ref{horunit}, with the moduli spaces and operators with subscript $\perp$ replaced by those with subscript $\pm,m$.
\end{proof}
Set
\begin{align*}
    \iota_{7}(\eta)=|\eta|+n+1,
\end{align*}
and as in section \ref{gop}, let
\begin{align*}
    m'(m,i,k_{2})=
    \begin{cases}
    m,&1\leq m<i,\\
    i,&i\leq m<i+k_{2},\\
    m-k_{2}+1,&i+k_{2}\leq m\leq k.
    \end{cases}
\end{align*}
\begin{prop}[Structure equation for $\Ml$]\label{lstruc}
We have
\begin{align*}
    0&=-\mathfrak{q}_{k,l;+,m}(\alpha;d\eta)\\
    &+\sum_{\substack{k_{1}+k_{2}=k+1\\1\leq i\leq k_{1}\\J_{1}\sqcup J_{2}=[l]\\1\in J_{1}}}(-1)^{\iota(\alpha,\eta;i,J_{1})}\mathfrak{q}_{k_{1},l_{1};+,m'(m,i,k_{2})}(\alpha^{1}\otimes \mathfrak{q}_{k_{2},l_{2}}(\alpha^{2};\eta^{2})\otimes\alpha^{3};\eta^{1})\\
    &+\sum_{\substack{k_{1}+k_{2}=k+1\\m-k_{2}<i\leq m\\J_{1}\sqcup J_{2}=[l]\\1\in J_{2}}}(-1)^{\iota(\alpha,\eta;i,J_{1})}\mathfrak{q}_{k_{1},l_{1}}(\alpha^{1}\otimes \mathfrak{q}_{k_{2},l_{2};+,m-i+1}(\alpha^{2};\eta^{2})\otimes\alpha^{3};\eta^{1})\\
    &+\sum_{\substack{k_{1}+k_{2}=k+1\\m<i\\J_{1}\sqcup J_{2}=[l]\\1\in J_{2}}}(-1)^{\iota(\alpha,\eta;i,J_{1})}\mathfrak{q}_{k_{1},l_{1}}(\alpha^{1}\otimes \mathfrak{q}_{k_{2},l_{2}}(\alpha^{2};\eta^{2})\otimes\alpha^{3};\eta^{1})\\
    &+(-1)^{\iota_{7}(\eta)}\mathfrak{q}_{k,l;m}(\alpha;\eta).
\end{align*}
\end{prop}
\begin{figure}[H]
    \centering
    \includegraphics[scale=0.18]{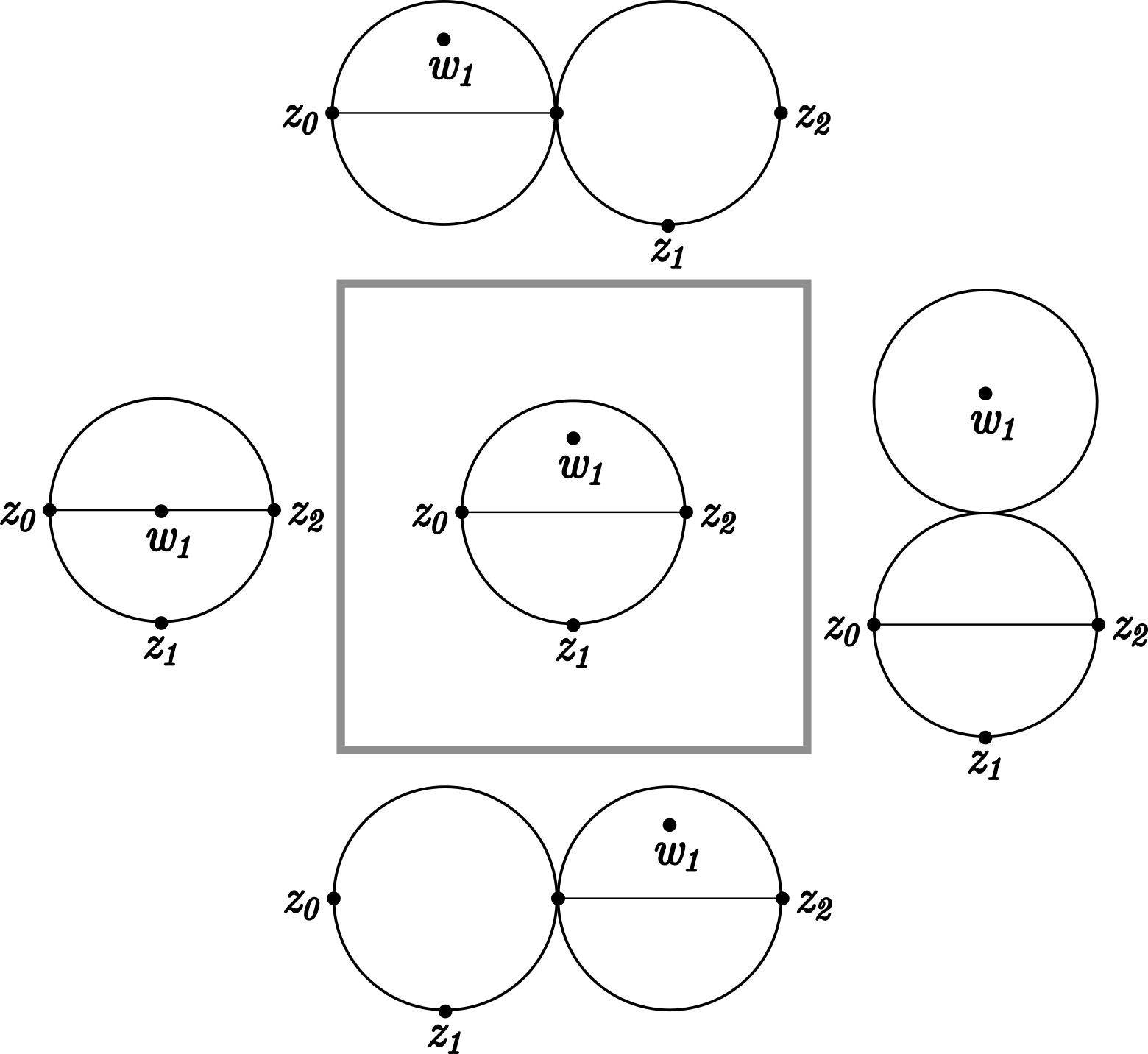}
    \caption{Boundary components of $\mathcal{M}_{3,1;+,2}(\beta_{0})$}
\end{figure}
\begin{prop}[Structure equation for $\Mr$]\label{rstruc}
We have
\begin{align*}
    0&=-\mathfrak{q}_{k,l;-,m}(\alpha;d\eta)\\
    &+\sum_{\substack{k_{1}+k_{2}=k+1\\1\leq i\leq k_{1}\\J_{1}\sqcup J_{2}=[l]\\1\in J_{1}}}(-1)^{\iota(\alpha,\eta;i,J_{1})}\mathfrak{q}_{k_{1},l_{1};-,m'(m,i,k_{2})}(\alpha^{1}\otimes \mathfrak{q}_{k_{2},l_{2}}(\alpha^{2};\eta^{2})\otimes\alpha^{3};\eta^{1})\\
    &+\sum_{\substack{k_{1}+k_{2}=k+1\\m-k_{2}<i\leq m\\J_{1}\sqcup J_{2}=[l]\\1\in J_{2}}}(-1)^{\iota(\alpha,\eta;i,J_{1})}\mathfrak{q}_{k_{1},l_{1}}(\alpha^{1}\otimes \mathfrak{q}_{k_{2},l_{2};-,m-i+1}(\alpha^{2};\eta^{2})\otimes\alpha^{3};\eta^{1})\\
    &+\sum_{\substack{k_{1}+k_{2}=k+1\\i\leq m-k_{2}\\J_{1}\sqcup J_{2}=[l]\\1\in J_{2}}}(-1)^{\iota(\alpha,\eta;i,J_{1})}\mathfrak{q}_{k_{1},l_{1}}(\alpha^{1}\otimes \mathfrak{q}_{k_{2},l_{2}}(\alpha^{2};\eta^{2})\otimes\alpha^{3};\eta^{1})\\
    &-(-1)^{\iota_{7}(\eta)}\mathfrak{q}_{k,l;m}(\alpha;\eta).
\end{align*}
\end{prop}
\begin{figure}[H]
    \centering
    \includegraphics[scale=0.18]{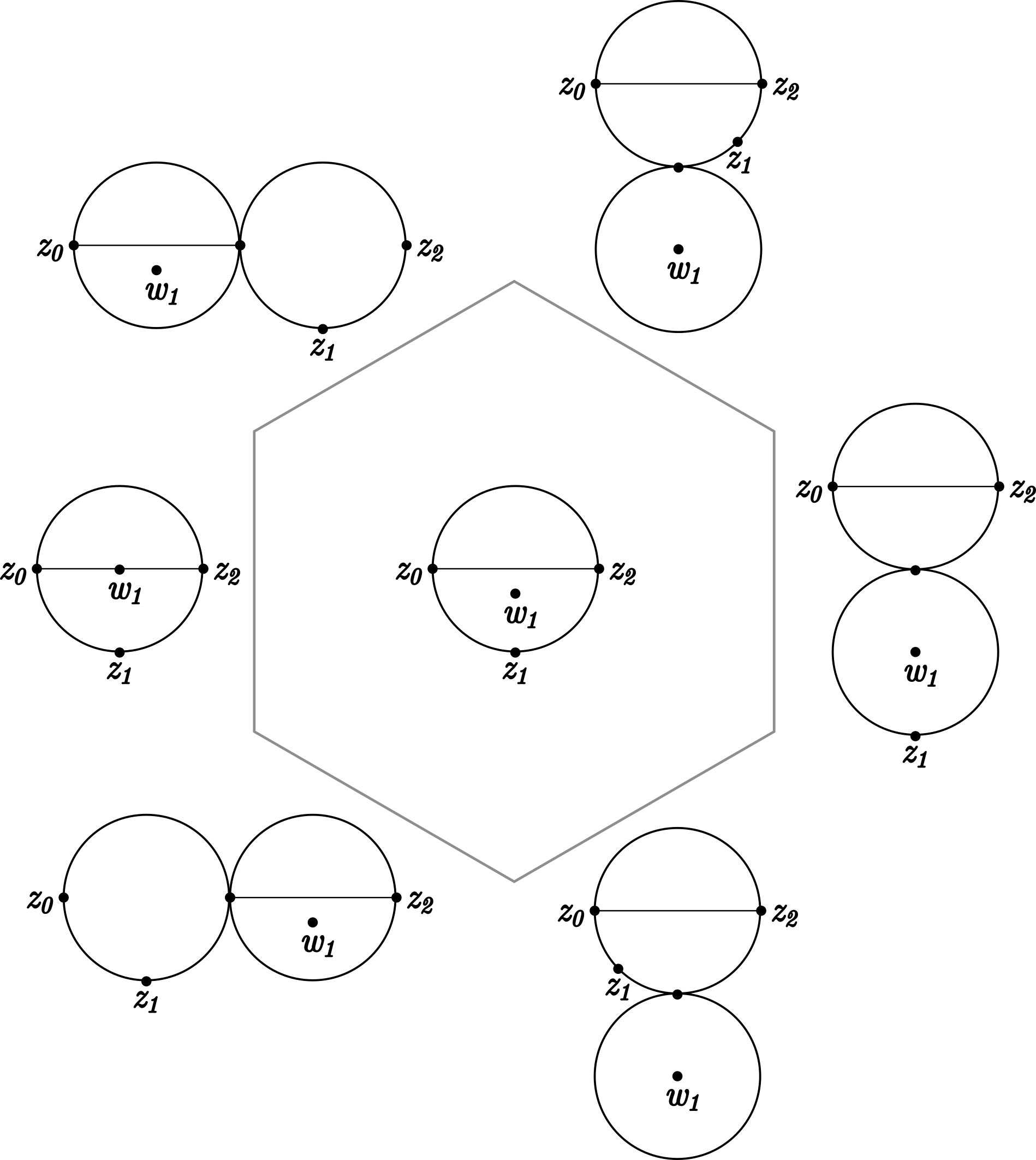}
    \caption{Boundary components of $\mathcal{M}_{3,1;-,2}(\beta_{0})$}
\end{figure}
The proof of propositions \ref{lstruc} and \ref{rstruc} is given in the next two sections. Defining deformed operators as in Section \ref{defop},
\begin{align*}
    \mathfrak{q}_{k,l;\pm,m}^{\gamma,b}(\alpha;\eta)&=\sum_{t,s_{0},...,s_{k}\geq0}\mathfrak{q}_{k+s_{0}+...+s_{k},l+t;\pm,m+s_{0}+...+s_{m-1}}(\alpha^{b;s_{0},...,s_{k}};\eta^{\gamma;t}),
\end{align*}
we get the following.
\begin{prop}
Propositions \ref{sunit}, \ref{lstruc} and \ref{rstruc} hold for the deformed operators as well.
\end{prop}

\subsection{Description of the boundary}\label{sbdry}

We turn to proving propositions \ref{lstruc} and \ref{rstruc}. We begin by describing the boundary of $\Ms=\D\times_{\Cinf}\M$. Proposition \ref{fiberbdry} gives
\begin{align*}
    \partial\Ms
    &=\partial\D\times_{\Cinf}\M+\D\times_{\Cinf}\partial\M.    
\end{align*}
\par
We begin by analysing the first summand, which consists of the boundary component
\begin{align*}
    \mathcal{B}_{\pm}=\partial\D\times_{\Cinf}\M.
\end{align*}
In the following, the signs $\pm$ and $\mp$ correspond to the cases of $\Ms$, where the upper sign corresponds to the case $\Ml$ and the lower sign to the case $\Mr$. Recall from Section \ref{gop} the definition of the geodesic moduli space as a fiber product
\begin{align*}
    \Mg=I\times_{\Cinf}\M,
\end{align*}
where the fiber product is taken with respect to the inclusion $I\hookrightarrow\Cinf$ and to $\chi_{m}$.
\par
As in the proof of Proposition \ref{ext}, we have $\theta_{m}^{-1}(0)=\chi_{m}^{-1}(\frac{1}{2})$. Note that $\partial\D\cap\im\theta_{m}=\{0\}$ and that $I\cap\im\chi_{m}=\{\frac{1}{2}\}$, thus we have a diffeomorphism
\begin{align*}
    \psi_{\pm}:\mathcal{B}_{\pm}\rightarrow\Mg
\end{align*}
given by
\begin{align*}
    \psi_{\pm}(0,[\Sigma,u,\vec{z},\vec{w}])\mapsto(\tfrac{1}{2},[\Sigma,u,\vec{z},\vec{w}]).
\end{align*}
In the following we compute the sign of $\psi_{\pm}$.
\par
Denote by
\begin{align*}
    p:\Mg\rightarrow\M
\end{align*}
the projection and fix $\lambda=[\Sigma,u,\vec{z},\vec{w}]\in\theta_{m}^{-1}(0)=\chi_{m}^{-1}(\frac{1}{2})$. As in section \ref{fibtrans}, the orientation of $\Mg$ is determined by splitting the following short exact sequence:
\begin{center}
\begin{tikzcd}
    0\rightarrow T_{(\frac{1}{2},\lambda)}\Mg\arrow{rr}{0\oplus d_{(\frac{1}{2},\lambda)}p}&&\mathbb{R}\oplus T_{\lambda}\M\arrow{rr}{q-d_{\lambda}\chi_{m}}&&\mathbb{C}\rightarrow0,
\end{tikzcd}  
\end{center}
where
\begin{align*}
    q:\mathbb{R}\oplus T_{\lambda}\M\rightarrow\mathbb{C}
\end{align*}
is the projection to the first coordinate (thinking of $\mathbb{R}$ as a subset of $\mathbb{C}$), and we identify $T_{\frac{1}{2}}I=\mathbb{R}$ and $T_{\frac{1}{2}}\Cinf=\mathbb{C}$. Similarly, the orientation of $\mathcal{B}_{\pm}$ is determined by splitting the following short exact sequence:
\begin{center}
\begin{tikzcd}
    0\rightarrow T_{(0,\lambda)}\mathcal{B}_{\pm}\arrow{rr}{0\oplus d_{(0,\lambda)}(p\circ\psi_{\pm})}&&\mathbb{R}\oplus T_{\lambda}\M\arrow{rr}{\mp i\cdot q-d_{\lambda}\theta_{m}}&&\mathbb{C}\rightarrow0,
\end{tikzcd}  
\end{center}
where we identify $T_{0}\D=\mp i\mathbb{R}\cong\mathbb{R}$ and $T_{0}\Cinf=\mathbb{C}$.
\par
Note that we have an isomorphism of exact sequences given by
\begin{center}
\begin{tikzcd}[column sep = 5mm]
    0\arrow{r}&T_{(0,\lambda)}\mathcal{B}_{\pm}\arrow{rrr}{0\oplus d_{(0,\lambda)}(p\circ\psi_{\pm})}\arrow{d}{d_{(0,\lambda)}\psi_{\pm}}&&&\mathbb{R}\oplus T_{\lambda}\M\arrow{rrrr}{\mp i\cdot q-d_{\lambda}\theta_{m}}\arrow{d}{\pm\textup{id}\oplus\textup{id}}&&&&\mathbb{C}\arrow{d}{i}\arrow{r}&0\;\\
    0\arrow{r}&T_{(\frac{1}{2},\lambda)}\Mg\arrow{rrr}{0\oplus d_{(\frac{1}{2},\lambda)}p}&&&\mathbb{R}\oplus T_{\lambda}\M\arrow{rrrr}{q-d_{\lambda}\chi_{m}}&&&&\mathbb{C}\arrow{r}&0.
\end{tikzcd}  
\end{center}
Therefore, choosing a splitting
\begin{align*}
    \tau_{\pm}:\mathbb{R}\oplus T_{\lambda}\M\rightarrow T_{(0,\lambda)}\mathcal{B}_{\pm}\oplus\mathbb{C}
\end{align*}
induces a splitting
\begin{align*}
    (d_{(0,\lambda)}\psi_{\pm}\oplus i)\circ\tau_{\pm}\circ(\pm\textup{id}\oplus\textup{id}):\mathbb{R}\oplus T_{\lambda}\M\rightarrow T_{(\frac{1}{2},\lambda)}\Mg\oplus\mathbb{C}.
\end{align*}
We conclude that
\begin{align*}
    1=\sgn(d_{(0,\lambda)}\psi_{\pm}\oplus i)\cdot\sgn(\tau_{\pm})\cdot\sgn(\pm\textup{id}\oplus\textup{id})=\pm\sgn(\psi_{\pm}).
\end{align*}
In particular,
\begin{align*}
    \mathcal{B}_{\pm}\cong\pm\Mg.
\end{align*}
\par
We turn to analysing the second summand. Recall that \cite{Ainf}*{Proposition 2.8} states that
\begin{align*}
    \partial\M\cong\sum_{\substack{\beta_{1}+\beta_{2}=\beta\\k_{1}+k_{2}=k+1\\J_{1}\sqcup J_{2}=[l]\\1\leq i\leq k_{1}\\(\beta_{1},k_{1},l_{1})\neq(\beta_{0},1,0)\\(\beta_{2},k_{1},l_{1})\notin\{(\beta_{0},1,0),(\beta_{0},0,0)\}}}(-1)^{\delta_{1}}\mathcal{M}_{k_{1}+1,|J_{1}|}(\beta_{1})\times_{L}\mathcal{M}_{k_{2}+1,|J_{2}|}(\beta_{2}),
\end{align*}
where $\delta_{1}=k_{2}(k_{1}+i)+i+n$ and the fiber product is taken with respect to $evb^{\beta_{1}}_{i},evb^{\beta_{2}}_{0}$. This fiber product represents the marked points $\{z_{j}\}_{j=i}^{i+k_{2}-1}$ and $\{w_{j}\}_{j\in J_{2}}$ bubbling off to a disk. Thus
\begin{align*}
    \D\times_{\Cinf}\partial\M&
    =\sum_{\substack{\beta_{1}+\beta_{2}=\beta\\k_{1}+k_{2}=k+1\\J_{1}\sqcup J_{2}=[l]\\1\leq i\leq k_{1}\\(\beta_{1},k_{1},l_{1})\neq(\beta_{0},1,0)\\(\beta_{2},k_{1},l_{1})\notin\{(\beta_{0},1,0),(\beta_{0},0,0)\}}}(-1)^{\delta_{1}}\mathcal{B}_{\beta_{1},k_{1},J_{1},i}
\end{align*}
where
\begin{align*}
    \mathcal{B}_{\beta_{1},k_{1},J_{1},i}\cong\D\times_{\Cinf}(\mathcal{M}_{k_{1}+1,|J_{1}|}(\beta_{1})\times_{L}\mathcal{M}_{k_{2}+1,|J_{2}|}(\beta_{2})).
\end{align*}
Fix $\beta_{1},k_{1},J_{1},i$. We describe $\mathcal{B}=\mathcal{B}_{\beta_{1},k_{1},J_{1},i}$ explicitly by case. Let
\begin{align*}
    m'=m'(m,i,k_{2})=
    \begin{cases}
    m,&1\leq m<i,\\
    i,&i\leq m<i+k_{2},\\
    m-k_{2}+1,&i+k_{2}\leq m\leq k,
    \end{cases}
\end{align*}
as in Section \ref{gop}, and let
\begin{align*}
    \theta_{m'}^{i}=\theta_{m'}:\mathcal{M}_{k_{i},|J_{i}|}(\beta_{i})\rightarrow\Cinf.
\end{align*}
Denote by $p_{i}:\mathcal{B}\rightarrow\mathcal{M}_{k_{i}+1,|J_{i}|}(\beta_{i})$ the projections, and let $D_{\pm}^{i}$ represents the copy of $D_{\pm}$ in the range of $\theta_{m}^{i}$.
\par
\textbf{Case 1.} If $w_{1}$ is in the component corresponding to $\mathcal{M}_{k_{1}+1,|J_{1}|}(\beta_{1})$, that is $1\in J_{1}$, then $\theta_{m}|_{\mathcal{B}}=\theta_{m'}^{1}\circ p_{1}$. Therefore, by Proposition \ref{fiber},
\begin{align*}
    \mathcal{B}&\cong(\D^{1}\times_{\Cinf}\mathcal{M}_{k_{1}+1,|J_{1}|}(\beta_{1}))\times_{L}\mathcal{M}_{k_{2}+1,|J_{2}|}(\beta_{2})\\
    &=\mathcal{M}_{k_{1}+1,|J_{1}|;\pm,m'}(\beta_{1})\times_{L}\mathcal{M}_{k_{2}+1,|J_{2}|}(\beta_{2}).
\end{align*}
\par
\textbf{Case 2.} Assume now that $w_{1}$ is in the component corresponding to $\mathcal{M}_{k_{2}+1,|J_{2}|}(\beta_{2})$, that is $1\in J_{2}$.
\par
\textbf{Case 2a.} If $i\leq m<i+k_{2}$, then $\theta_{m}|_{\mathcal{B}}=\theta_{m-i+1}^{2}\circ p_{2}$. Therefore, by Proposition \ref{fiber},
\begin{align*}
    \mathcal{B}&\cong(\mathcal{M}_{k_{1}+1,|J_{1}|}(\beta_{1})\times_{L}\mathcal{M}_{k_{2}+1,|J_{2}|}(\beta_{2}))\times_{\Cinf}\D^{2}\\
    &=\mathcal{M}_{k_{1}+1,|J_{1}|}(\beta_{1})\times_{L}(\mathcal{M}_{k_{2}+1,|J_{2}|}(\beta_{2})\times_{\Cinf}\D^{2})\\
    &=\mathcal{M}_{k_{1}+1,|J_{1}|}(\beta_{1})\times_{L}(\D^{2}\times_{\Cinf}\mathcal{M}_{k_{2}+1,|J_{2}|}(\beta_{2}))\\
    &=\mathcal{M}_{k_{1}+1,|J_{1}|}(\beta_{1})\times_{L}\mathcal{M}_{k_{2}+1,|J_{2}|;\pm,m-i+1}(\beta_{2}).
\end{align*}
\par
\textbf{Case 2b.} If $1\leq m<i$, then $\theta_{m}|_{\mathcal{B}}=1$. In the case of $\Ml$ we obtain
\begin{align*}
    \mathcal{B}&\cong\mathcal{M}_{k_{1}+1,|J_{1}|}(\beta_{1})\times_{L}\mathcal{M}_{k_{2}+1,|J_{2}|}(\beta_{2}),
\end{align*}
and in the case of $\Mr$ we obtain
\begin{align*}
    \mathcal{B}&=\emptyset.
\end{align*}
\par
\textbf{Case 2c.} Similarly, if $i+k_{2}\leq m\leq k$, then $\theta_{m}|_{\mathcal{B}}=-1$.
In the case of $\Ml$ we obtain
\begin{align*}
    \mathcal{B}&=\emptyset,
\end{align*}
and in the case of $\Mr$ we obtain
\begin{align*}
    \mathcal{B}&\cong\mathcal{M}_{k_{1}+1,|J_{1}|}(\beta_{1})\times_{L}\mathcal{M}_{k_{2}+1,|J_{2}|}(\beta_{2}).
\end{align*}

\subsection{Applying Stokes' theorem}

We are now ready to derive the structure equation. Set 
\begin{align*}
    \xi=\bigwedge_{j=1}^{l}(\evi)^{*}\eta_{j}\wedge\bigwedge_{j=1}^{k}(\evb)^{*}\alpha_{j},
\end{align*}
and note that
\begin{align*}
    \dim\Ms=\dim\M\equiv n+k\pmod{2}.
\end{align*}
Applying Stokes' theorem (Proposition \ref{Stokes}), we obtain
\begin{align*}
    d((\evbz)_{*}\xi)=(\evbz)_{*}(d\xi)+(-1)^{n+k+|\alpha|+|\eta|}(\evbz|_{\partial\Ms})_{*}\xi.
\end{align*}
As in section \ref{computegg}, we have
\begin{align*}
    d((\evbz)_{*}\xi)=(-1)^{\iota(\alpha,\eta;1,\emptyset)+\varepsilon(\alpha)}\mathfrak{q}_{1,0}^{\beta_{0}}(\qs(\alpha;\eta))
\end{align*}
in addition to
\begin{align*}
    (\evbz)_{*}(d\xi)
    &=(-1)^{\varepsilon(\alpha)}\qs(\alpha;d\eta)\\
    &+\sum_{i=1}^{k}(-1)^{\varepsilon(\alpha)+\iota(\alpha,\eta;i,J)+1}\qs(\alpha^{1}\otimes\mathfrak{q}_{1,0}^{\beta_{0}}(\alpha_{i})\otimes\alpha^{2};\eta),
\end{align*}
and for the last term we need to compute $(\evbz|_{\mathcal{B}})_{*}\xi$ for each boundary component $\mathcal{B}$.
\par
\textbf{Case 1 (bubbling).} Let $\mathcal{B}$ be a boundary component of the form 
\begin{align*}
    \mathcal{B}=(-1)^{\delta_{1}}\mathcal{M}_{k_{1}+1,l_{1};o_{1}}(\beta_{1})\times_{L}\mathcal{M}_{k_{2}+1,l_{2};o_{2}}(\beta_{2}),
\end{align*}
where the set $o_{i}\in\{\emptyset,\{\pm,1\},...,\{\pm,k_{i}\}\}$ indicates the one-sided constraints. We proceed with a computation analogous to Case 1 in Section \ref{computegg}, the only changes being that $g_{i}$ is replaced by $o_{i}$ and $g$ is replaced by $0$. This yields
\begin{align*}
    (\evbz|_{\mathcal{B}})_{*}\xi
    &=(-1)^{\delta_{5}}\mathfrak{q}_{k_{1},l_{1};o_{1}}^{\beta_{1}}(\alpha^{1}\otimes \mathfrak{q}_{k_{2},l_{2};o_{2}}^{\beta_{2}}(\alpha^{2};\eta^{2})\otimes\alpha^{3};\eta^{1}),
\end{align*}
where
\begin{align*}
    \delta_{5}=\iota(\alpha,\eta;i,J_{1})+\varepsilon(\alpha)+n+|\alpha|+|\eta|+k+1.
\end{align*}
\par
\textbf{Case 2 (geodesic).} For the boundary component
\begin{align*}
    \mathcal{B}=\pm\Mg,
\end{align*}
we have by definition
\begin{align*}
    (\evbz|_{\mathcal{B}})_{*}\xi
    &=\pm(-1)^{\varepsilon(\alpha)+|\alpha|+k}\mathfrak{q}_{k_{1},l_{1};m}^{\beta}(\alpha;\eta).
\end{align*}
\par
Combining the above, multiplying by $T^{\beta}$ and summing over $\beta\in\Pi$, we get propositions \ref{lstruc} and \ref{rstruc}.

\section{Cohomologies}\label{cohomology}

\subsection{Quantum cohomology and currents}

For a closed $\gamma\in(\mathcal{I}E)_{2}$, let $\mathfrak{q}_{\emptyset,l}^{\gamma}$ be as in Section \ref{defop}. In the following $\gamma$ serves as a bulk deformation in the sense of \cite{LIFT}. Recall that the quantum cohomology $\qhx$ of $X$ is, as an additive group, the cohomology of the complex $E$ defined in Section \ref{op}, with differential $d$. Equivalently,
\begin{align*}
    \qhx=H^{*}(X)\otimes R,
\end{align*}
where $\otimes$ stands for the completed tensor product as in Section \ref{defop}.
\par
The quantum product is defined via the operation
\begin{align*}
    \xmx:E\otimes E\rightarrow E,
\end{align*}
which is a chain map and thus descends to cohomology (see \cite{Ainf}*{Proposition 3.13}). Denote by
\begin{align*}
    *:\qhx\otimes\qhx\rightarrow\qhx
\end{align*}
the resulting product on $\qhx$, which makes it a graded-commutative ring (see propositions \ref{degree}, \ref{symmetry}, and \ref{unit}, as well as \cite{RQH}*{Proposition 3.5}).
\begin{con}
We denote by $y$ a general element of $\qhx$, possibly adding an index.
\end{con}
Recall that, for a manifold $M$, the Poincaré pairing of $x_{1},x_{2}\in H^{*}(M)$ is defined by
\begin{align*}
    \langle x_{1},x_{2}\rangle_{M}=\int_{M}x_{1}\wedge x_{2}.
\end{align*}
The quantum product has the following property.
\begin{prop}[\cite{Jhol}*{Proposition 11.1.11}]\label{pairing}
We have
\begin{align*}
    \langle y_{1}*y_{2},y_{3}\rangle_{X}=\langle y_{1},y_{2}*y_{3}\rangle_{X}.
\end{align*}
\end{prop}
\par
We turn to define a modified quantum product, allowing one of the inputs to be a current. Set
\begin{align*}
    \overline{E}=\overline{A}^{*}(X)\otimes R.
\end{align*}
\begin{con}
We denote by $\zeta$ a general element of $\overline{E}$, possibly adding an index.
\end{con}
Define
\begin{align*}
    \overline{\mathfrak{q}}^{\beta,\gamma}_{\emptyset,2}:E\otimes\overline{E}\rightarrow\overline{E}
\end{align*}
via
\begin{align*}
    \overline{\mathfrak{q}}^{\beta,\gamma}_{\emptyset,2}(\eta_{1},\zeta)(\eta_{2})=\zeta(\mathfrak{q}^{\beta,\gamma}_{\emptyset,2}(\eta_{2},\eta_{1})),
\end{align*}
and set
\begin{align*}
    \xmcx=\sum_{\beta\in\Pi}T^{\beta}\overline{\mathfrak{q}}^{\beta,\gamma}_{\emptyset,2}.
\end{align*}
\begin{prop}
$\overline{\mathfrak{q}}^{\beta,\gamma}_{\emptyset,2}$ is well defined.
\end{prop}
\begin{proof}
The pullback and pushforward are continuous (see \cite{orb}), hence so is $\mathfrak{q}^{\beta,\gamma}_{\emptyset,2}$. Denote by $h_{\eta_{1}}:E\rightarrow E\otimes E$ the map given by $\eta_{2}\mapsto\eta_{2}\otimes\eta_{1}$, which is continuous. Note that $\overline{\mathfrak{q}}^{\beta,\gamma}_{\emptyset,2}(\eta_{1},\zeta)=\zeta\circ\mathfrak{q}^{\beta,\gamma}_{\emptyset,2}\circ h_{\eta_{1}}$ is continuous, thus if $\zeta$ is continuous then so is $\overline{\mathfrak{q}}^{\beta,\gamma}_{\emptyset,2}(\eta_{1},\zeta)$.
\end{proof}
\begin{prop}
$\xmcx$ is of degree $0$.
\end{prop}
\begin{proof}
Take $\eta_{1},\eta_{2}\in A^{*}(X)$ and $\zeta\in\overline{A}^{*}(X)$. If $\xmcx(\eta_{1},\zeta)(\eta_{2})\neq0$ we obtain
\begin{align*}
    2n&=|\zeta|+|\mathfrak{q}^{\beta,\gamma}_{\emptyset,2}(\eta_{1},\eta_{2})|\\&=|\zeta|+|\eta_{1}|+|\eta_{2}|-\mu(\beta),
\end{align*}
so that
\begin{align*}
    |\overline{\mathfrak{q}}^{\beta,\gamma}_{\emptyset,2}(\eta_{1},\zeta)|&=2n-|\eta_{2}|\\&=2n-(2n-|\zeta|-|\eta_{1}|+\mu(\beta))\\&=|\zeta|+|\eta_{1}|-\mu(\beta).
\end{align*}
\end{proof}
Denote by $\qhcx$ the cohomology of $\overline{E}$ with the differential $d$ as defined in Section \ref{orientint}.
\begin{con}
We denote by $z$ a general element of $\qhcx$, possibly adding an index.
\end{con}
We have the following.
\begin{prop}
$\xmcx$ is a chain map, and thus decends to a map
\begin{align*}
    \star:\qhx\otimes\qhcx\rightarrow\qhcx.
\end{align*}
\end{prop}
\begin{proof}
Since $\xmx$ is a chain map, we have
\begin{align*}
    d\xmcx(\eta_{1},\zeta)(\eta_{2})&=(-1)^{|\eta_{2}|+1}\xmcx(\eta_{1},\zeta)(d\eta_{2})\\
    &=(-1)^{|\eta_{2}|+1}\zeta(\xmx(d\eta_{2},\eta_{1}))\\
    &=\zeta(\xmx(\eta_{2},d\eta_{1}))+(-1)^{|\eta_{2}|+1}\zeta(d\xmx(\eta_{2},\eta_{1}))\\
    &=\xmcx(d\eta_{1},\zeta)(\eta_{2})+(-1)^{|\eta_{1}|}d\zeta(\xmx(\eta_{2},\eta_{1}))\\
    &=\xmcx(d\eta_{1},\zeta)(\eta_{2})+(-1)^{|\eta_{1}|}\xmcx(\eta_{1},d\zeta)(\eta_{2}).
\end{align*}
\end{proof}
\begin{rmk}
Recall that the tensor product of two cochain complexes $A_{1},A_{2}$ is the cochain complex $A_{1}\otimes A_{2}$ with the obvious grading and with coboundary given by the Leibniz rule
\begin{align*}
    d(x_{1}\otimes x_{2})=dx_{1}\otimes x_{2}+(-1)^{|x_{1}|}x_{1}\otimes dx_{2},
\end{align*}
so that any chain map $A_{1}\otimes A_{2}\rightarrow A_{3}$ descends to a map $H^{*}(A_{1})\otimes H^{*}(A_{2})\rightarrow H^{*}(A_{3})$ in cohomology.
\end{rmk}
\begin{rmk}
Denote by $\cqhx$ the dual of $\qhx$. We have a well defined linear map $\psi:\qhcx\rightarrow\cqhx$ given by $\psi([\zeta])([\eta])=\zeta(\eta)$, which is an isomorphism as a corollary of \cite{dR}*{Theorem 14}. Thus we can think of the elements of $\qhcx$ as elements of $\cqhx$, with
\begin{align*}
    (y_{1}\star z)(y_{2})=z(y_{2}*y_{1}).
\end{align*}
\end{rmk}
\begin{prop}
The map $\star$ makes $\qhcx$ a module over $\qhx$.
\end{prop}
\begin{proof}
From the associativity of $*$ we obtain
\begin{align*}
    (y_{1}\star(y_{2}\star z))(y_{3})&=(y_{2}\star z)(y_{3}*y_{1})\\&=z(y_{3}*y_{1}*y_{2})\\&=((y_{1}*y_{2})\star z)(y_{3}).
\end{align*}
We also have
\begin{align*}
    (1\star z)(y)&=z(y*1)\\&=z(y).
\end{align*}
\end{proof}
Recall the map $\hat{\varphi}:H^{*}(X)\rightarrow\overline{H}^{*}(X)$ from Proposition \ref{curhomprop}. We may extend it to a map $\qhx\rightarrow\qhcx$, also denoted by $\hat{\varphi}$. Thinking of $\qhx$ with the quantum product as a module over itself, we show the following.
\begin{prop}\label{qcuriso}
$\hat{\varphi}$ is a module isomorphism.
\end{prop}
\begin{proof}
By Proposition \ref{curhomiso}, $\hat{\varphi}$ is a linear isomorphism over $\mathbb{R}$, and is therefore bijective. To see that $\hat{\varphi}$ is a homomorphism of modules over $\qhx$, note that Proposition \ref{pairing} gives
\begin{align*}
    (y_{1}\star\hat{\varphi}(y_{2}))(y_{3})
    &=\hat{\varphi}(y_{2})(y_{3}*y_{1})\\
    &=\langle y_{3}*y_{1},y_{2}\rangle_{X}\\
    &=\langle y_{3},y_{1}*y_{2}\rangle_{X}\\
    &=\hat{\varphi}(y_{1}*y_{2})(y_{3}).
\end{align*}
Thus $\hat{\varphi}$ is a bijective module homomorphism, that is, a module isomorphism.
\end{proof}

\subsection{Lagrangian Floer cohomology}\label{Lcohomology}

In this section we review the definition of the Lagrangian Floer cohomology $\qhl$. Let $(\gamma,b)$ be a bounding pair as in Definition \ref{bpdef}, and let $\mathfrak{q}_{k,l}^{\gamma,b}$ be as in Section \ref{defop}. In the terminology of \cite{LIFT}, $\gamma$ is a bulk deformation, and $b$ is a weakly bounding cochain for $\gamma$. 
\begin{prop}\label{lcomplex}
$(C,\dl)$ forms a cochain complex.
\end{prop}
\begin{proof}
Proposition \ref{struc} with $k=1$, $l=0$ gives
\begin{align*}
    0&=\dl(\dl(\alpha))+\lml(\zl,\alpha)+(-1)^{|\alpha|+1}\lml(\alpha,\zl).
\end{align*}
Writing $\zl=c\cdot1$ and applying Proposition \ref{unit}, we get
\begin{align*}
    0&=\dl(\dl(\alpha))+c\cdot\alpha+(-1)^{|\alpha|+1+|\alpha|}c\cdot\alpha\\&=\dl(\dl(\alpha)).
\end{align*}
\end{proof}
Define $\qhl$ to be the cohomology of the complex $(C,\dl)$.
\begin{con}
We denote by $a$ a general element of $\qhl$, possibly adding an index.
\end{con}
We turn to defining the ring structure.
\begin{prop}\label{lml}
The map $C\otimes C\rightarrow C$ given by
\begin{align*}
    (\alpha_{1},\alpha_{2})\mapsto(-1)^{|\alpha_{1}|}\lml(\alpha_{1},\alpha_{2})
\end{align*}
is a chain map, and thus descends to a map 
\begin{align*}
    \circ:\qhl\otimes\qhl\rightarrow\qhl.
\end{align*}
\end{prop}
\begin{proof}
Proposition \ref{struc} with $k=2$, $l=0$ gives
\begin{align*}
    0&=\dl(\lml(\alpha_{1},\alpha_{2}))\\&+\lml(\dl(\alpha_{1}),\alpha_{2})\\&+(-1)^{|\alpha_{1}|+1}\lml(\alpha_{1},\dl(\alpha_{2}))\\
    &+\mathfrak{q}^{\gamma,b}_{3,0}(\zl,\alpha_{1},\alpha_{2})\\&+(-1)^{|\alpha_{1}|+1}\mathfrak{q}^{\gamma,b}_{3,0}(\alpha_{1},\zl,\alpha_{2})\\&+(-1)^{|\alpha_{1}|+|\alpha_{2}|}\mathfrak{q}^{\gamma,b}_{3,0}(\alpha_{1},\alpha_{2},\zl).
\end{align*}
Since $\zl=c\cdot1$, applying Proposition \ref{unit} in this case, we get
\begin{align*}
    0&=\dl(\lml(\alpha_{1},\alpha_{2}))\\&+\lml(\dl(\alpha_{1}),\alpha_{2})\\&+(-1)^{|\alpha_{1}|+1}\lml(\alpha_{1},\dl(\alpha_{2})).
\end{align*}
Rearranging and multiplying by $(-1)^{|\alpha_{1}|}$ gives
\begin{align*}
    \dl((-1)&^{|\alpha_{1}|}\lml(\alpha_{1},\alpha_{2}))=\\&=(-1)^{|\dl(\alpha_{1})|}\lml(\dl(\alpha_{1}),\alpha_{2})+(-1)^{|\alpha_{1}|}(-1)^{|\alpha_{1}|}\lml(\alpha_{1},\dl(\alpha_{2})).
\end{align*}
\end{proof}
\begin{rmk}
In order to achieve associativity our product differs by a sign from the one found in \cite{LIFT}, where it is defined via just $\lml(\alpha_{1},\alpha_{2})$.
\end{rmk}
In order to show that this product makes $\qhl$ a graded ring, we need to show respect for the grading, the existence of a unit, and associativity.
\begin{prop}\label{degprod}
$\circ$ if of degree $0$.
\end{prop}
\begin{proof}
This is a special case of Proposition \ref{degree}.
\end{proof}
\begin{prop}
$\circ$ has a unit $1:=[1]\in HF^{0}(L)$.
\end{prop}
\begin{proof}
Proposition \ref{unit} gives $\dl(1)=0$, thus $1$ is closed and represents a class $[1]\in HF^{0}(L)$. The same proposition also gives
\begin{align*}
    \lml(1,\alpha)=(-1)^{|\alpha|}\lml(\alpha,1)=\alpha,
\end{align*}
that is
\begin{align*}
    1\circ a=a\circ1=a.
\end{align*}
\end{proof}
\begin{prop}\label{QHLasc}
$\circ$ is associative.
\end{prop}
\begin{proof}
Let $a_{i}=[\alpha_{i}]\in\qhl$, $1\leq i\leq3$. Similarly to the proof of \ref{lml} above, Proposition \ref{struc} for $k=3$, $l=0$ gives
\begin{align*}
        0&=\dl(\mathfrak{q}^{\gamma,b}_{3,0}(\alpha_{1},\alpha_{2},\alpha_{3}))\\&+\lml(\lml(\alpha_{1},\alpha_{2}),\alpha_{3})\\&+(-1)^{|\alpha_{1}|+1}\lml(\alpha_{1},\lml(\alpha_{2},\alpha_{3}))\\
        &+\mathfrak{q}^{\gamma,b}_{3,0}(\dl(\alpha_{1}),\alpha_{2},\alpha_{3})\\&+(-1)^{|\alpha_{1}|+1}\mathfrak{q}^{\gamma,b}_{3,0}(\alpha_{1},\dl(\alpha_{2}),\alpha_{3})\\&+(-1)^{|\alpha_{1}|+|\alpha_{2}|}\mathfrak{q}^{\gamma,b}_{3,0}(\alpha_{1},\alpha_{2},\dl(\alpha_{3}))\\&+\mathfrak{q}^{\gamma,b}_{4,0}(\zl,\alpha_{1},\alpha_{2},\alpha_{3})\\&+(-1)^{|\alpha_{1}|+1}\mathfrak{q}^{\gamma,b}_{4,0}(\alpha_{1},\zl,\alpha_{2},\alpha_{3})\\&+(-1)^{|\alpha_{1}|+|\alpha_{2}|}\mathfrak{q}^{\gamma,b}_{4,0}(\alpha_{1},\alpha_{2},\zl,\alpha_{3})\\&+(-1)^{|\alpha_{1}|+|\alpha_{2}|+|\alpha_{3}|+1}\mathfrak{q}^{\gamma,b}_{4,0}(\alpha_{1},\alpha_{2},\alpha_{3},\zl).
\end{align*}
Since $\zl=c\cdot1$, applying Proposition \ref{unit}, in cohomology we get
\begin{align*}
    0&=(-1)^{|a_{1}|+(|a_{1}\circ a_{2}|)}(a_{1}\circ a_{2})\circ a_{3}\\&+(-1)^{(|a_{1}|+1)+|a_{2}|+|a_{1}|}a_{1}\circ(a_{2}\circ a_{3}),
\end{align*}
that is, by Proposition \ref{degprod},
\begin{align*}
    (a_{1}\circ a_{2})\circ a_{3}=a_{1}\circ(a_{2}\circ a_{3}).
\end{align*}
\end{proof}
\begin{rmk}\label{QHLint}
Recall that Proposition \ref{defprop} gives $\int_{L}\dl(\alpha)=0$. Hence there is a well defined map $\int_{L}:\qhl\rightarrow R$. The same proposition also gives $\int_{L}a_{1}\circ a_{2}=\langle a_{1},a_{2}\rangle_{L}$.
\end{rmk}
\begin{prop}\label{Lpairing}
We have
\begin{align*}
    \langle a_{1}\circ a_{2},a_{3}\rangle_{L}=\langle a_{1},a_{2}\circ a_{3}\rangle_{L}
\end{align*}
\end{prop}
\begin{proof}
This follows immediately from Remark \ref{QHLint} and Proposition \ref{QHLasc}.
\end{proof}
\begin{rmk}
Assume that there exists an anti-symplectic involution $\phi$ as in \cite{BP}*{Section 4.1}, that $b=\gamma=0$, and that the relative spin structure on $L$ is in fact induced by a spin structure. Assume also that the minimal Maslov number of $L$ is $\mu_{0}\in4\mathbb{Z}$ and take $R=\Lambda$, so that $\phi^{*}=id$. Then \cite{BP}*{Proposition 4.5} yields
\begin{align*}
    \mathfrak{q}_{2,0}(\alpha_{1},\alpha_{2})=(-1)^{|\alpha_{1}||\alpha_{2}|+|\alpha_{1}|+|\alpha_{2}|}\mathfrak{q}_{2,0}(\alpha_{2},\alpha_{1}),
\end{align*}
so that $\qhl$ is commutative in the graded sense, that is
\begin{align*}
    a_{1}\circ a_{2}=(-1)^{|a_{1}||a_{2}|}a_{2}\circ a_{1}.
\end{align*}
\end{rmk}

\section{Main theorems}

\subsection{Algebra structure}

This section is devoted to the proof of Theorem \ref{algebra}.
\begin{prop}\label{algmap}
The map
\begin{align*}
    (-1)^{n}\xml:E\otimes C\rightarrow C
\end{align*}
is a chain map, and thus descends to a map 
\begin{align*}
    \circledast:\qhx\otimes\qhl\rightarrow\qhl.
\end{align*}
\end{prop}
\begin{proof}
Applying Proposition \ref{gstruc} with $k=m=l=1$, we get
\begin{align*}
    0
    &=\xml(\alpha;d\eta)\\
    &+(-1)^{|\eta|+|\alpha|+1}\mathfrak{q}^{\gamma,b}_{2,1;1}(\alpha,\zl;\eta)\\
    &+(-1)^{|\eta|}\mathfrak{q}^{\gamma,b}_{2,1;2}(\zl,\alpha;\eta)\\
    &-\dl(\xml(\alpha;\eta))\\
    &+(-1)^{|\eta|}\xml(\dl(\alpha);\eta).
\end{align*}
Since $\zl=c\cdot1$, Proposition \ref{gunit} then gives
\begin{align*}
    \dl(\xml(\alpha;\eta))=\xml(\alpha;d\eta)+(-1)^{|\eta|}\xml(\dl(\alpha);\eta).
\end{align*}
\end{proof}
\begin{prop}
$\circledast$ is of degree $0$.
\end{prop}
\begin{proof}
    This follows immediately from Proposition \ref{gdegree}.
\end{proof}
\begin{prop}\label{mod}
The map $\circledast$ makes $\qhl$ a module over $\qhx$.
\end{prop}
\begin{proof}
By Proposition \ref{gfunclass} we have $1\circledast a=a$. Applying Proposition \ref{ggstruc} with $k=m=1$, $l=2$, we get
\begin{align*}
    0&=-\mathfrak{q}^{\gamma,b}_{1,2;0,1}(\alpha;d\eta_{1},\eta_{2})\\
    &-(-1)^{|\eta_{1}|}\mathfrak{q}^{\gamma,b}_{1,2;0,1}(\alpha;\eta_{1},d\eta_{2})\\
    &+(-1)^{|\eta_{1}|+|\eta_{2}|+1+n}\mathfrak{q}^{\gamma,b}_{1,1;1}(\alpha;\mathfrak{q}^{\gamma}_{\emptyset,2}(\eta_{1},\eta_{2}))\\
    &+(-1)^{|\eta_{1}|+|\eta_{2}|+|\alpha|}\mathfrak{q}^{\gamma,b}_{2,2;0,2}(\mathfrak{q}^{\gamma,b}_{0,0},\alpha;\eta_{1},\eta_{2})\\
    &+(-1)^{|\eta_{1}|+|\eta_{2}|}\mathfrak{q}^{\gamma,b}_{1,2;0,1}(\mathfrak{q}^{\gamma,b}_{1,0}(\alpha);\eta_{1},\eta_{2})\\
    &+(-1)^{|\eta_{1}|+|\eta_{2}|+|\alpha|+1}\mathfrak{q}^{\gamma,b}_{2,2;0,1}(\alpha,\mathfrak{q}^{\gamma,b}_{0,0};\eta_{1},\eta_{2})\\
    &+\mathfrak{q}^{\gamma,b}_{1,0}(\mathfrak{q}^{\gamma,b}_{1,2;0,1}(\alpha;\eta_{1},\eta_{2}))\\
    &+(-1)^{|\eta_{2}|+|\eta_{1}|+|\eta_{2}||\eta_{1}|}\mathfrak{q}^{\gamma,b}_{1,1;1}(\mathfrak{q}^{\gamma,b}_{1,1;1}(\alpha;\eta_{1});\eta_{2}).
\end{align*}
Since $\zl=c\cdot1$, Proposition \ref{ggunit} then gives
\begin{align*}
    0&=(-1)^{|y_{1}|+|y_{2}|+1+n+n}(y_{1}*y_{2})\circledast a\\&+(-1)^{|y_{2}|+|y_{1}|+|y_{2}||y_{1}|+n+n}y_{2}\circledast(y_{1}\circledast a),
\end{align*}
that is,
\begin{align*}
    (y_{2}*y_{1})\circledast a=y_{2}\circledast(y_{1}\circledast a).
\end{align*}
\end{proof}
\begin{prop}\label{alg}
The map $\circledast$ makes $\qhl$ an algebra over $\qhx$.
\end{prop}
\begin{proof}
Applying Proposition \ref{gstruc} with $k=2$, $m=l=1$, we get
\begin{align*}
    0
    &=\mathfrak{q}^{\gamma,b}_{2,1;1}(\alpha_{1},\alpha_{2};d\eta)\\
    &+(-1)^{|\eta|+|\alpha_{1}|+|\alpha_{2}|}\mathfrak{q}^{\gamma,b}_{3,1;1}(\alpha_{1},\alpha_{2},\zl;\eta)\\
    &+(-1)^{|\eta|+|\alpha_{1}|+1}\mathfrak{q}^{\gamma,b}_{3,1;1}(\alpha_{1},\zl,\alpha_{2};\eta)\\
    &+(-1)^{|\eta|+|\alpha_{1}|+1}\mathfrak{q}^{\gamma,b}_{2,1;1}(\alpha_{1},\dl(\alpha_{2});\eta)\\
    &+(-1)^{|\eta|}\mathfrak{q}^{\gamma,b}_{3,1;2}(\zl,\alpha_{1},\alpha_{2};\eta)\\
    &-\lml(\xml(\alpha_{1};\eta),\alpha_{2})\\
    &-\dl(\mathfrak{q}^{\gamma,b}_{2,1;1}(\alpha_{1},\alpha_{2};\eta))\\
    &+(-1)^{|\eta|}\mathfrak{q}^{\gamma,b}_{2,1;1}(\dl(\alpha_{1}),\alpha_{2};\eta)\\
    &+(-1)^{|\eta|}\xml(\lml(\alpha_{1};\alpha_{2});\eta).
\end{align*}
Since $\zl=c\cdot1$, Proposition \ref{gunit} then gives
\begin{align*}
    0&=(-1)^{|y|+|a_{1}|+n+1}(y\circledast a_{1})\circ a_{2}\\&+(-1)^{|y|+|a_{1}|+n}y\circledast(a_{1}\circ a_{2}),
\end{align*}
that is,
\begin{align*}
   y\circledast(a_{1}\circ a_{2})=(y\circledast a_{1})\circ a_{2}.
\end{align*}
Similarly, applying Proposition \ref{gstruc} with with $k=m=2$, $l=1$, we get
\begin{align*}
    0
    &=\mathfrak{q}^{\gamma,b}_{2,1;2}(\alpha_{1},\alpha_{2};d\eta)\\
    &+(-1)^{|\eta|+|\alpha_{1}|+|\alpha_{2}|}\mathfrak{q}^{\gamma,b}_{3,1;2}(\alpha_{1},\alpha_{2},\zl;\eta)\\
    &+(-1)^{|\eta|+|\alpha_{1}|+1}\mathfrak{q}^{\gamma,b}_{3,1;3}(\alpha_{1},\zl,\alpha_{2};\eta)\\
    &+(-1)^{|\eta|}\mathfrak{q}^{\gamma,b}_{3,1;3}(\zl,\alpha_{1},\alpha_{2};\eta)\\
    &+(-1)^{|\eta|}\mathfrak{q}^{\gamma,b}_{2,1;2}(\dl(\alpha_{1}),\alpha_{2};\eta)\\
    &+(-1)^{|\eta||\alpha_{1}|+|\eta|+1}\lml(\alpha_{1},\xml(\alpha_{2};\eta))\\
    &-\dl(\mathfrak{q}^{\gamma,b}_{2,1;2}(\alpha_{1},\alpha_{2};\eta))\\
    &+(-1)^{|\eta|+|\alpha_{1}|+1}\mathfrak{q}^{\gamma,b}_{2,1;2}(\alpha_{1},\dl(\alpha_{2});\eta)\\
    &+(-1)^{|\eta|}\xml(\lml(\alpha_{1},\alpha_{2});\eta).
\end{align*}
As before, this gives
\begin{align*}
    0&=(-1)^{|y||a_{1}|+|y|+|a_{1}|+n+1}a_{1}\circ(y\circledast a_{2})\\&+(-1)^{|y|+|a_{1}|+n}y\circledast(a_{1}\circ a_{2}),
\end{align*}
that is,
\begin{align*}
    y\circledast(a_{1}\circ a_{2})=(-1)^{|y||a_{1}|}a_{1}\circ(y\circledast a_{2}).
\end{align*}
\end{proof}

\subsection{Homomorphisms}

This section is devoted to the proof of Theorem \ref{maps}.

\subsubsection{The closed-open map}

\begin{prop}\label{cochain}
The map
\begin{align*}
    -\xtl:E\rightarrow C
\end{align*}
is a chain map, and thus descends to a map 
\begin{align*}
    \co:\qhx\rightarrow\qhl.
\end{align*}
\end{prop}
\begin{proof}
By Proposition \ref{struc} with $k=0$, $l=1$, we have
\begin{align*}
    0&=-\xtl(d\eta)\\&+(-1)^{|\eta|}\mathfrak{q}^{\gamma,b}_{1,1}(\zl;\eta)\\&+\dl(\xtl(\eta)).
\end{align*}
Since $\zl=c\cdot1$, Proposition \ref{unit} then yields
\begin{align*}
    \xtl(d\eta)=\dl(\xtl(\eta)).
\end{align*}
\end{proof}
\begin{prop}\label{COdeg}
$\xtl$ is of degree $0$.
\end{prop}
\begin{proof}
    This follows immediately from Proposition \ref{degree}.
\end{proof}
\begin{prop}\label{COmod}
The map $\co$ is a module homomorphism.
\end{prop}
\begin{proof}
By Proposition \ref{gzstruc} with $k=0$, $l=2$, we have
\begin{align*}
    0
    &=\mathfrak{q}^{\gamma,b}_{0,2;0}(d\eta_{1},\eta_{2})\\\
    &+(-1)^{|\eta_{1}|}\mathfrak{q}^{\gamma,b}_{0,2;0}(\eta_{1},d\eta_{2})\\
    &+(-1)^{|\eta_{1}|+|\eta_{2}|}\mathfrak{q}^{\gamma,b}_{1,2;0}(\zl;\eta_{1},\eta_{2})\\
    &-\dl(\mathfrak{q}^{\gamma,b}_{0,2;0}(\eta_{1},\eta_{2}))\\
    &+(-1)^{|\eta_{1}|+|\eta_{2}|+|\eta_{1}||\eta_{2}|}\xml(\xtl(\eta_{1});\eta_{2})\\
    &+(-1)^{|\eta_{1}|+|\eta_{2}|+n+1}\xtl(\xmx(\eta_{1},\eta_{2})).
\end{align*}
Since $\zl=c\cdot1$, Proposition \ref{gunit} then yields
\begin{align*}
    0&=(-1)^{|y_{1}|+|y_{2}|+|y_{1}||y_{2}|+n}y_{2}\circledast\co(y_{1})\\&+(-1)^{|y_{1}|+|y_{2}|+n+1}\co(y_{1}*y_{2}),
\end{align*}
that is,
\begin{align*}
    y_{2}\circledast\co(y_{1})=\co(y_{2}*y_{1}).
\end{align*}
\end{proof}
\begin{prop}\label{COrng}
The map $\co$ is a ring homomorphism.
\end{prop}
\begin{proof}
By Proposition \ref{hstruc} with $k=0$, $l=2$, we have
\begin{align*}
    0
    &=-\mathfrak{q}^{\gamma,b}_{0,2;\perp}(d\eta_{1},\eta_{2})\\
    &+(-1)^{|\eta_{1}|+1}\mathfrak{q}^{\gamma,b}_{0,2;\perp}(\eta_{1},d\eta_{2})\\
    &+\xtl(\xmx(\eta_{1},\eta_{2}))\\
    &-\dl(\mathfrak{q}^{\gamma,b}_{0,2;\perp}(\eta_{1},\eta_{2}))\\
    &+(-1)^{|\eta_{1}|+|\eta_{2}|}\mathfrak{q}^{\gamma,b}_{1,2;\perp}(\zl;\eta_{1},\eta_{2})\\
    &+(-1)^{|\eta_{1}||\eta_{2}|+|\eta_{2}|}\lml(\xtl(\eta_{2}),\xtl(\eta_{1})).
\end{align*}
Since $\zl=c\cdot1$, Proposition \ref{horunit} then gives
\begin{align*}
    0&=-\co(y_{2}*y_{1})\\&+(-1)^{|y_{2}||y_{1}|}\co(y_{1})\circ\co(y_{2}),   
\end{align*}
that is,
\begin{align*}
    \co(y_{1}*y_{2})&=(-1)^{|y_{1}||y_{2}|}\co(y_{2}*y_{1})\\&=\co(y_{1})\circ\co(y_{2}).
\end{align*}
\end{proof}
\begin{prop}\label{COcom}
We have $\co(y)\circ a=(-1)^{|y||a|}a\circ\co(y)$.
\end{prop}
\begin{proof}
By Proposition \ref{struc} with $k=l=1$, we have
\begin{align*}
    0
    &=-\mathfrak{q}^{\gamma,b}_{1,1}(\alpha;d\eta)\\
    &+\dl(\mathfrak{q}^{\gamma,b}_{1,1}(\alpha;\eta))\\
    &+(-1)^{|\eta|}\mathfrak{q}^{\gamma,b}_{1,1}(\dl(\alpha);\eta)\\
    &+(-1)^{|\eta||\alpha|+|\eta|+|\alpha|+1}\lml(\alpha,\xtl(\eta))\\
    &+\lml(\xtl(\eta),\alpha)\\
    &+(-1)^{|\eta|}\mathfrak{q}^{\gamma,b}_{2,1}(\zl,\alpha;\eta)\\
    &+(-1)^{|\eta|+|\alpha|+1}\mathfrak{q}^{\gamma,b}_{2,1}(\alpha,\zl;\eta).
\end{align*}
Since $\zl=c\cdot1$, Proposition \ref{unit} then gives
\begin{align*}
    0
    &+(-1)^{|y||a|+|y|}a\circ\co(y)\\
    &+(-1)^{|y|+1}\co(y)\circ a,
\end{align*}
that is,
\begin{align*}
    \co(y)\circ a=(-1)^{|y||a|}a\circ\co(y).
\end{align*}
\end{proof}
\begin{prop}
The map
\begin{align*}
    \qhx\otimes\qhl&\rightarrow\qhl\\
    (y,a)&\mapsto\co(y)\circ a
\end{align*}
makes $\qhl$ a unital graded algebra over $\qhx$.
\end{prop}
\begin{proof}
This follows from propositions \ref{funclass}, \ref{COrng}, \ref{QHLasc}, and \ref{COcom}.
\end{proof}

\subsubsection{The open-closed map}

Define
\begin{align*}
    \mathfrak{p}^{\beta,\gamma,b}_{1,0}:C\rightarrow\overline{E}
\end{align*}
by
\begin{align*}
    \mathfrak{p}^{\beta,\gamma,b}_{1,0}(\alpha)(\eta)=\langle\mathfrak{q}^{\beta,\gamma,b}_{0,1}(\eta),\alpha\rangle_{L}
\end{align*}
and set
\begin{align*}
    \mathfrak{p}^{\gamma,b}_{1,0}=\sum_{\beta\in\Pi}T^{\beta}\mathfrak{p}^{\beta,\gamma,b}_{1,0},
\end{align*}
so that
\begin{align*}
    \mathfrak{p}^{\gamma,b}_{1,0}(\alpha)(\eta)=\langle\mathfrak{q}^{\gamma,b}_{0,1}(\eta),\alpha\rangle_{L}.
\end{align*}
\begin{rmk}
The use of currents in the definition of $\mathfrak{p}^{\beta,\gamma,b}_{1,0}$ is required to push differential forms forward along maps which are not necessarily submersions. To demonstrate, note that
\begin{align*}
    \mathfrak{p}^{\beta,0,0}_{1,0}(\alpha)(\eta)&=-pt_{*}\big((\evbz)_{*}(evi^{\beta}_{1})^{*}\eta\wedge\alpha\big)\\
    &=(-1)^{|\alpha|\cdot\rdim\evbz+1}pt_{*}(\evbz)_{*}\big((evi^{\beta}_{1})^{*}\eta\wedge(\evbz)^{*}\alpha\big)\\
    &=(-1)^{|\alpha|\cdot n+1}pt_{*}\big((evi^{\beta}_{1})^{*}\eta\wedge(\evbz)^{*}\alpha\big)\\
    &=(-1)^{|\alpha|\cdot n+1}\hat{\varphi}((\evbz)^{*}\alpha)((evi^{\beta}_{1})^{*}\eta)\\
    &=(-1)^{|\alpha|\cdot n+1}\big((evi^{\beta}_{1})_{*}\hat{\varphi}((\evbz)^{*}\alpha)\big)(\eta),
\end{align*}
where the map $evi^{\beta}_{1}$ is not necessarily a submersion.
\end{rmk}
\begin{prop}
$\ltcx$ is of degree $n$.
\end{prop}
\begin{proof}
Take $\alpha\in A^{*}(L)$ and $\eta\in A^{*}(X)$. If $\mathfrak{p}^{\beta,\gamma,b}_{1,0}(\alpha)(\eta)\neq0$ we obtain
\begin{align*}
    n&=|\alpha|+|\mathfrak{q}^{\beta,\gamma,b}_{0,1}(\eta)|\\&=|\alpha|+|\eta|-\mu(\beta),
\end{align*}
so that
\begin{align*}
    |\mathfrak{p}^{\beta,\gamma,b}_{1,0}(\alpha)|&=2n-|\eta|\\&=2n-(n-|\alpha|+\mu(\beta))\\&=|\alpha|+n-\mu(\beta).
\end{align*}
\end{proof}
\begin{lem}\label{occhainlem}
We have
\begin{align*}
    \langle\alpha_{1},\dl(\alpha_{2})\rangle_{L}=(-1)^{|\alpha_{1}|+1}\langle\dl(\alpha_{1}),\alpha_{2}\rangle_{L}.\\
\end{align*}
\end{lem}
\begin{proof}
By Proposition \ref{struc} with $k=2$, $l=0$, we have
\begin{align*}
    0=\dl(\lml(\alpha_{1},\alpha_{2}))+\lml(\dl(\alpha_{1}),\alpha_{2})+(-1)^{|\alpha_{1}|+1}\lml(\alpha_{1},\dl(\alpha_{2})).
\end{align*}
Proposition \ref{defprop} then gives
\begin{align*}
    \langle\alpha_{1},\dl(\alpha_{2})\rangle_{L}
    &=(-1)^{|\alpha_{1}|}\int_{L}\lml(\alpha_{1},\dl(\alpha_{2}))\\
    &=\int_{L}\lml(\dl(\alpha_{1}),\alpha_{2})+\int_{L}\dl(\lml(\alpha_{1},\alpha_{2}))\\
    &=(-1)^{|\alpha_{1}|+1}\langle\dl(\alpha_{1}),\alpha_{2}\rangle_{L}.
\end{align*}
\end{proof}
\begin{prop}
The map
\begin{align*}
    -\ltcx:C\rightarrow\overline{E}
\end{align*}
is a chain map, and thus descends to a map 
\begin{align*}
    \occ:\qhl\rightarrow\qhcx
\end{align*}
given by
\begin{align*}
    \occ(a)(y)=\langle\co(y),a\rangle_{L}.
\end{align*}
\end{prop}
\begin{proof}
By Lemma \ref{occhainlem} together with propositions \ref{COdeg} and \ref{cochain}, we have
\begin{align*}
    \ltcx(\dl(\alpha))(\eta)&=\langle\xtl(\eta),\dl(\alpha)\rangle_{L}\\
    &=(-1)^{|\eta|+1}\langle\dl(\xtl(\eta)),\alpha\rangle_{L}\\
    &=(-1)^{|\eta|+1}\langle\xtl(d\eta),\alpha\rangle_{L}\\
    &=(-1)^{|\eta|+1}\ltcx(\alpha)(d\eta)\\
    &=d\ltcx(\alpha)(\eta)
\end{align*}
\end{proof}
Recall that by Proposition \ref{qcuriso} we have a module isomorphism
\begin{align*}
    \hat{\varphi}:\qhx\rightarrow\qhcx.
\end{align*}
\begin{defn}\label{ocdef}
Define
\begin{align*}
    \oc:\qhl\rightarrow\qhx
\end{align*}
by
\begin{align*}
    \oc=\hat{\varphi}^{-1}\circ\occ.
\end{align*}
\end{defn}
\begin{rmk}\label{OCchar}
Note that $\oc$ is characterized by the relation
\begin{align*}
    \langle y,\oc(a)\rangle_{X}=\langle\co(y),a\rangle_{L}
\end{align*}
\end{rmk}
\begin{prop}\label{OCmod}
The map $\oc$ is a module homomorphism.
\end{prop}
\begin{proof}
By remarks \ref{OCchar} and \ref{QHLint} together with propositions \ref{alg}, \ref{COmod}, and \ref{pairing}, we have
\begin{align*}
    \langle y_{2},\oc(y_{1}\circledast a)\rangle_{X}
    &=\langle\co(y_{2}),y_{1}\circledast a\rangle_{L}\\
    &=\int_{L}\co(y_{2})\circ(y_{1}\circledast a)\\
    &=(-1)^{|y_{1}||y_{2}|}\int_{L}(y_{1}\circledast\co(y_{2}))\circ a\\
    &=(-1)^{|y_{1}||y_{2}|}\langle y_{1}\circledast \co(y_{2}),a\rangle_{L}\\
    &=(-1)^{|y_{1}||y_{2}|}\langle\co(y_{1}*y_{2}),a\rangle_{L}\\
    &=\langle\co(y_{2}*y_{1}),a\rangle_{L}\\
    &=\langle y_{2}*y_{1},\oc(a)\rangle_{X}\\
    &=\langle y_{2},y_{1}*\oc(a)\rangle_{X}.
\end{align*}
\end{proof}
\begin{prop}
We have $\oc(1)=\pdl$.
\end{prop}
\begin{proof}
Remark \ref{OCchar} and Proposition \ref{topint} yield
\begin{align*}
    \langle y,\oc(1)\rangle_{X}&=\langle\co(y),1\rangle_{L}\\&=\int_{L}y|_{L}.
\end{align*}
\end{proof}
\begin{prop}\label{OCCOprod}
We have $\oc(\co(y)\circ a)=y*\oc(a)$.
\end{prop}
\begin{proof}
Remark \ref{OCchar} together with propositions \ref{Lpairing}, \ref{COrng} and \ref{pairing} yield
\begin{align*}
    \langle y_{2},\oc(\co(y_{1})\circ a)\rangle_{X}&=\langle\co(y_{2}),\co(y_{1})\circ a\rangle_{L}\\
    &=\langle\co(y_{2})\circ\co(y_{1}),a\rangle_{L}\\
    &=\langle\co(y_{2}*y_{1}),a\rangle_{L}\\
    &=\langle y_{2}*y_{1},\oc(a)\rangle_{L}\\
    &=\langle y_{2},y_{1}*\oc(a)\rangle_{L}.
\end{align*}
\end{proof}
\begin{rmk}\label{occo}
Note that this implies
\begin{align*}
    \oc(\co(y))&=\oc(\co(y)\circ1)\\&=y*\oc(1)\\&=y*\pdl.
\end{align*}
\end{rmk}

\subsection{Comparison of algebra structures}\label{comparison}

This section is devoted to the proof of Theorem \ref{compare}.
\begin{prop}\label{compare2}
We have
\begin{align*}
    y\circledast a&=\co(y)\circ a.
\end{align*}
\end{prop}
\begin{proof}
Proposition \ref{rstruc} with $k=m=l=1$ yields
\begin{align*}
    0&=-\mathfrak{q}^{\gamma,b}_{1,1;-,1}(\alpha;d\eta)\\
    &+(-1)^{|\eta|}\mathfrak{q}^{\gamma,b}_{2,2;-,2}(\zl,\alpha;\eta)\\
    &+(-1)^{|\eta|+|\alpha|+1}\mathfrak{q}^{\gamma,b}_{2,1;-,1}(\alpha,\zl;\eta)\\
    &+(-1)^{|\eta|}\mathfrak{q}^{\gamma,b}_{2,1;-,1}(\dl(\alpha);\eta)\\
    &+\mathfrak{q}^{\gamma,b}_{1,0}(\mathfrak{q}^{\gamma,b}_{1,1;-,1}(\alpha;\eta))\\
    &+\mathfrak{q}^{\gamma,b}_{2,0}(\mathfrak{q}^{\gamma,b}_{0,1}(\eta),\alpha)\\
    &+(-1)^{|\eta|+n}\mathfrak{q}^{\gamma,b}_{1,1;1}(\alpha;\eta).
\end{align*}
Since $\zl=c\cdot1$, applying Proposition \ref{sunit}, in $\qhl$ we get
\begin{align*}
    0=&(-1)^{|y|+1}\co(y)\circ a\\&+(-1)^{|y|}y\circledast a,
\end{align*}
that is,
\begin{align*}
    y\circledast a=\co(y)\circ a.
\end{align*}
\end{proof}
\begin{rmk}
Proposition \ref{compare2} (together with Proposition \ref{QHLasc}) gives another proof of Proposition \ref{alg}. It also gives an equivalence between propositions \ref{mod}, \ref{COmod}, and \ref{COrng}.
\end{rmk}

\section{An example}\label{exm}
Let $X=\Cp^{2}$, and let $L\subset X$ be a monotone Lagrangian satisfying our assumptions with minimal Maslov number $\mu_{\min}=2$. Choose $\Pi=H_{2}(X,L;\mathbb{Z})/\ker\mu\cong2\mathbb{Z}$, and let $\beta\in\Pi$ be the generator with $\mu(\beta)=2$. Choose also $(\gamma,b)=(0,0)$ as in Proposition \ref{BPexist} and $R=\Lambda$. Let $\hat{\beta}=[\Cp^{1}]\in H_{2}(X;\mathbb{Z})$, and set $y=PD(\hat{\beta})\in H^{2}(X)$. Recall that $\qhx$ is generated as a ring by $y$, hence the module structure is determined by the action of $y$. 

\begin{prop}\label{exclif}
Assume that $\hat{\beta}\cdot[L]=0$, where $\cdot$ stands for the intersection number. Then for all $a=[\alpha]\in\qhl$ we have
\begin{align*}
    y\circledast a=[T^{\beta}]\circ a=[T^{\beta}\alpha].
\end{align*}
\end{prop}
\begin{proof}
By Proposition \ref{compare2}, it suffices to prove that
\begin{align*}
    \co(y)=[T^{\beta}].
\end{align*}
Note that $y=PD(\hat{\beta})\in H^{2}(X)$, therefore
\begin{align*}
    \int_{L}y|_{L}=\int_{X}y\wedge PD([L])=\hat{\beta}\cdot[L]=0,
\end{align*}
so that $y|_{L}=0\in H^{2}(L)$. Writing $y=[\eta]$ for $\eta\in A^{2}(X)$, we get that $\eta|_{L}$ is exact, hence we can write $\eta|_{L}=-d\sigma$ for some $\sigma\in A^{1}(L)$.
\par
By propositions \ref{degree} and \ref{ezero}, we can write
\begin{align*}
    \mathfrak{q}_{1,0}(\sigma)&=-\eta|_{L}+fT^{\beta},\\
    \mathfrak{q}_{0,1}(\eta)&=-\eta|_{L}+gT^{\beta},
\end{align*}
where $f,g\in A^{0}(L)$. We obtain
\begin{align*}
    \co(y)
    =[-\mathfrak{q}_{0,1}(\eta)]
    =[-\mathfrak{q}_{0,1}(\eta)+\mathfrak{q}_{1,0}(\sigma)]
    =[\lambda T^{\beta}],
\end{align*}
where $\lambda=f-g\in A^{0}(L)$. Since $T^{\beta}\lambda$ is closed, by propositions \ref{degree} and \ref{ezero}, we have
\begin{align*}
    0
    =\mathfrak{q}_{1,0}(\lambda T^{\beta})
    =d\lambda T^{\beta},
\end{align*}
thus $\lambda\in\mathbb{R}$.
\par
Recall that the standard computation of $QH^{*}(X)$ shows that $y*y*y=[T^{\varpi(\hat{\beta})}]\in QH^{6}(X)$. Since $\mu(\varpi(\hat{\beta}))=2c_{1}(\hat{\beta})=6=3\mu(\beta)$, it follows from the choice of $\Pi$ that $\varpi(\hat{\beta})=3\beta$. Therefore, by Proposition \ref{COrng}, we have
\begin{align*}
    [T^{3\beta}]
    =\co([T^{3\beta}])
    =\co(y*y*y)
    =\co(y)\circ\co(y)\circ\co(y)
    =[\lambda^{3}T^{3\beta}].
\end{align*}
Therefore
\begin{align*}
    0=[(\lambda^{3}-1)T^{3\beta}]=(\lambda^{3}-1)[T^{3\beta}].
\end{align*}
Thus either $\lambda=1$ as required or $[T^{3\beta}]=0$. In the latter case, note that we may write $T^{3\beta}=\mathfrak{q}_{1,0}(T^{2\beta}\xi)$ for some $\xi\in A^{1}(L)$. Equivalently, $T^{\beta}=\mathfrak{q}_{1,0}(\xi)$, and in particular
\begin{align*}
    \co(y)=\lambda[T^{\beta}]=0=[T^{\beta}].
\end{align*}
\end{proof}
\begin{exmp}
The Clifford torus $L=\mathbb{T}^{2}_{\textup{clif}}$ satisfies the assumptions of Proposition \ref{exclif}. This can be compared to the computation of the Biran-Cornea module structure for this case, which is also given by multiplication by a Novikov coefficient (see \cite{LQH}*{Theorem 2.3.2}).
\end{exmp}

\addcontentsline{toc}{section}{References}
\bibliography{bib}

\begin{bibdiv}
\begin{biblist}

\bib{LQH}{article}{
      author={Biran, Paul},
      author={Cornea, Octav},
       title={{Rigidity and uniruling for Lagrangian submanifolds}},
        date={2009},
     journal={Geom. Topol.},
      volume={13},
      number={5},
       pages={2881\ndash 2989},
         url={10.2140/gt.2009.13.2881},
}

\bib{Maslov}{article}{
      author={Cieliebak, Kai},
      author={Goldstein, Edward},
       title={{A note on mean curvature, Maslov class and symplectic area of
  Lagrangian immersions}},
        date={2004},
     journal={J. Symplectic Geom.},
      volume={2},
      number={2},
       pages={261\ndash 266},
         url={10.4310/JSG.2004.v2.n2.a4},
}

\bib{dR}{book}{
      author={de~Rham, Georges},
       title={{Differentiable manifolds}},
      series={Grundlehren der mathematischen Wissenschaften},
        date={1973},
      volume={266},
}

\bib{cyc}{article}{
      author={Fukaya, Kenji},
       title={{Cyclic symmetry and adic convergence in Lagrangian Floer
  theory}},
        date={2010},
     journal={Kyoto J. Math.},
      volume={50},
      number={3},
       pages={521\ndash 590},
         url={10.1215/0023608X-2010-004},
}

\bib{LIFT}{book}{
      author={Fukaya, Kenji},
      author={Oh, Yong-Geun},
      author={Ohta, Hiroshi},
      author={Ono, Kaoru},
       title={{Lagrangian intersection Floer theory: anomaly and obstruction.
  Parts I,II}},
      series={AMS/IP Studies in Advanced Mathematics},
        date={2009},
      volume={46},
}

\bib{tor}{article}{
      author={Fukaya, Kenji},
      author={Oh, Yong-Geun},
      author={Ohta, Hiroshi},
      author={Ono, Kaoru},
       title={{Lagrangian Floer theory on compact toric manifolds I}},
        date={2010},
     journal={Duke Math. J.},
      volume={151},
      number={1},
       pages={23\ndash 175},
         url={10.1215/00127094-2009-062},
}

\bib{torbd}{article}{
      author={Fukaya, Kenji},
      author={Oh, Yong-Geun},
      author={Ohta, Hiroshi},
      author={Ono, Kaoru},
       title={{Lagrangian Floer theory on compact toric manifolds II: bulk
  deformations}},
        date={2011},
     journal={Sel. Math. New Ser.},
      volume={17},
      number={3},
       pages={609\ndash 711},
         url={10.1007/s00029-011-0057-z},
}

\bib{sympH}{article}{
      author={Ganatra, Sheel},
       title={{Symplectic cohomology and duality for the wrapped Fukaya
  category}},
        date={2013},
     journal={arXiv},
      eprint={1304.7312},
}

\bib{OC}{article}{
      author={Hugtenburg, Kai},
       title={{The cyclic open-closed map, u-connections and R-matrices}},
        date={2024},
     journal={Sel. Math. New Ser.},
      volume={30},
      number={2},
         url={10.1007/s00029-024-00925-7},
}

\bib{corner}{article}{
      author={Joyce, Dominic},
       title={{On manifolds with corners}},
        date={2012},
     journal={Adv. Lect. Math.},
      volume={21},
       pages={225\ndash 258},
}

\bib{Jhol}{book}{
      author={McDuff, Dusa},
      author={Salamon, Dietmar},
       title={{$J$-holomorphic curves and symplectic topology}},
     edition={2},
      series={American Mathematical Society Colloquium Publications},
        date={2012},
      volume={52},
}

\bib{prep}{article}{
      author={Salamon, Jake~P.},
       title={{Structure equations for geodesic operators}},
     journal={preprint},
}

\bib{orb}{article}{
      author={Salamon, Jake~P.},
      author={Tukachinsky, Sara~B.},
       title={{Differential forms on orbifolds with corners}},
     journal={to appear in J. Topol. Anal.},
         url={10.1142/S1793525323500048},
}

\bib{RQH}{article}{
      author={Salamon, Jake~P.},
      author={Tukachinsky, Sara~B.},
       title={{Relative quantum cohomology}},
     journal={to appear in J. Eur. Math. Soc.},
         url={10.4171/JEMS/1337},
}

\bib{BP}{article}{
      author={Salamon, Jake~P.},
      author={Tukachinsky, Sara~B.},
       title={{Point-like bounding chains in open Gromov-Witten theory}},
        date={2021},
     journal={Geom. Funct. Anal.},
      volume={31},
      number={5},
       pages={1245\ndash 1320},
         url={10.1007/s00039-021-00583-3},
}

\bib{Ainf}{article}{
      author={Salamon, Jake~P.},
      author={Tukachinsky, Sara~B.},
       title={{Differential forms, Fukaya $A_{\infty}$ algebras, and
  Gromov-Witten axioms}},
        date={2022},
     journal={J. Symplectic Geom.},
      volume={20},
      number={4},
       pages={927\ndash 994},
         url={10.4310/JSG.2022.v20.n4.a5},
}

\bib{pseudocyc}{article}{
      author={Sheridan, Nick},
       title={{On the Fukaya category of a Fano hypersurface in projective
  space}},
        date={2016},
     journal={Publ. Math. IHÉS},
      volume={124},
      number={1},
       pages={165\ndash 317},
         url={10.1007/s10240-016-0082-8},
}

\bib{complex}{book}{
      author={Zill, Dennis~G.},
      author={Shanahan, Patrick~D.},
       title={{A First Course in Complex Analysis with Applications}},
     edition={2},
        date={2003},
}

\end{biblist}
\end{bibdiv}

\end{document}